\documentclass[10pt]{article}
\usepackage[T1]{fontenc}
\usepackage[utf8]{inputenc}
\usepackage[english]{babel}
\usepackage[autostyle, english=american]{csquotes}
\MakeOuterQuote{"}

\usepackage{xcolor}
\definecolor{reference}{rgb}{0.20,0.36,0.74}
\definecolor{citation}{rgb}{0,.40,.80}
\usepackage{url}
\usepackage{breakurl}
\usepackage[breaklinks,colorlinks,linktocpage,urlcolor=blue,linkcolor=reference,citecolor=citation]{hyperref}

\usepackage{scalerel}
\usepackage{tikz-cd}
\usepackage{verbatim}
\usepackage{mathrsfs}
\usepackage{mathtools}
\usepackage{latexsym}
\usepackage{graphicx}
\usepackage{amscd,amssymb,amsmath,amsbsy,amsfonts,amsthm}
\usepackage[mathscr]{eucal}
\usepackage[all, 2cell]{xy}
\UseTwocells
\xyoption{2cell}{\UseTwocells}

\usepackage{enumerate}
\usepackage{tikz}
\usepackage[left=2.7cm, top=2.7cm, right=2.7cm, bottom=2.7cm]{geometry}
\usepackage{enumitem} 

\usepackage[backend=biber,style=alphabetic,sorting=nty,firstinits=true,maxbibnames=99,maxalphanames=99]{biblatex}
\usepackage[capitalise]{cleveref}
\crefformat{equation}{(#2#1#3)}

\DeclareFontFamily{OT1}{pzc}{}
\DeclareFontShape{OT1}{pzc}{m}{it}{<-> s * [1.200] pzcmi7t}{}
\DeclareMathAlphabet{\mathpzc}{OT1}{pzc}{m}{it}

\usepackage{relsize}
\usepackage[bbgreekl]{mathbbol}
\DeclareSymbolFontAlphabet{\mathbb}{AMSb} 
\DeclareSymbolFontAlphabet{\mathbbl}{bbold}

\newcommand{\Prism}{{\mathlarger{\mathbbl{\Delta}}}}

\DeclareTextFontCommand{\emdef}{\it}

\newtheorem{thm}{Theorem}
\newtheorem*{thm*}{Theorem}
\newtheorem{lem}[thm]{Lemma}
\newtheorem{cor}[thm]{Corollary}
\newtheorem*{cor*}{Corollary}
\newtheorem{prop}[thm]{Proposition}
\newtheorem*{prop*}{Proposition}

\theoremstyle{definition}
\newtheorem{defn}[thm]{Definition}
\newtheorem{construction}[thm]{Construction}
\newtheorem{notation}[thm]{Notation}

\newtheorem{ex}[thm]{Example}

\newtheorem{rem}[thm]{Remark}
\newtheorem{question}[thm]{Question}

\newtheorem*{conj*}{Conjecture}

\newtheorem{warn}[thm]{Warning}

\definecolor{note_color}{rgb}{0.0,0.7,0.0}

\newcommand{\fcat}{\mathscr}
\newcommand{\inj}{\hookrightarrow}
\newcommand{\surj}{\twoheadrightarrow}
\newcommand{\areq}{\mathbin{{\xrightarrow{\,\sim\,}}}}

\renewcommand{\phi}{\varphi}
\renewcommand{\epsilon}{\varepsilon}

\DeclareMathOperator{\rank}{rank}

\newcommand{\Mod}{\mathrm{Mod}}

\DeclareMathOperator{\SL}{SL}
\DeclareMathOperator{\gr}{gr}
\DeclareMathOperator{\rk}{rk}

\DeclareMathOperator{\Frac}{Frac}

\renewcommand{\Im}{\mathrm{Im}}

\newcommand{\Top}{ {\mathrm{Top}} }

\DeclareMathOperator{\Spec}{Spec}
\DeclareMathOperator{\Spf}{Spf}

\DeclareMathOperator{\Coh}{Coh}

\DeclareMathOperator{\QCoh}{QCoh}

\newcommand{\mstack}{\mathpzc}
\DeclareMathOperator{\PStk}{{\mathscr{PS}tk}}
\DeclareMathOperator{\Stk}{{\mathscr{S}tk}}
\newcommand{\Hdg}{{\mathrm H}}
\newcommand{\dR}{{\mathrm{dR}}}
\newcommand{\crys}{{\mathrm{crys}}}

\DeclareMathOperator{\Aff}{Aff}
\newcommand{\sm}{{\mathrm{sm}}}

\DeclareMathOperator{\Spa}{Spa}

\newcommand{\ket}{{\mathrm{k\acute et}}}
\newcommand{\et}{{\mathrm{\acute et}}}

\newcommand{\proket}{{\mathrm{prok\acute et}}}

\newcommand{\Set}{ {\fcat{S}\mathrm{et}} }

\DeclareMathOperator{\Ab}{Ab}

\DeclareMathOperator{\Hom}{Hom}
\DeclareMathOperator{\Ext}{Ext}
\DeclareMathOperator{\Tor}{Tor}

\DeclareMathOperator{\End}{End}

\DeclareMathOperator{\Comp}{Comp}

\DeclareMathOperator*{\colim}{colim}
\DeclareMathOperator*{\fib}{fib}

\DeclareMathOperator*{\cofib}{cofib}
\DeclareMathOperator*{\coker}{coker}

\DeclareMathOperator{\Ran}{Ran}

\newcommand{\prolim}[1][]{\lim\limits_{\xleftarrow[#1]{}}}
\newcommand{\indlim}[1][]{\lim\limits_{\xrightarrow[#1]{}}}

\DeclareMathOperator{\Tot}{Tot}

\DeclareMathOperator{\Id}{Id}

\newcommand{\Shv}{{\mathcal S\mathrm{hv}}}
\newcommand{\op}{{\mathrm{op}}}

\DeclareMathOperator{\Kos}{Kos}

\mathchardef\mdef="2D

\usepackage{enumerate}
\usepackage{tocloft}
\usepackage{textcomp}
\DeclareFontFamily{OT1}{pzc}{}
\DeclareFontShape{OT1}{pzc}{m}{it}{<-> s * [1.200] pzcmi7t}{}
\DeclareMathAlphabet{\mathpzc}{OT1}{pzc}{m}{it}

\newcommand{\cotimes}{\mathbin{\widehat\otimes}}

\newcommand{\ol}{\overline}

\newcommand{\Ainf}{{A_{\mathrm{inf}}}}

\newcommand{\mf}{\mathfrak}
\newcommand{\mc}{\mathcal}
\newcommand{\mbb}{\mathbb}
\newcommand{\mr}{\mathrm}

\newcommand{\ul}{\underline}
\newcommand{\xra}{\xrightarrow}

\newcommand\proet{\text{pro\'et}}

\newcommand{\RG}{R\Gamma}

\newcommand{\ra}{\rightarrow}
\newcommand{\nc}{{\mathrm{nc}}}

\newcommand{\fg}{{\mathrm{fg}}}
\newcommand{\free}{{\mathrm{free}}}

\newcommand{\id}{{\mathrm{id}}}
\newcommand{\fr}{^{(1)}}

\newcommand{\ism}{{\xymatrix{\ar[r]^\sim &}}}
\newcommand{\arr}{{\xymatrix{\ar[r]&}}}

\newcommand{\comp}{\mr{comp}}
\newcommand{\DMod}[1]{D(\Mod_{#1})}
\newcommand{\UMod}[1]{\Mod_{#1}}

\newcommand{\BK}{{\mathrm{BK}}}

\newcommand{\an}{{\mathrm{an}}}

\usepackage{tikz}
\usetikzlibrary{matrix}

\newcommand{\sing}{{\mathrm{sing}}}

\newcommand{\bt}{\blacktriangle}
\newcommand{\ccond}{\blacktriangle}
\newcommand{\tto}{{\xymatrix{\ar[r]&}}}
\newcommand{\Solid}{{\mathrm{Solid}}}
\def\sq{\mathbin{\scalerel*{\strut\rule[-.5ex]{10ex}{10ex}}{\cdot}}}
\newcommand{\solid}{{\sq}}
\newcommand{\sotimes}{\otimes^\solid}
\newcommand{\Cond}{{\mathrm{Cond}}}

\newcommand{\mmax}{{\mathrm{max}}}
\newcommand{\innf}{{\mathrm{inf}}}

\let\emptyset\varnothing

\numberwithin{thm}{section}
\numberwithin{equation}{section}

\title{Rational $p$-adic Hodge theory for $d$-de Rham-proper stacks}

\author{Dmitry Kubrak and Artem Prikhodko \smallskip\\
	\MakeLowercase{with an appendix by} Haoyang Guo}

\date{}

\begin{document}
\maketitle

\begin{abstract}
	In this follow-up paper we show that smooth Hodge-proper stacks over $\mc O_K$ are $\mbb Q_p$-locally acyclic: namely the natural map between \'etale $\mbb Q_p$-cohomology of the algebraic and Raynaud generic fibers is an equivalence. This establishes the $\mbb Q_p$-case of general conjectures made in \cite{KubrakPrikhodko_pHodge}. As a corollary, we get that if a smooth Artin stack over $K$ has a smooth Hodge-proper model over $\mc O_K$, its $\mbb Q_p$-\'etale cohomology is a crystalline Galois representation.
	 We then also establish a truncated version of the above results in more general setting of smooth $d$-de Rham-proper stacks over $\mc O_K$: here we only require first $d$ de Rham cohomology groups be finitely-generated over $\mc O_K$. As an application, we deduce a certain purity-type statement for \'etale $\mbb Q_p$-cohomology of Raynaud generic fiber, as well as crystallinity of a first several \'etale cohomology groups in the presence of a Cohen--Macauley model over $\mc O_K$ in the schematic setting.
\end{abstract}

\tableofcontents

\section{Introduction}
\subsection{$p$-adic Hodge theory for Hodge-proper stacks}
Let $K$ be a complete discretely valued extension of $\mathbb Q_p$ with the ring of integers $\mathcal O_K$ and perfect residue field $k$. In \cite{KubrakPrikhodko_pHodge} we developed some aspects of $p$-adic Hodge theory in the setting of Artin stacks over $\mathcal O_K$. The following concept played a key role:

\begin{defn}[{\cite{KubrakPrikhodko_HdR}}]\label{intro_defn_of_hodge_properness}
	A smooth quasi-compact quasi-separated Artin stack $\mstack X$ over a Noetherian ring $R$ is called \emdef{Hodge-proper} if for any $i,j\in \mathbb Z_{\ge 0}$ the cohomology $H^j(\mstack X, \wedge^i \mathbb L_{\mstack X/R})\in \UMod{R}$ is a finitely generated $R$-module.
\end{defn}

Our observation, that we investigated in previous works \cite{KubrakPrikhodko_HdR}, \cite{KubrakPrikhodko_pHodge}, was that Hodge-proper stacks\footnote{Or slight variations like Hodge-properly spreadable stacks (see \cite[Definition 1.4.1]{KubrakPrikhodko_HdR}).} are often as good as smooth proper schemes when considering certain cohomological features (like Hodge-to-de Rham degeneration in char 0 \cite{KubrakPrikhodko_HdR} or some aspects $p$-adic Hodge theory \cite{KubrakPrikhodko_pHodge}). This work establishes more results supporting this idea in the context of \emph{rational} $p$-adic Hodge theory (see Remark \ref{rem:Hodge-proper as proper} below). 

Let us start by reminding some examples of smooth Hodge-proper stacks:
\begin{ex}\label{ex:Hodge-proper stack}\begin{enumerate}
		\item 
	A smooth proper stack $\mstack X$ over $R$ is Hodge-proper (\cite[Corollary 2.2.13]{KubrakPrikhodko_HdR}).
\item
	Let $X$ be a smooth proper scheme over $R$ with an action of a reductive group $G$. Then the quotient stack $[X/G]$ is Hodge-proper (\cite[Section 1.3]{KubrakPrikhodko_pHodge}).
	
\item Let $X$ be a smooth scheme with an action $G$ and assume that $X$ has a $G$-invariant affine cover such that all intersections are affine. In this situation there is a well-defined "categorical quotient" scheme $X/\!\!/G$. Then (see \cite[Proposition 1.4.6]{KubrakPrikhodko_pHodge}) if $X/\!\!/G$ is proper, the stack $[X/G]$ is Hodge-proper. This includes many examples when $X$ is not necessary proper (see \cite[Example 3.1.6]{KubrakPrikhodko_HdR}), e.g. $[\mbb A^n/\SL_n]$ with the tautological action.

\item Some smooth $\Theta$-stratified stacks (for the definition see \cite{HL-instability} and \cite{Halpern-Leistner_Theta}) are also Hodge-proper. Namely (see \cite[Example 1.4.8]{KubrakPrikhodko_pHodge}), if a smooth Artin stack $\mstack X$ with affine diagonal is endowed with a finite $\Theta$-stratification and all $\Theta$-strata are cohomologically proper (see \cite[Definition 2.2.2]{KubrakPrikhodko_HdR}), then $\mstack X$ itself is Hodge-proper. The unstable strata in this criterion can also be replaced by their centra. The particular examples then include even more general quotient stacks $[X/G]$ where the action is only Kempf-Ness complete.
		\end{enumerate}
\end{ex}

A particularly important aspect studied in \cite{KubrakPrikhodko_pHodge} was the potential discrepancy between the \'etale cohomology of the algebraic and Raynaud generic fibers of a Hodge-proper stack which one doesn't see in the usual smooth proper setting. Let us recall this in some detail. Namely, let $K$ be a finite extension of $\mbb Q_p$ and let ${C}=\widehat{\ol K}$ be the completion of the algebraic closure of $K$. Let $\mc O_K\subset K$ and $\mc O_{C}\subset {C}$ be the corresponding rings of integers. In \cite{KubrakPrikhodko_pHodge} to an Artin stack $\mstack X$ over $\mc O_{C}$ we associated three other stacks: the algebraic generic fiber $\mstack X_{C}$, its analytification $\mstack X_{C}^\an$ and the Raynaud generic fiber $\widehat{\mstack X}_{C}$ (the latter two are rigid-analytic stacks, see \cite[Section 3.3]{KubrakPrikhodko_pHodge}). One then has natural maps (see \cite[Section 3.3, Remark 4.1.16]{KubrakPrikhodko_pHodge})
	$$
\xymatrix{
	\RG_{\et}(\mstack X_{C},\mbb Z_p) \ar@/_2.0pc/[rr]_{\Upsilon_{\mstack X}}\ar[r]_\sim^{\phi_{\mstack X}^{-1}}& \RG_\et(\mstack X^\an_{C}, \mbb Z_p) \ar[r]^{\psi_{\mstack X}^{-1}}&\RG_\et(\widehat{\mstack X}_{C}, \mbb Z_p)},
$$
between the $\mbb Z_p$-\'etale cohomology. Here the map  $\phi_{\mstack X}^{-1}$ is always an equivalence, while $\psi_{\mstack X}^{-1}$ is a priori an equivalence only in the smooth proper case. Still, in \cite{KubrakPrikhodko_pHodge} we have shown that the map $\Upsilon_{\mstack X}$ is an equivalence when $\mstack X$ is as in part 2 of \Cref{ex:Hodge-proper stack}, but the situation in the general Hodge-proper case remained unclear. 

In this work we study the map $\Upsilon_{\mstack X}$ rationally, namely when we replace $\mbb Z_p$-coefficients with $\mbb Q_p$ ones. We show that then the following general result holds:
\begin{thm*}[\ref{thm:main theorem}]
Let $\mstack X$ be a smooth Hodge-proper stack over $\mc O_K$. Then the map 
	$$
	\Upsilon_{\mstack X,\mbb Q_p}\colon R\Gamma_\et(\mstack X_{C}, \mathbb Q_p) \tto R\Gamma_\et(\widehat{\mstack X}_{C}, \mathbb Q_p)
	$$ 
	is an equivalence.
\end{thm*}

\begin{rem}\label{rem:Q_p-acyclic stacks} This was stated as Conjecture 4.3.16 in \cite{KubrakPrikhodko_pHodge}. By analogy with \cite[Definition 4.1.18]{KubrakPrikhodko_pHodge} stacks for which the map $\Upsilon_{\mstack X,\mbb Q_p}$ is an equivalence could be called \textit{$\mbb Q_p$-locally acyclic}.
\end{rem}
Let $G_K$ be the absolute Galois group of $K$. From \Cref{thm:main theorem}, using the results in \cite{KubrakPrikhodko_pHodge} one deduces that the Galois representation in \'etale cohomology of $\mstack X_{C}$ is crystalline:
\begin{cor*}[\ref{cor:Fontaine's conjecture for Hodge-proper stacks}]
	Let $\mstack X$ be a smooth Hodge-proper stack over $\mc O_K$. Then for any $i\ge 0$ the $G_K$-representation given by $H^i({\mstack X}_{C},\mbb Q_p)$ is crystalline.
\end{cor*}
As a consequence we also obtain all the comparisons that one usually has for a smooth proper scheme (see \Cref{rem:rational p-adic Hodge theory for Hodge proper stacks}).

Let us try to motivate \Cref{thm:main theorem} and \Cref{cor:Fontaine's conjecture for Hodge-proper stacks} by putting them into some broader context with remarks below.

\begin{rem}[Hodge-proper is (almost) as good as proper]\label{rem:Hodge-proper as proper} Let $Y$ be a smooth scheme over $K$. Then using the work of de Jong on alterations \cite{deJong_alterations} Kisin showed in \cite{kisin2002potential} that $H^n(Y_{C},\mbb Q_p)$ is a potentially semi-stable $G_K$-representation. However, by the Fontaine's $C_\crys$-conjecture if $Y$ has a smooth proper model over $\mc O_K$ this representation has a finer structure: namely it is crystalline.
	
	Let now $\mstack Y$ be a smooth quasi-compact quasi-separated Artin stack over $K$. Resolving $\mstack Y$ by a simplicial scheme and using smooth descent by a similar argument one can show that the \'etale cohomology $H^i({\mstack Y}_{C},\mbb Q_p)$ is also a potentially semistable representation of the Galois group $G_K$. However, \Cref{cor:Fontaine's conjecture for Hodge-proper stacks} tells us that  for the $G_K$-representation $H^i({\mstack Y}_{C},\mbb Q_p)$ to be crystalline it is in fact enough for $\mstack Y$ to have a smooth Hodge-proper model over $\mc O_K$. This gives some evidence that from the point of view of $\mbb Q_p$-\'etale cohomology smooth Hodge-proper stacks are as good as smooth proper schemes. Let us remark that there are some examples of smooth schemes that are Hodge-proper but not proper\footnote{Though we do not know such an example over $\mc O_K$ with $K/\mbb Q_p$ finite, we do know some explicit examples over $W(k)$ with $k$ being the perfection of ${\mbb F_p(t)}$ (e.g. adapting the construction in \cite[Section 2.3.3]{KubrakPrikhodko_HdR}).} (see \cite[Section 2.3.3]{KubrakPrikhodko_HdR}).
\end{rem} 
	
\begin{rem}($G$-equivariant $p$-adic Hodge theory).\label{rem:G-equivaraiant p-adix Hodge} In \cite{KubrakPrikhodko_pHodge} we proved a $G$-equivariant variant of  $p$-adic Hodge theory for a smooth proper $\mc O_K$-scheme $X$ with a reductive group action (see \cite[Section 0.4]{KubrakPrikhodko_pHodge}). It is natural to ask whether the properness condition is necessary. \Cref{thm:main theorem} and \Cref{cor:Fontaine's conjecture for Hodge-proper stacks} show that in fact a much milder condition is sufficient: $[X/G]$ should just be Hodge-proper (and smooth). This includes all the examples listed in \ref{ex:Hodge-proper stack}. 
	
	To be more precise (see \Cref{rem:rational p-adic Hodge theory for Hodge proper stacks} and apply to $\mstack X=[X/G]$) in this situation we have a $(\phi,G_K)$-equivariant isomorphism 
	$$
	H^n_\et({[X/G]}_{{C}}, \mathbb Q_p)\otimes_{\mbb Q_p} B_\crys \simeq H^n_{\crys}({[X/G]}_k/W(k))[\tfrac{1}{p}]\otimes_{K_0} B_\crys,
	$$
	a filtered $G_K$-equivariant isomorphism 
	$$
	H^n_\et({[X/G]}_{{C}}, \mathbb Q_p)\otimes_{\mbb Q_p} B_\dR \simeq H^n_\dR([X/G]/K)\otimes_K B_\dR
	$$
	and the Hodge-Tate decomposition 
	$$
	H^n_{\et}({[X/G]}_{{C}},\mbb Q_p)\otimes_{\mbb Q_p}{C} \simeq \bigoplus_{i+j=n} H^j({[X/G]}_{K},\wedge^i\mbb L_{\mstack [X/G]/K})\otimes_{K}{C}(-i).
	$$
Moreover, if $K$ is a finite extension of $\mathbb Q_p$, after a choice of isomorphism $\iota\colon {C}\xra{\sim }\mbb {\mbb C}$ by \cite[Example 4.1.5 and Proposition 4.1.6]{KubrakPrikhodko_pHodge} one can identify $H^n_{\et}({[X/G]}_{{C}},\mbb Q_p)$ with the $G(\mbb C)$-equivariant cohomology $H^n_{\sing, G(\mbb C)}(X(\mbb C),\mbb Q_p)$.
\end{rem}

\subsection{Generalization: $d$-de Rham proper stacks} In fact, in this work we show a more subtle statement in a more general setting. Here is our framework:

\begin{defn}[$d$-de Rham-proper stacks]\label{def:intro_d-de Rham proper}
	Fix a positive integer $d\in \mbb Z_{\ge 0}$. A smooth quasi-compact quasi-separated Artin stack $\mstack X$ over a Noetherian ring $R$ is called $d$-de Rham-proper if $H^i_\dR(\mstack X/R)$ is a finitely generated $R$-module for $i\le d$.
\end{defn}
 This is a $d$-truncated version of \Cref{intro_defn_of_hodge_properness}, where Hodge cohomology is replaced by de Rham. This notion is much more general (in particular, a Hodge-proper stack is $d$-de Rham for any $d$), and the main motivation to consider it comes from the fact that there are many natural examples of $d$-de Rham-proper schemes that are not necessarily proper (see \Cref{lem:example of d-Hodge-proper}).
 
 The possible expectation could be that a truncated variant of $p$-adic Hodge theory exists in this setting: namely, the standard properties of and comparisons between different cohomology should still be there, but maybe only up to a certain degree. This is exactly what we show.
 
 \begin{thm}\label{intro_thm_d-de Rham proper}
 	Let $\mc X$ be a smooth $(d+1)$-de Rham-proper stack over $\mc O_K$. Then \begin{enumerate}
 		\item The natural map $\Upsilon_{\mstack X,\mbb Q_p}\colon \RG_\et(\mstack X_{C}, \mathbb Q_p)\ra \RG_\et(\widehat{\mstack X}_{C}, \mathbb Q_p)$ induces isomorphisms 
 		$$
 		H^i_\et(\mstack X_{C}, \mathbb Q_p) \simeq H^i_\et(\widehat{\mstack X}_{C}, \mathbb Q_p)  
 		$$ for $i\le d$, and an embedding
 		$$      
 		H^{d+1}_\et(\mstack X_{C}, \mathbb Q_p) \hookrightarrow H^{d+1}_\et(\widehat{\mstack X}_{C}, \mathbb Q_p)   
 		$$
 		for $i=d+1$ (\Cref{thm:even mainer theorem}).
 		\item  The $G_K$-representation $
 		H^i_\et(\mstack X_{C}, \mathbb Q_p) \simeq H^i_\et(\widehat{\mstack X}_{C}, \mathbb Q_p)  
 		$ is crystalline for $i\le d$ with $D_\crys(H^i_\et(\mstack X_{C}, \mathbb Q_p))\simeq H^i_\crys(\mstack X_k/W(k))[\frac{1}{p}]$. Consequently, one has a $(G_K,\phi)$-equivariant isomorphism 
 		$$
 		H^i_\et(\mstack X_{C}, \mathbb Q_p)\otimes_{\mbb Q_p}B_\crys \ism H^i_\crys(\mstack X_k/W(k))\otimes_{W(k)} B_\crys
 		$$
 		in that range of degrees (\Cref{cor:Fontaine's conjecture for Hodge-proper stacks}).
 	\item  Provided $H^i_\crys(\mstack X/W(k))$ is $p$-torsion free and $i\le d$, the Breuil-Kisin module $H^i_\Prism(\mstack X/\mf S)$ corresponds to the $G_K$-equivariant crystalline lattice $H^i_\et(\widehat{\mstack X}_{C},\mbb Z_p)\subset H^i_\et(\mstack X_{C},\mbb Q_p)$ under the Breuil-Kisin functor:
 	$$
 	\BK(H^i_\et(\widehat{\mstack X}_{C},\mbb Z_p))\simeq H^i_\Prism(\mstack X/\mf S),
 	$$	
 	(\Cref{rem:lattice corresponding to prismatic cohomology}).
 	\end{enumerate}
 \end{thm}
We note right away that the proof of \Cref{intro_thm_d-de Rham proper} essentially follows the same ideas as the proof of analogous results in \cite{KubrakPrikhodko_pHodge}. However, to get the optimal degree range in the comparisons above one needs to analyze very carefully what happens with the completed tensor products in the boundary degrees (typically $d$ and $d-1$). Condensed mathematics of Clausen-Scholze turned out to be a very convenient (and, in fact, so far the only suitable for us) framework to do so.

\begin{question}
	Note that in part 3 of \Cref{intro_thm_d-de Rham proper} the crystalline lattice corresponding to prismatic cohomology is given by the cohomology of the \textit{Raynaud} generic fiber of $\mstack X$. This poses a natural question: do the lattices $H^i_\et(\widehat{\mstack X}_{C},\mbb Z_p)$ and $H^i_\et({\mstack X}_{C},\mbb Z_p)$ agree inside $H^i_\et({\mstack X}_{C},\mbb Q_p)$ (under the assumption that $\mstack X$ is $d$-de Rham proper and $i\le d$)? We do not expect this to be true, but also don't know a counterexample. 
\end{question}

We then apply \Cref{intro_thm_d-de Rham proper} to the schematic setting. A quite general example of a $d$-de Rham-proper scheme over $\mc O_K$ can be constructed as follows: take a proper Cohen-Macauley scheme $X$ over $\mc O_K$ and assume that the singularities $Z\hookrightarrow X$ have codimension $d+1$ in $X$; then the complement $U\coloneqq X\setminus Z$ is $(d-1)$-de Rham proper over $\mc O_K$. Most interesting is the case when $Z$ is in fact contained in the closed fiber $X_k$ (and is of codimension $d$): then $U_C\simeq X_C$, but the Raynaud generic fiber $\widehat U_C$ is given by the complement in $X_C$ to an "open tube" around $Z\subset X_k$. Thus application of \Cref{intro_thm_d-de Rham proper} in this situation leads to a purity-type result for erasing the latter, as well as the crystallinity of \'etale cohomology of $X_C$ in certain range.

\begin{thm*}[\ref{cor:applications for schemes}]\label{intro_applications for schemes}
	Let $X$ be a proper scheme over $\mc O_K$ that is Cohen-Macauley. Let $Z\hookrightarrow X_k$ be a codimension $d$ closed subscheme such that complement $X\backslash Z$ is smooth over $\mc O_K$. Then one has
	\begin{enumerate}
		\item  (Purity) There are natural isomorphisms $H^i_\et(\widehat{(X\backslash Z)}_C,\mbb Q_p)\simeq H^i_\et(\widehat X_C,\mathbb Q_p)$ for $i\le d-2$ and an embedding $H^{d-1}_\et(\widehat{X}_C,\mbb Q_p)\hookrightarrow  H^{d-1}_\et(\widehat{(X\backslash Z)}_C,\mathbb Q_p)$;
		\item (Crystallineness) $H^i_\et(X_C,\mbb Q_p)$ is a crystalline Galois representation for $i\le d-2$.
	\end{enumerate}
\end{thm*}
Another way to phrase Part 2 of \Cref{intro_applications for schemes} is that having a smooth proper scheme over $K$, if we found a Cohen-Macaulay model over $\mc O_K$ such that singularities are in codimension $d$ in the closed fiber, then \'etale cohomology $H^i_\et(\mstack X_C,\mbb Q_p)$ is still a crystalline $G_K$-representation at least up to degree $d-2$.

\begin{question}
	Is the bound in Theorem \ref{cor:applications for schemes}(2) sharp?
\end{question}

\subsection{Plan of the proof}
The original goal of this paper was to prove \Cref{thm:main theorem}: namely that for Hodge-proper stack over $\mc O_K$ the \'etale $\mbb Q_p$-cohomology of $\widehat{\mstack X}_C$ and $\mstack X_C$ agree. The general strategy for our argument was inspired by the proof of Totaro's inequality given in \cite{BhattLi} by Bhatt and Li. Namely, by \cite[Corollary 4.3.15]{KubrakPrikhodko_pHodge} (at least in the case when $K/\mbb Q_p$ is finite) we already knew that their dimensions agree, and so it remained to show that the map $H^i_\et(\mstack X_{C}, \mathbb Q_p) \ra H^i_\et(\widehat{\mstack X}_{C}, \mathbb Q_p)$ is injective for any $i\ge 0$. Our idea then was to replace "approximation by proper schemes" that had been used in \cite{BhattLi} by taking values of some other cohomology theory. As a result we established injectivity of the above map in the more general $d$-de Rham-proper setting (where it holds in degrees up to $d$, see \Cref{prop:Upsilon for d-Hodge-proper}).

Let us briefly sketch our proof of injectivity. The key idea is to look at the "log-de Rham complex over $B_\dR$" that was defined in \cite{DiaoKaiWenLiuZhu_LogarithmicRH} and that we denote by $\Omega^\bullet_{X,D,\log \dR}\cotimes B_\dR$.  For $\Omega^\bullet_{X,D,\log \dR}\cotimes B_\dR$ one has a natural "$B_\dR$-comparison map" (see \Cref{constr:map theta}): namely, given a smooth adic space $X$ over $K$ with a simple normal crossings divisor $D$ and the complement $U$ one has a natural map $$\Theta_{X,D}\colon \RG_\et(U_{C},\mbb Q_p)\otimes_{\mbb Q_p}B_\dR \tto \RG(X, \Omega^\bullet_{X,D,\log \dR}\widehat\otimes B_\dR).$$ As shown in \cite{DiaoKaiWenLiuZhu_LogarithmicRH}, when $X$ is proper the map $\Theta_{X,D}$ is an equivalence. Now, given an affine smooth $\mc O_K$-scheme $U$ one can take the Raynaud generic fiber $\widehat U_K$ with an empty divisor $D=\emptyset$, or take the algebraic generic fiber $U_K$ and consider the analytification of any compactification $X$ of $U_K$ by a normal crossings divisor $D$. For any choice of $(X,D)$ there is a natural map of pairs $(\widehat{U}_K,\emptyset)\ra (X^\an,D^\an)$ and, consequently,  a transformation between maps $\Theta_{\widehat{U}_K,\emptyset}$ and $\Theta_{X^\an,D^\an}$.

We then notice two things: first, that the category $\Comp(U_K)_\nc$ of compactifications as above is weakly contractible and, second, that the functor $(X,D)\mapsto \RG(X^\an, \Omega^\bullet_{X^\an,D^\an ,\log \dR}\widehat\otimes B_\dR)$ is identified with the constant functor $(X,D)\mapsto \RG_\dR(U_K/K)\otimes_K B_\dR$. From this we obtain a functorial commutative square \Cref{eq:key commutative diagram} for any $U$, which after right Kan extension gives a commutative square 
\begin{equation*}
\xymatrix{
	R\Gamma_\et(\mstack X_{C}, \mathbb Q_p)\otimes_{\mathbb Q_p} B_\dR \ar[r]^{\Upsilon_{\mstack X,\mbb Q_p}}\ar[d]_\sim & R\Gamma_\et(\widehat{\mstack X}_{C}, \mathbb Q_p)\otimes_{\mathbb Q_p} B_\dR \ar[d] \\
	R\Gamma_\dR(\mstack X_K/K)\otimes_{K} B_\dR \ar[r] & R\Gamma_\dR(\widehat{\mstack X}_K/B_\dR),
}\end{equation*}
(see \Cref{constr:key construction} and \Cref{prop:commutative diagram}). The left vertical arrow here is an equivalence for any smooth quasi-compact quasi-separated Artin stack $\mstack X$, but when $\mstack X$ is $d$-de Rham proper over $\mc O_K$, the bottom horizontal one gives an isomorphism in degrees up to $d-1$ and an embedding in degree $d$ (\Cref{prop:key proposition}). Consequently, the composition induces an embedding in cohomological degrees up to $d$, which forces $\Upsilon_{\mstack X,\mbb Q_p}$ to be injective in the same range (\Cref{prop:Upsilon for d-Hodge-proper}).

To prove the $B_\crys$-comparison in the case of $(d+1)$-de Rham-proper stacks over $\mc O_K$ we extend some relevant results of \cite[Section 2 and 4]{KubrakPrikhodko_pHodge} to this setting. This we do mainly in Sections \ref{ssec:complete modules miscellany} and \ref{ssec:integral p-adic Hodge theory}. In this context it is only true that the truncation $\tau^{\le d}\RG_{\Prism}(\mstack X/\mf S)$ is coherent (\Cref{cor:BK prismatic cohomology d-coherent}), and it causes a problem for establishing the usual comparisons with other cohomology theories (like $\Ainf$ or \'etale). The problem is simple: the completed tensor products are not necessarily $t$-exact and if one blindly follows the strategy of \textit{loc.cit.} one usually loses 1 or 2 last cohomological degrees in the comparisons. The problem is there even in the case when we tensor up with an $I$-completely free module: namely, if $I$ has at least two generators the derived $I$-completed direct sum functor can very well be not $t$-exact. 
At least in the latter case situation is better in the condensed world: namely, for prodiscrete $I$-complete solid modules derived $I$-completed direct sums are $t$-exact (see \Cref{prop:completed direct sum preserves t-structure}). By an elaborate argument, that we perform in \Cref{ssec: crystalline comparison} we are able to deduce the $B_\crys$-comparison from a condensed version of it. However to do so, we need to use a slightly unusual, however nicely behaved, period ring $B_\mmax$. The main property of the non-coherent part $\tau^{\le d+1}\RG_{\Prism}(\mstack X/\mf S)$ of prismatic cohomology that allows to control the interplay between usual and condensed worlds is the following.  Namely $d$-th and $(d-1)$-st cohomology groups of the derived reduction $[\tau^{\le d+1}\RG_{\Prism}(\mstack X/\mf S))/(p,u)]$ are finite dimensional $k$-vector spaces.  As we show in \Cref{sec:appendix about complete sums}, this property implies a good behavior with respect to completed direct sums or certain more complicated completed tensor products, and, most importantly, that the derived $(p,u)$-completion $(\tau^{\le d+1}\RG_{\Prism}(\mstack X/\mf S))^\bt_{(p,u)}$ in solid modules is still concentrated in cohomological degrees $\ge d+1$ (\Cref{lem_discretness_of_cond_completion}).

Finally, in \Cref{sec:applications for schemes} we record the applications of the above results in the case of schemes (\Cref{cor:applications for schemes}).

\subsection{Notations and conventions}\label{ssec: Notations}
\begin{enumerate}[wide,itemindent=*]
\item In this work by Artin stacks we always mean (higher) Artin stacks in the sense of \cite[Section 1.3.3]{TV_HAGII} or \cite[Chapter 2.4]{GaitsRozI}: these are sheaves in \'etale topology admitting a smooth $(n-1)$-representable atlas for some $n\ge 0$. We stress that we (mostly) work with non-derived Artin stacks, i.e. they are defined on the category of \emph{ordinary} commutative rings.


\item Let $R$ be a commutative ring equipped with an ideal $I$ and let $\fcat C$ be a presentable stable $R$-linear $(\infty, 1)$-category. An object $X \in \fcat C$ is called \emdef{derived $I$-complete} if for every $f\in I$ the limit
$$T(X, f) := \prolim\left(\xymatrix{\ldots \ar[r]^-{\cdot f} & X \ar[r]^-{\cdot f} & X \ar[r]^-{\cdot f} & X} \right)$$
vanishes. We denote the full subcategory of $\fcat C$ spanned by derived $I$-complete objects by $\fcat C_{I-\comp}$. If $I$ is finitely generated then the inclusion functor $\fcat C_{I-\comp} \inj \fcat C$ admits a left adjoint $X \mapsto X_I^\wedge$ that can be explicitly described (see formula \eqref{eq:derived completion}). Also see \cite[Tag 091N]{StacksProject} and \cite[Chapter 7]{Lur_SAG} for more details.

\item Some arguments in the work rely on the theory of condensed mathematics developed by Clausen--Scholze in \cite{ClausenScholze_Condensed1}. We closely follow the notations from \textit{loc.~cit.} In particular, for an $(\infty, 1)$-category $\fcat C$ we denote the category of $\fcat C$-valued sheaves on the pro-\'etale site of a point by $\Cond(\fcat C)$. We also denote the natural functor $\Top \to \Cond(\Set)$ for \cite[Lecture I]{ClausenScholze_Condensed1} by $X\mapsto \underline X$. This functor induces a fully-faithful embedding $D(\Ab) \inj D(\Cond(\Ab))$ which by abuse of notation we will also denote by $M \mapsto \underline M$.

\end{enumerate}

\paragraph{Acknowledgments.} This paper owes its existence to Sasha Petrov who pointed us to the log-$B_\dR$-cohomology of Diao--Lan--Liu--Zhu and sketched how it could help to prove the injectivity of the map between \'etale cohomology of two generic fibers. The generalization to $d$-de Rham stacks was fundamentally inspired by conversations with Shizhang Li who pointed us to some potential applications in schematic setting (which are now \Cref{cor:applications for schemes}). The optimal bounds for crystalline comparison would not be there without the help of Peter Scholze, who showed to the first author how condensed mathematics could be of some help here and patiently explained some basic aspects of the theory. We are also grateful to Sasha Petrov, Shizhang Li and Haoyang Guo for the comments on different versions of the draft and, in the latter case, also writing the Appendix \ref{Appendix: Cohomological descent, de Rham comparison, and resolution of singularities} that extends the generality of \Cref{cor:applications for schemes}.

The first author would like to express his gratitude to Max Planck Institute for the excellent work conditions during his stay there. He is also grateful to Institut des Hautes Études Scientifiques where the last parts of this manuscript were written. The study has been funded within the framework of the HSE University Basic Research Program.

\section{$d$-de Rham and $d$-Hodge-proper stacks}

\subsection{Definition and basic properties}
One can naturally introduce the following truncated analogue of Hodge-properness condition.

\begin{defn}\label{def:Hodge-proper up to degree d}
	A smooth quasi-compact quasi-separated Artin stack $\mstack X$ over a Noetherian ring $R$ is called \emph{$d$-Hodge-proper} if the Hodge cohomology $H^{j,i}(\mstack X/R)\coloneqq H^i(\mstack X,\wedge^j\mbb L_{\mstack X/R})$ is finitely generated for all $i+j\le d$.
\end{defn}
\begin{rem}\label{rem:reformulation of d-Hodge-properness}
	In other words, a smooth qcqs Artin stack $\mstack X$ is {$d$-Hodge-proper} if $\tau^{\le d-j}\RG(\mstack X,\wedge^j\mbb L_{\mstack X/R})\in \Coh(R)$ for any $j\ge 0$.
\end{rem}
\begin{rem}
	Obviously, a stack $\mstack X$ is Hodge-proper (see \Cref{intro_defn_of_hodge_properness}) if and only if it is $d$-Hodge-proper for any $d\ge 0$.
\end{rem}

The $d$-Hodge-properness condition is much more flexible than just Hodge-properness. In particular, there are many more schemes that are Hodge-proper up to some degree, but are not themselves proper:

\begin{lem}\label{lem:example of d-Hodge-proper}
	Let $X$ be a Cohen--Macauley scheme that is proper over $\Spec R$. Assume $Z\subset X$  is a closed $R$-subscheme that has codimension $d+2$. Assume that the complement $U\coloneqq X\!\setminus \!Z$ is smooth over $R$. Then $U$ is Hodge-proper up to degree $d$.
\end{lem}
\begin{proof}
	Let $j\colon U\ra X$ be the embedding. Since $X$ is Cohen--Macaulay and $Z$ has codimension $d+2$ we have $R^0j_*\mc O_U=\mc O_X$ and $R^ij_* \mc O_U=0$ for $0< i\le d$. From \cite[Tag 0BLT]{StacksProject} it follows that for $0\le i\le d$ $R^ij_* \mc F$ is a coherent sheaf when $\mc F$ is finite locally free. In other words, $\tau^{\le d}\mc F\in \Coh(X)$. Taking $\mc F=\Omega^k_{U/R}$ and applying \Cref{lem:useful lemma} to the global sections functor $\RG\colon \QCoh(X) \ra \DMod{R}$ (which is left $t$-exact) and $M=Rj_* \Omega^k_{U/R}$ we get that for any $k\ge 0$
	$$
	\tau^{\le d}\RG(X, \tau^{\le d}Rj_*\Omega^k_{U/R})\ism \tau^{\le d}\RG(U, \Omega^k_{U/R}).
	$$
	Since $X$ is proper over $R$ we get $\tau^{\le d}\RG(U, \Omega^k_{U/R})\in \Coh(R)$ for any $k\ge 0$. We are done by \Cref{rem:reformulation of d-Hodge-properness}. 
\end{proof}

The property that will be actually used by us in the proofs will even be slightly weaker, though completely analogous in the spirit:
\begin{defn}
A smooth quasi-compact quasi-separated Artin stack $\mstack X$ over a Noetherian ring $R$ is called \emph{$d$-de Rham-proper} if $H^i_\dR(\mstack X/R)$ is finitely generated for $i\le d$.
\end{defn}
\begin{rem}\label{rem:reformulation of d-de Rham-properness}
	In other words, a smooth qcqs Artin stack $\mstack X$ is {$d$-de Rham-proper} if $\tau^{\le d}\RG_\dR(\mstack X/R)\in \Coh(R)$.
\end{rem}
The following is immediate:
\begin{lem}
If a smooth Artin stack $\mstack X$ over $R$ is $d$-Hodge-proper it is also $d$-de Rham-proper. 
\end{lem}
\begin{proof}
	By \cite[Corollary 1.1.6(2)]{KubrakPrikhodko_HdR} one has a convergent (Hodge-to-de Rham) spectral sequence $$
	E_1^{i,j}=H^j(\mstack X,\wedge^i\mbb L_{\mstack X/R})\Rightarrow H^{i+j}_\dR(\mstack X/R),
	$$
	from which we see that if $\mstack X$ is $d$-Hodge-proper then $H^i_\dR(\mstack X/R)$ is finitely generated over $R$ for $i\le d$.
\end{proof}
As we will see, the following simple lemma can be used to control the "degree of coherence" of complexes under certain operations:

\begin{lem}\label{lem:useful lemma}
	Let $\fcat C$ and $\fcat D$ be stable $\infty$-categories endowed with $t$-structures. Let $F\colon \fcat C \ra \fcat D$ be a functor which is left $t$-exact up to a shift by $s$.\footnote{More precisely, by this we mean that $F[-s]$ is left $t$-exact.} Then for any $d\in \mbb Z$ and any $M\in \fcat C$ one has a natural equivalence
	$$
	\tau^{\le d-s} F(\tau^{\le d}M) \ism \tau^{\le d-s} F(M).
	$$
\end{lem}
\begin{proof}
	Omitted.
\end{proof}

\begin{prop}\label{prop:general properties of d-Hodge-properness}
	Let $\mstack X$ be a smooth stack over $R$ that is $d$-Hodge-proper (resp. $d$-de Rham-proper).
	\begin{enumerate}
		\item Let $R\ra R'$ be a faithfully flat map of Noetherian rings. Then $\mstack X$ is $d$-Hodge-proper (resp. $d$-de Rham-proper) if and only if the base change $\mstack X_{R'}\coloneqq \mstack X \times_{\Spec R} \Spec R'$ is.
		\item Let $R\ra R'$ be a map of Noetherian rings that has Tor-amplitude $[-s,0]$. Then the base change $\mstack X_{R'}$ is $(d-s)$-Hodge proper (resp. $(d-s)$-de Rham-proper).
	\end{enumerate}
\end{prop}
\begin{proof}
Consider the projection $q\colon \mstack X_{R'} \ra \mstack X$. Since $\mstack X$ is smooth (and qcqs) and $R\ra R'$ is of finite Tor-amplitude we can apply base change to get an equivalence
$$
\RG(\mstack X, \wedge^i\mbb L_{\mstack X/R})\otimes_R R'\ism \RG(\mstack X_{R'}, q^*(\wedge^i\mbb L_{\mstack X/R})),
$$
By base change for cotangent complex (and its exterior powers, see \cite[Corollary A.2.48]{KubrakPrikhodko_pHodge}) we also have $q^*(\wedge^i\mbb L_{\mstack X/R})\simeq \wedge^i\mbb L_{\mstack X_{R'}/R'}$. Similarly, for de Rham cohomology we have
$$
\RG_\dR(\mstack X/R)\otimes_R R'\ism \RG_\dR(\mstack X_{R'}/R').
$$ In point 1, by flatness of $R'$, $\tau^{\le d-i}\RG(\mstack X_{R'},\wedge^i\mbb L_{\mstack X_{R'}/R'})\simeq \RG(\mstack X, \wedge^i\mbb L_{\mstack X/R})\otimes_R R'$ for any $i\ge 0$ and we are done by \Cref{lem:descent for Coh} below. The proof for de Rham cohomology is analogous. For point 2, \Cref{lem:useful lemma}  (applied to the tensor product $-\otimes_R R'$ and $M=\RG(\mstack X, \wedge^i\mbb L_{\mstack X/R})$) gives us an equivalence
$$
\tau^{\le d-s-i}\left(\tau^{\le d-i}\RG(\mstack X, \wedge^i\mbb L_{\mstack X/R})\otimes_R R'\right)\ism  \tau^{\le d-s-i}\RG(\mstack X_{R'},\wedge^i\mbb L_{\mstack X_{R'}/R'})
$$
for any $i\ge 0$ and, similarly, also
$$
\tau^{\le d-s}(\tau^{\le d}(\RG_\dR(\mstack X/R))\otimes_R R')\ism \tau^{\le d-s}\RG_\dR(\mstack X_{R'}/R').
$$
Since $\tau^{\le d-i}\RG(\mstack X, \wedge^i\mbb L_{\mstack X/R})\in \Coh(R)$ for any $i\ge 0$ by $d$-Hodge-properness of $\mstack X$ we have $\tau^{\le d-i}\RG(\mstack X, \wedge^i\mbb L_{\mstack X/R})\otimes_R R'\in \Coh(R')$ (see e.g. the argument in \cite[Lemma 1.1.6]{KubrakPrikhodko_pHodge}), which then also holds for the $(d-s-i)$-th truncation. This shows that $\mstack X_{R'}$ is $(d-s)$-Hodge-proper. The rest of the proof for the de Rham-proper condition is similar.
\end{proof}

\begin{lem}\label{lem:descent for Coh}
	Let $R\ra R'$ be a faithfully flat map. Then a complex $M\in \DMod{R}$ lies in $\Coh(R)$ is and only if $M\otimes_RR'$ lies in $\Coh(R')$. 
\end{lem}
\begin{proof}
	By flatness we have $H^i(M\otimes_RR')\simeq H^i(M)\otimes_R R'$. In particular, $M$ is bounded if and only if $M\otimes_RR'$ is. Finally, by \cite[Tag 03C4(2)]{StacksProject}, $H^i(M)\otimes_R R'$ is finitely generated over $R'$ if and only if $H^{i}(M)$ is.
\end{proof}

\section{$\mbb Q_p$-\'etale cohomology, de Rham cohomology over $B_\dR$ and log period sheaves}

\subsection{De Rham cohomology over $B_\dR$}
In this section we show that the "de Rham cohomology over $B_\dR$" of the Raynaud generic fiber of a $d$-Hodge proper stack agrees with the algebraic de Rham cohomology up to degree $d-1$. We describe in detail the notion of $(p,t)$-completed tensor product $-\widehat\otimes_{\mc O_K} B_\dR$ and show that it agrees with the one in \cite{DiaoKaiWenLiuZhu_LogarithmicRH} in the cases of our interest (see \Cref{rem:product agrees with one in the analytic setting}).

 Recall the Fontaine's map $\theta\colon \Ainf \surj \mathcal O_{C}$. Given a finite extension $K$ of $\mbb Q_p$ we can consider the tensor product $A_{\mathrm{inf}, K} \coloneqq \mathcal O_K \otimes_{W(k)} \Ainf$. Let $e\coloneqq [K:K_0]$ be the ramification index. Let also $B_{\dR,K}^+\coloneqq \left(A_{\mathrm{inf}, K}[\frac{1}{p}]\right)^\wedge_{\ker \theta}$ be the corresponding analogue of $B_\dR^+$. Note that the natural map $\mc O_K\otimes_{W(k)}B_\dR^+ \ra B_{\dR,K}^+$ is an equivalence. Indeed, both sides are $t$-adically complete $B_\dR^+$-modules and the map becomes identity modulo $(t)\simeq \ker \theta\subset B_\dR$. 
 
Recall that the embedding $\ol K \inj C\simeq B_\dR^+/(t)$ lifts canonically to a map $\ol K \dashrightarrow B_\dR^+$ since $B_\dR^+$ is a strictly Henselian local ring. One then has a canonical map $B_{\dR,K}^+\simeq \mc O_K\otimes_{W(k)}B_\dR^+\ra B_\dR^+$ which is the tensor product of the identity map $B_\dR^+\xra{\id}B_\dR^+$ and the restriction of $\ol K\ra B_\dR^+$ to $\mc O_K$.

\begin{rem}\label{rem:formular for BdRK}
	Note that $B_{\dR,K}^+$ is in fact isomorphic to the product of $e$ copies of $B_\dR^+$. Indeed, since $B_\dR^+$ is a $\ol K$-algebra and $\mc O_K\otimes_{W(k)}\ol K\simeq \ol K^{\oplus e}$, one has isomorphisms
	$$
	B_{\dR,K}^+\ism \mc O_K\otimes_{W(k)}B_\dR^+\ism (\ol K)^{\oplus e}\otimes_{\ol K} B_\dR^+ \ism (B_\dR^+)^{\oplus e}.
	$$
	The map $B_{\dR,K}^+\ra B_{\dR}^+$ in these terms is just the projection on one of the components. Note that this map is flat.
\end{rem}
The following construction plays a key role for this paper. Below, we also consider the composition $B_{\dR,K}^+\ra B_{\dR}^+\ra B_\dR$ (recall that $B_\dR\coloneqq B_\dR^+[\frac{1}{t}]$). 

 \begin{construction}
 	Let $M$ be a complex of $\mathcal O_K$-modules. We define the \emdef{$(p,t)$-completed tensor products $M\cotimes_{\mc O_K} B_\dR^+$ and $M\cotimes_{\mc O_K} B_\dR$} as follows:
 	\begin{gather*}
M\cotimes_{\mc O_K} B_\dR^+ \coloneqq  \left((M\otimes_{W(k)} \Ainf )_p^\wedge[\tfrac{1}{p}]\right)_{\ker \theta}^\wedge\otimes_{B_{\dR,K}^+}B_\dR^+,\\
 M\cotimes_{\mc O_K} B_\dR \coloneqq  \left((M\otimes_{W(k)} \Ainf )_p^\wedge[\tfrac{1}{p}]\right)_{\ker \theta}^\wedge\otimes_{B_{\dR,K}^+}B_\dR.
 	\end{gather*}

 \end{construction}
\begin{rem}\label{rem_cotimes_BdR_OK}
	By construction, $\mc O_K\widehat\otimes_{\mc O_K}B_\dR\simeq B_{\dR,K}^+\otimes_{B_{\dR,K}^+}B_\dR\simeq B_\dR$. Also, for any $M$ we have $M\cotimes_{\mc O_K} B_\dR\simeq (M\cotimes_{\mc O_K} B_\dR^+)[\frac{1}{t}]$.
\end{rem}
The following lemma will be useful later:
\begin{lem}\label{lem: tensor product is t-exact up to a shift}
	The $(p,t)$-completed tensor product functor $-\cotimes_{\mc O_K} B_\dR\colon \DMod{\mc O_K} \ra \DMod{B_\dR}$ is left $t$-exact up to a shift by $1$.
\end{lem}
\begin{proof}
	Note that the functor $M\mapsto (M\widehat\otimes_{W(k)}\Ainf)[\frac{1}{p}]/\xi^k$ is left $t$-exact up to a shift by 1. Indeed, for $k=1$ we have $M\mapsto (M\widehat\otimes_{W(k)}\mc O_C)[\frac{1}{p}]$, which is the composition of tensor product $-\otimes_{W(k)}\mc O_C$ (which is $t$-exact since $\mc O_C$ is $p$-torsion free), derived $p$-adic completion (which is left $t$-exact up to a shift by 1) and  localization (which is $t$-exact). To get the statement for a general $k$ one can argue by induction using the fiber sequence 
	$$
(M\widehat\otimes_{W(k)}\Ainf)[\tfrac{1}{p}]/\xi^{k-1} \tto 	(M\widehat\otimes_{W(k)}\Ainf)[\tfrac{1}{p}]/\xi^k \tto (M\widehat\otimes_{W(k)}\Ainf)[\tfrac{1}{p}]/\xi.
	$$
	$M\cotimes_{\mc O_K} B_\dR$ then is obtained by first taking limit of $(M\widehat\otimes_{W(k)}\Ainf)[\tfrac{1}{p}]/\xi^k$ over $k$ 
and then the tensor product $-\otimes_{B_{\dR,K}^+}B_\dR$ (which is $t$-exact by \Cref{rem:formular for BdRK}). Since limits are left $t$-exact this gives the claim. 
\end{proof}

 \begin{rem} The map $\theta$ and the natural embedding $\mathcal O_K \inj \mathcal O_{C}$ agree on $W(k)$ and produce a natural morphism
 $$\theta_K\colon A_{\mathrm{inf}, K} \xymatrix{\ar@{->>}[r] &} \mathcal O_{C}.$$
One can show that there is also another formula for $M\widehat\otimes_{\mc O_K} B_\dR^+$ in terms of $\ker \theta_K$-adic completion
$$
M\cotimes_{\mc O_K} B^+_\dR \coloneqq  \left((M\otimes_{W(k)} \Ainf )_p^\wedge[\tfrac{1}{p}]\right)_{\ker \theta_K}.
$$
\end{rem}
Now we can define "Hodge" and "de Rham" cohomology of the Raynaud generic fiber "over $B_\dR$".
\begin{construction}\label{constr:Hodge and de Rham cohomology over BdR}
Let $U$ be a smooth affine scheme over $\mathcal O_K$. We define
$$R\Gamma(\widehat U_K, \Omega^i\cotimes_{\mc O_K} B_\dR) \coloneqq \Omega^i_U\cotimes_{\mc O_K} B_\dR \quad\text{and}\quad R\Gamma_\dR(\widehat U_K /B_\dR) \coloneqq \Omega^\bullet_{U, \dR} \cotimes_{\mc O_K} B_\dR.$$
We call $R\Gamma_\dR(\widehat U_K /B_\dR)$ \textit{the de Rham cohomology of $U_K$ over $B_\dR$}.
We extend this definition to all smooth Artin $\mathcal O_K$-stacks via the right Kan extension. That is, for a smooth Artin stack $\mc O_K$-stack $\mstack X$ we have 
\begin{gather*}
R\Gamma(\widehat{\mstack X}_K, \Omega^i\cotimes_{\mathcal O_K} B_\dR) \simeq \lim_{U \in \mathrm{Aff}^{\mathrm{sm}, \op}_{/\mstack X}} R\Gamma(\widehat U_K, \Omega^i\cotimes_{\mathcal O_K} B_\dR),\\
R\Gamma_\dR(\widehat{\mstack X}_K / B_\dR) \simeq \lim_{U \in \mathrm{Aff}^{\mathrm{sm}, \op}_{/\mstack X}} R\Gamma_\dR(\widehat U_K / B_\dR).
\end{gather*}
One also has an obivous version with $B_\dR^+$ instead of $B_\dR$.
\end{construction}
\begin{rem}
Note that for a smooth affine scheme $U$ the complex $R\Gamma_\dR(\widehat U_K / B_\dR)$ admits a finite Hodge filtration with associated graded pieces $R\Gamma(\widehat U_K, \Omega^i\cotimes_{\mathcal O_K} B_\dR)$. Passing to the limit we find a similar complete Hodge filtration on $R\Gamma_\dR(\widehat{\mstack X}_K / B_\dR)$ for any smooth Artin $\mathcal O_K$-stack $\mstack X$.
\end{rem}
\begin{rem}\label{rem:product agrees with one in the analytic setting}
	We point out that the definitions in \Cref{constr:Hodge and de Rham cohomology over BdR} agree with the ones in \cite[Section 3.1]{liu2017rigidity} and \cite[Definition 3.1.1]{DiaoKaiWenLiuZhu_LogarithmicRH} in the case of Raynaud generic fiber of a smooth scheme. Indeed, let $X=\Spa(R,R^+)$ be a smooth affinoid adic space over $\Spa(K,\mc O_K).$ The complex $\RG(X_\an, \Omega^i_X\widehat\otimes B_\dR^+)$ (in the notations of \cite[Section 3.1]{liu2017rigidity}) is by definition given by  the limit over $n$ of the $p$-completed tensor products $\Omega^i_X\widehat{\otimes}_K B_\dR^+/t^n$. Let $X=\widehat{U}_K$ be the Raynaud generic fiber of a smooth affine $\mc O_K$-scheme $U\coloneqq \Spec A$: in this case $(R,R^+)\simeq (A^\wedge_p[\frac{1}{p}],A^\wedge_p)$. We will construct the natural isomorphisms
	$$
	\Omega^i_U\widehat\otimes_{\mc O_K}B_\dR^+\ism \Omega^i_X\widehat{\otimes}B_\dR^+;
	$$ the analogous isomorphisms $\Omega^i_U\widehat\otimes_{\mc O_K}B_\dR \xra{\sim}\Omega^i_X\widehat{\otimes}B_\dR$ and $\Omega^i_{U,\dR}\widehat\otimes_{\mc O_K}B_\dR\xra{\sim} \Omega^i_{X,\dR}\widehat{\otimes}B_\dR$ will then follow by inverting $t$ and considering the Hodge filtration correspondingly. Note that by smoothness, $\Omega^i_U$ is a perfect $A$-module, and since $\Omega^i_X\simeq \Omega^i_U\otimes_A R$, it is enough to consider the case $i=0$; namely, of the structure sheaf. This reduces to the existence of a natural equivalence
	\begin{equation}\label{eq:comparing with other tensor}
	\left((A\otimes_{W(k)} \Ainf)^\wedge_p[\tfrac{1}{p}]/\xi^n\right)\otimes_{B_{\dR,K}^+/t^n} B_\dR^+/t^n \ism R\widehat{\otimes}_K B_\dR^+/t^n
	\end{equation}
	for any $n\ge 0$.
	For this recall how the completion in the tensor product $R\widehat{\otimes}_K (B_\dR^+/t^n)$ can be defined: namely, the submodule $R^+\otimes_{W(k)}(\Ainf/\xi^n)\subset R{\otimes}_K (B_\dR^+/t^n)$ defines a $p$-adic lattice and $R\widehat{\otimes}_K (B_\dR^+/t^n)$ is the completion of $R{\otimes}_K (B_\dR^+/t^n)$ with respect to the corresponding topology. Consider another tensor product $R\widehat{\otimes}_{K_0} (B_\dR^+/t^n)$ with the $p$-adic completion defined by $R^+ {\otimes}_{W(k)}(\Ainf/\xi^n) \subset R {\otimes}_{K_0} (B_\dR^+/t^n)$; one has a formula 
	$$
	R\widehat{\otimes}_{K_0} (B_\dR^+/t^n)\simeq (R^+ {\otimes}_{W(k)}(\Ainf/\xi^n))^\wedge_p[\tfrac{1}{p}].
	$$
	Note that since $R^+\simeq A^\wedge_p$ we have an equivalence $$ (A {\otimes}_{W(k)}(\Ainf/\xi^n))^\wedge_p\ism (R^+ {\otimes}_{W(k)}(\Ainf/\xi^n))^\wedge_p.$$ Also, since $\Ainf/\xi^n$ is a perfect $\Ainf$-module, one can commute it through the completion: $$(A\otimes_{W(k)} \Ainf)^\wedge_p[\tfrac{1}{p}]/\xi^n\ism A {\otimes}_{W(k)}(\Ainf/\xi^n))^\wedge_p[\tfrac{1}{p}].$$
	 Finally, note that $R\widehat{\otimes}_{K_0} (B_\dR^+/t^n)$ is a  $K\otimes_{K_0}(B_\dR^+/t^n)\simeq B_{\dR,K}^+/t^n$-algebra and, moreover, that  $$R\widehat{\otimes}_K B_\dR^+/t^n\simeq R\widehat{\otimes}_{K_0} (B_\dR^+/t^n)\otimes_{B_{\dR,K}^+/t^n}(B_\dR^+/t^n).$$ Putting the last three isomorphisms together gives the desired isomorphism \Cref{eq:comparing with other tensor}.
\end{rem}	
We will need the following basic result about $B_\dR$-cohomology.
\begin{prop}\label{prop:key proposition}
Let $\mstack X$ be a smooth Artin stack over $\mc O_K$. Then
\begin{enumerate}[label=(\arabic*)]
\item The natural map
$$R\Gamma(\mstack X, \wedge^i \mathbb L_{\mstack X/\mathcal O_K}) \cotimes_{\mathcal O_K} B_\dR \tto R\Gamma(\widehat{\mstack X}_K, \Omega^i\cotimes_{\mathcal O_K} B_\dR)$$
is an equivalence.
\item So is the natural map $\RG_\dR(\mstack X/\mc O_K)\cotimes_{\mathcal O_K} B_\dR \tto R\Gamma_\dR(\widehat {\mstack X}_K/B_\dR)$.
\item If additionally $\mstack X$ is $d$-de Rham-proper over $\mc O_K$, then the natural map
$$
\RG_\dR(\mstack X_K/K)\otimes_K B_\dR \tto \RG_\dR(\widehat {\mstack X}_K/B_\dR)
$$
induces an isomorphism $H^i_\dR(X_K/K)\otimes_K B_\dR\simeq H^i_\dR(\widehat {\mstack X}_K/B_\dR)$ for $i\le d-1$ and an embedding $H^d_\dR(X_K/K)\otimes_K B_\dR\hookrightarrow H^d_\dR(\widehat {\mstack X}_K/B_\dR)$
\end{enumerate}

\begin{proof}
For the first assertion note that by construction both sides coincide on smooth affine schemes. Hence to prove the claim it is enough to show that both sides satisfy smooth descent. For the left hand side this follows from the flat descent for cotangent complex (see e.g. \cite[Proposition 1.1.5]{KubrakPrikhodko_HdR}) and \Cref{lem_cotimes_BdR_1} below. For the right hand side, note that by the same lemma the functor on smooth affine schemes that sends $U\mapsto \Omega^i_U\cotimes_{\mathcal O_K} B_\dR$ also satisfies smooth descent. It is then formal that the right Kan extension also satisfies smooth descent. 
The proof for the second part is completely analogous. 

Finally, part (3) follows from (2) and applying \Cref{lem_cotimes_BdR_2} below to $M=\RG_\dR(\mstack X/\mc O_K)$.
\end{proof}
\end{prop}

\begin{rem}
	If $\mstack X$ is Hodge-proper over $\mc O_K$, then it is $d$-Hodge and, consequently, $d$-de Rham proper for any $d\ge 0$. Thus, from \Cref{prop:key proposition}(3) we get that the map 
	$$
	\RG_\dR(\mstack X_K/K)\otimes_K B_\dR \tto \RG_\dR(\widehat {\mstack X}_K/B_\dR)
	$$
	is a quasi-isomorphism.
\end{rem}
\begin{lem}\label{lem_cotimes_BdR_1}
Let $M^\bullet$ be a co-simplicial diagram of uniformly bounded below complexes of $\mathcal O_K$-modules. Then the natural map
\begin{gather*}
\Tot(M^\bullet) \cotimes_{\mathcal O_K} B_\dR \tto \Tot(M^\bullet\cotimes_{\mathcal O_K} B_\dR)
\end{gather*}
is an equivalences.

\begin{proof}
This follows from \cite[Corollary C.6]{KubrakPrikhodko_pHodge}, since the functor $-\cotimes_{\mathcal O_K} B_\dR$ is left $t$-exact up to a shift (\Cref{lem: tensor product is t-exact up to a shift}).
\end{proof}
\end{lem}
\begin{lem}\label{lem_cotimes_BdR_2}
Let $M \in \DMod{\mc O_K}$ be such that $\tau^{\le d} M\in \Coh(\mathcal O_K)$. Then the natural map
$$M\otimes_{\mathcal O_K} B_\dR \tto M\cotimes_{\mathcal O_K} B_\dR$$
induces an isomorphism $H^i(M)\otimes_{\mc O_K}B_\dR\simeq H^i(M\cotimes_{\mathcal O_K} B_\dR)$ for $i\le d-1$ and an embedding $H^d(M)\otimes_{\mc O_K}B_\dR\hookrightarrow H^d(M\cotimes_{\mathcal O_K} B_\dR)$.

\begin{proof}
The map in question is an equivalence for $M = \mathcal O_K$ (\Cref{rem_cotimes_BdR_OK}), hence it is also so for all perfect (or equivalently, coherent) $\mathcal O_K$-modules. This applies in particular to $\tau^{\le d}(M)$. Recall that $-\cotimes_{\mc O_K}B_\dR$ is left $t$-exact up to a shift by 1, while $-\otimes_{\mc O_K}B_\dR$ is just left $t$-exact since $B_\dR$ is $p$-torsion free. From \Cref{lem:useful lemma} we get a commutative square 
$$
\xymatrix{\tau^{\le d-1}(\tau^{\le d}(M)\otimes_{\mc O_K}B_\dR)\ar[r]^\sim\ar[d]^\wr & \tau^{\le d-1}(\tau^{\le d}(M)\cotimes_{\mc O_K}B_\dR)\ar[d]^\wr\\
	\tau^{\le d-1}(M\otimes_{\mc O_K}B_\dR)\ar[r] & \tau^{\le d-1}(M\cotimes_{\mc O_K}B_\dR),
}
$$
which shows that the low horizontal map is an equivalence. By passing to cohomology, this gives the statement for $i\le d-1$. 

Note also that for any $i$ (from left $t$-exactness of $-\cotimes_{\mc O_K}B_\dR$ up to a shift by 1) we have a short exact sequence
$$
0 \ra H^0(H^i(M)\cotimes_{\mc O_K} B_\dR) \ra H^i(M\cotimes_{\mc O_K} B_\dR) \ra H^{-1}( H^{i+1}(M)\cotimes_{\mc O_K} B_\dR) \ra 0.
$$ 
By the assumption, for $i=d$ the module $H^d(M)$ is still finitely generated, so $H^d(M)\cotimes_{\mc O_K} B_\dR\simeq H^d(M)\otimes_{\mc O_K} B_\dR$, and the short exact sequence gives that $H^d(M)\otimes_{\mc O_K}B_\dR\hookrightarrow H^d(M\cotimes_{\mathcal O_K} B_\dR)$.
\end{proof}
\end{lem}

\subsection{Log period sheaves and $B_\dR$-comparison map}
Let $X$ be a smooth adic space over $\Spa(K,\mc O_K)$. Let $i\colon D \ra X$ be a normal crossings divisor (see \cite[Example 2.3.17]{log_adic} for the definition). The divisor $D$ endows $X$ with a natural structure of a (smooth) log adic space (see \cite[Example 2.1.2]{DiaoKaiWenLiuZhu_LogarithmicRH}). We will denote by $X_{D,\ket}$, $X_{D,\proket}$ the associated Kummer \'etale (see \cite[Section 4.1]{log_adic}) and pro-Kummer \'etale (see \cite[Definition 5.1.2]{log_adic}) sites correspondingly.

\begin{ex}
	When $D=\emptyset$ is empty, $X_{D,\ket}$ and $X_{D,\proket}$ are identified with the \'etale and pro-\'etale sites of $X$. 
\end{ex}

\begin{rem}
	Let $U\coloneqq X\setminus D$ be the complement to $D$ and let  $j\colon U \ra X$ be the natural embedding. Pull-back of $D$ to $U$ is empty, $U_{\emptyset,\ket}\simeq U_\et$, and thus one has a natural (derived) push-forward functor $j_{*,\ket}\colon \Shv(U_\et,\mbb Z/n)\ra \Shv(X_{D,\ket},\mbb Z/n)$ between the derived categories of sheaves of $\mbb Z/n$-modules (for any $n\ge 0$). By \cite[Lemma 4.6.5]{log_adic} for the constant sheaf $\mbb Z/n\in \Shv(U_\et,\mbb Z/n)$ one has an equivalence $j_{*,\ket} \mbb Z/n \xra{\sim} \mbb Z/n$; this then induces an equivalence
	\begin{equation*}
	\RG(U_\et,\mbb Z/n) \ism  \RG(X_{D,\ket},\mbb Z/n)
	\end{equation*}
	for the global sections. Taking $n=p^k$ and passing to the limit over $k$ one also gets an equivalence
	\begin{equation}\label{eq:kummer-etale vs etale}
	\RG(U_\et,\mbb Z_p) \ism \RG(X_{D,\ket},\mbb Z_p).
	\end{equation}
\end{rem}


In \cite{DiaoKaiWenLiuZhu_LogarithmicRH} the period sheaves $\mbb B_\dR^+$, $\mbb B_\dR$, $\mc O\mbb B_{\dR,\log}^+$, $\mc O\mbb B_{\dR,\log}$ on $X_{D,\proket}$ were introduced. While the definition of $\mbb B_\dR^+$, $\mbb B_\dR$ is essentially the same as in \cite{scholze2013adic} (see \cite[Definition 2.2.3]{DiaoKaiWenLiuZhu_LogarithmicRH}), the definition of $\mc O\mbb B_{\dR,\log}^+$, $\mc O\mbb B_{\dR,\log}$ is more involved and takes into account the log structure in a more significant way (see \cite[Definition 2.2.10]{DiaoKaiWenLiuZhu_LogarithmicRH}).  Let also $\Omega^{1}_{X,D,\log}$ be the sheaf of log differential 1-forms on $X$ (see \cite[Definition 3.2.25]{log_adic}). Similarly, one defines $\Omega^{j}_{X,D,\log}\coloneqq \wedge^j_{\mc O_X}\Omega^{1}_{X,D,\log}$ for all $j\ge 0$ (see \cite[Definition 3.2.29]{log_adic}).

We denote by $\mu\colon X_{D,\proket} \ra X_\an$ the natural map of sites. We also consider the map $\mu'\colon (X_C)_{D,\proket}\simeq (X_{D,\proket})_{/X_C} \ra X_\an$ where $X_C\coloneqq X\times_K C$ is the base change of $X$ to $C\coloneqq \widehat{\ol K}$. For brevity we will continue to denote the pull-backs $\mu^*\Omega^{j}_{X,D,\log}$ and $\mu'^*\Omega^{j}_{X,D,\log}$ by $\Omega^{j}_{X,D,\log}$.

The sheaf $\mc O \mbb B_{\dR,\log}^+$ on $X_{D,\proket}$ has a natural log connection $\nabla
\colon \mc O \mbb B_{\dR}^+ \ra \mc O \mbb B_{\dR}^+\otimes_{\mc O_X} \Omega^1_{X,D,\log}$ (see \cite[2.2.15]{DiaoKaiWenLiuZhu_LogarithmicRH}) which then extends to a log connection $\nabla
\colon \mc O \mbb B_{\dR,\log} \ra \mc O \mbb B_{\dR,\log}\otimes_{\mc O_X} \Omega^1_{X,D,\log}$ on $\mc O \mbb B_{\dR,\log}$ \cite[2.2.17]{DiaoKaiWenLiuZhu_LogarithmicRH}. We then have:
\begin{prop*}[The Poincar\'e lemma, {\cite[Corollary 2.4.2]{DiaoKaiWenLiuZhu_LogarithmicRH}}]
 The complex $$
0\ra \mbb B_\dR \ra \mc O \mbb B_{\dR,\log} \xra{\nabla} \mc O \mbb B_{\dR,\log}\otimes_{\mc O_X} \Omega^1_{X,D,\log} \xra{\nabla} \mc O \mbb B_{\dR} \otimes_{\mc O_X} \Omega^2_{X,D,\log} \xra{\nabla} \ldots
$$
of sheaves on $X_{D,\proket}$ is acyclic.
\end{prop*}

 This gives a quasi-isomorphism $\mbb B_\dR\ism DR_{\log}(\mc O \mbb B_{\dR,\log},\nabla)$ in the derived categories of sheaves on $X_{D,\proket}$ and $(X_C)_{D,\proket}$ (here $DR_{\log}$ denotes the log de Rham complex). 
 
 One can also explicitely describe the pushforwards with respect to the map $\mu'\colon (X_C)_{D,\proket} \ra X_\an$ of the individual terms of the de Rham complex $DR_{\log}(\mc O \mbb B_{\dR,\log},\nabla)$. Namely, recall the sheaf $\mc O_X\widehat\otimes B_\dR$ (\cite[Definition 3.1.1]{DiaoKaiWenLiuZhu_LogarithmicRH}, see also the discussion in \Cref{rem:product agrees with one in the analytic setting}) and consider $\Omega^j_{X,D,\log}\widehat\otimes B_\dR\coloneqq \Omega^j_{X,D,\log}\otimes_{\mc O_X} \mc O_X\widehat\otimes B_\dR$ (these sheaves are called $\Omega^j_{X,\log}\widehat\otimes B_\dR$ in \cite[Definition 3.1.6]{DiaoKaiWenLiuZhu_LogarithmicRH}).  Let $d_{ B_\dR}\colon \mc O_X\widehat\otimes B_\dR \ra \Omega^1_{X,D,\log}\widehat\otimes B_\dR$ be the map induced by the natural log connection $d\coloneqq \mc O_X \ra \Omega^1_{X,D,\log}$ on $\mc O_X$ and let 
 $$
 \Omega^\bullet_{X,D,\log \dR}\widehat\otimes B_\dR\coloneqq \mc O_X\widehat\otimes B_\dR \ra \Omega^1_{X,D,\log}\widehat\otimes B_\dR \ra \Omega^2_{X,D,\log}\widehat\otimes B_\dR \ra \ldots
 $$ be the corresponding log de Rham complex.

  Returning to the previous discussion, one has an equivalence 
 $$
\mc O_X\widehat\otimes B_\dR \ism R\mu'_*(\mc O \mbb B_{\dR,\log}) 
$$
(see \cite[Lemma 3.3.2]{DiaoKaiWenLiuZhu_LogarithmicRH}), which, by projection formula, also gives $\Omega^j_{X,D,\log}\widehat\otimes B_\dR \xra{\sim} R\mu'_*(\mc O \mbb B_{\dR,\log}\otimes \Omega^j_{X,D,\log})$ for any $j\ge 0$. These isomorphisms then ultimately lead to equivalences
\begin{equation}\label{eq:pushfwrd as log de Rham}
\Omega^\bullet_{X,D,\log \dR}\widehat\otimes B_\dR \ism R\mu'_*\left(DR_{\log}(\mc O \mbb B_{\dR,\log},\nabla)\right)\xymatrix{&\ar[l]_\sim} R\mu'_*\left(\mbb B_\dR\right).
\end{equation}
 
\begin{rem}
	It is noted in \cite[Remark 2.2.11]{DiaoKaiWenLiuZhu_LogarithmicRH} that when the log-structure is trivial the construction of the log-period ring $\mc O\mbb B_{\dR,\log}$ still differs from the one of $\mc O\mbb B_{\dR}$ in \cite[Section 6]{scholze2013adic} by a certain completion procedure. However, as noted in \emph{loc.cit} all features essential for the applications (namely the Poincar\'e lemma and the description of the pushforward to the analytic site) stay valid, thus in fact for most of the arguments one can freely choose between any of the two versions.
\end{rem}
\begin{rem}\label{rem:log BdR complex for the empty divisor}
	If $D=\emptyset$ then the complex $\Omega^\bullet_{X,D,\log \dR}\widehat\otimes B_\dR$ on $X$ agrees with $\Omega_{X,\dR}^\bullet\cotimes B_\dR$ (see \cite[Section 3.1]{liu2017rigidity} or \Cref{rem:product agrees with one in the analytic setting} to recall the definition). By \Cref{rem:product agrees with one in the analytic setting}, when $X=\widehat{U}_K$ is the Raynaud generic fiber of a smooth affine $\mc O_K$-scheme $U$, this complex agrees with the "de Rham cohomology over $B_\dR$" $\RG_\dR(\widehat{U}_K/B_\dR)$ that we have defined in \Cref{constr:Hodge and de Rham cohomology over BdR}.
\end{rem}

We can now pass to the construction of the $B_\dR$-comparison map.
\begin{construction}\label{constr:map theta}
Let $X$, $i\colon D\hookrightarrow X$ be as in the beginning of this section and let $U\coloneqq X \setminus D$. Let $\RG_\et(U_C,\mbb Q_p)\coloneqq \RG_\et(U_C,\mbb Z_p)[\frac{1}{p}]$ and $\RG((X_C)_{D,\ket},\mbb Q_p)\coloneqq \RG((X_C)_{D,\ket},\mbb Z_p)[\frac{1}{p}]$. By \Cref{eq:kummer-etale vs etale} we have an equivalence 
$$
\RG_\et(U_C,\mbb Q_p)\ism \RG((X_C)_{D,\ket},\mbb Q_p).
$$
Having this (with a further identification $\RG((X_C)_{D,\ket},\mbb Z_p)\simeq \RG((X_C)_{D,\proket},\mbb Z_p)$), the map $\mbb Z_p \ra \mbb B_\dR$  of sheaves on $X_{D,\proket}$ induces a natural map
$$
\RG_\et(U_C,\mbb Q_p)\otimes_{\mbb Q_p}B_\dR \xymatrix{\ar[r]&}\RG(X_{D,\proket},\mbb B_\dR).
$$
Composing further with the equivalence $\RG(X_{D,\proket},\mbb B_\dR) \ism \RG(X_\an, \Omega^\bullet_{X,D,\log \dR}\widehat\otimes B_\dR)$ coming from \Cref{eq:pushfwrd as log de Rham} we get a natural map 
$$
\Theta_{X,D}\colon \RG_\et(U_C,\mbb Q_p)\otimes_{\mbb Q_p}B_\dR \xymatrix{\ar[r]&} \RG(X_\an, \Omega^\bullet_{X,D,\log \dR}\widehat\otimes B_\dR)
$$
which we will call the \textit{$B_\dR$-comparison map}.
\end{construction}

\begin{ex}\label{ex: log-BdR vs usual dR}
	By \cite[Theorem 3.2.3(3)]{DiaoKaiWenLiuZhu_LogarithmicRH} if $X$ is proper the map $\Theta_{X,D}$ is an equivalence for any normal crossings divisor $D$. Moreover, by \cite[Lemma 3.6.2]{DiaoKaiWenLiuZhu_LogarithmicRH} one can make an identification
	$$
	\RG_{\log \dR}(X/K)\otimes_{K}B_\dR \ism \RG(X, \Omega^\bullet_{X,D,\log \dR}\widehat\otimes B_\dR).
	$$
	In the case of the analytification  $X^\an, D^\an$ of a smooth proper $K$-scheme $X$ with a normal crossings divisor $D$, one can make further natural identifications 
	$$
	\RG_{\log \dR}(X^\an/K)\xymatrix{\ar[r]^{\mr{GAGA}}_\sim&} \RG_{\log \dR}(X/K) \ism \RG_\dR(U/K),
	$$
	where the first equivalence is induced by GAGA and the second one by the restiction of the algebraic log de Rham complex (as a complex of sheaves) to $U\coloneqq X\setminus D$.
\end{ex}

\section{Proof of $\mbb Q_p$-local acyclicity for Hodge-proper stacks}
\subsection{The key commutative square}
Let $U$ be a scheme over $\mc O_K$. Let $U_{K}\coloneqq U\times_{\mc O_K} K$ and $\widehat U_K$ be the algebraic and Raynaud generic fibers of $U$. We have a natural map $\psi_U\colon\widehat U_K\ra U_K^\an$, where $U_K^\an$ is the analytification of $U_K$ (for a more detailed discussion of this situation see \cite[Section 3.3]{KubrakPrikhodko_pHodge}).

Below we will also need to consider the auxiliary data of a compactification $X$ of $U_K$ by a simple normal crossings divisor $D$. To clarify, by this we mean a smooth proper $K$-scheme $X$ with a simple normal crossings divisor $D \hookrightarrow X$ and an isomorphism $\epsilon\colon U_K\xra{\sim }X\setminus D$. Let $\Comp(U_K)_{\nc}$ be the category of such compactifications $(X,D,\epsilon)$. An important property of $\Comp(U_K)_{\nc}$ is that it is cofiltered:
\begin{lem}\label{lem:category of compactifications is filtered}
Let $S$ be a smooth finite type scheme over a characteristic $0$ field. Then the category $\Comp(S)_{\nc}$ of compactifications of $S$ by a simple normal crossings divisor is co-filtered.

\begin{proof}
	First of all $\Comp(S)_{\nc}$ is non-empty. Indeed, by Nagata's compactification there exists a reduced proper scheme $X'$ containing $S$ as a dense open subscheme. By Hironaka's resolution of singularities one then has a birational map $X\surj X'$ which is trivial on $S$ and such that the complement $X\setminus S$ is a simple normal crossing divisor.

	Let $X_1,X_2$ be a pair of compactifications of $S$ by normal crossings divisors. Let $Z^\prime$ be a closure of $S$ in $X_1 \times X_2$. Resolving the singularities as above we find a  $Z\surj Z'$ which is trivial over $S$ and such that the complement $Z\setminus S$ is a simple normal crossings divisor. By construction such $Z$ maps both to $X_1$ and $X_2$ (using projections $X_1 \leftarrow Z'\rightarrow X_2$).
	
	Let now $f,g\colon X_1 \rightrightarrows X_2$ be a pair of parallel arrows in $\Comp(S)_{\nc}$. Consider the closure $Z^\prime$ of the diagonal embedding $S \inj X_1\times_{X_2} X_1$. Let $Z\surj Z^\prime$ be a resolution of singularities as above. By construction there is a map $e\colon Z \to Y$ such that $f\circ e = g\circ e$.
\end{proof}
\end{lem}

\begin{rem}\label{rem:nerve of comp is contraction}
	In particular it follows that the nerve $N(\Comp(S)_{\nc})$ is weakly contractible.
\end{rem}

In the next construction for an affine smooth $\mathcal O_K$-scheme $U$ we construct the following natural commutative square:
\begin{equation}\label{eq:key commutative diagram}
\xymatrix{\RG_\et(U_{C},\mbb Q_p)\otimes_{\mbb Q_p} B_\dR\ar[r]^{\Upsilon_{U,\mbb Q_p}}\ar[d]^\sim& \RG_\et(\widehat U_{C},\mbb Q_p)\otimes_{\mbb Q_p} B_\dR \ar[d]^{\Theta_{\widehat U_K,\emptyset}}\\ 
\RG_\dR(U_K/K)\otimes_K B_\dR \ar[r] & \RG_\dR(\widehat U_K/B_\dR).}
\end{equation}
This will be the key ingredient in the proof of local acyclicity.

\begin{construction}\label{constr:key construction}
	\begin{enumerate}[wide]
		\item Let $S$ be a smooth affine $K$-scheme and let $S^\an$ be its analytification. By associating to $(X,D,\epsilon)\in \Comp(S)_{\nc}$ the pair $(X^\an, D^\an)$ we can further consider a complex $\RG(X^\an, \Omega^\bullet_{X^\an,D^\an,\log \dR}\widehat{\otimes}B_\dR)\in \DMod{B_\dR}$. This extends to a functor 
		$$
		\RG(-^\an, \Omega^\bullet_{-^\an,-^\an,\log \dR}\widehat{\otimes}B_\dR)\colon \left(\Comp(S)_{\nc}\right)^\op \arr \DMod{B_\dR}.
		$$
		Note that since for any $(X,D,\epsilon)\in \Comp(S)_{\nc}$ the scheme $X$ is smooth and proper, so by (the references in) \Cref{ex: log-BdR vs usual dR} there is a sequence of natural identifications 
		
		$$\xymatrix{\RG_\dR(S/K)\otimes_K B_\dR \ar[r]^-\sim & \RG_{\log \dR}(X/K)\otimes_K B_\dR\ar[d]^\sim &  \\&
			\RG_{\log \dR}(X^\an/K)\otimes_K B_\dR\ar[r]^-\sim & \RG(X^\an, \Omega^\bullet_{X^\an,D^\an,\log \dR}\widehat\otimes B_\dR),
			}
		$$
which ultimately identifies $\RG(-^\an, \Omega^\bullet_{-^\an,-^\an,\log \dR}\widehat{\otimes}B_\dR)$ with the constant functor 
$$
\ul{\RG_\dR(S/K)\otimes_K B_\dR}\colon \left(\Comp(S)_{\nc}\right)^\op  \arr \DMod{B_\dR},
$$
that sends any $(X,D,\epsilon)\in \Comp(S)_{\nc}$ to $\RG_\dR(S/K)\otimes_F B_\dR$.
	
		\item Let $U$ be an affine smooth $\mc O_K$-scheme and $U_K$ be the algebraic generic fiber. For any $(X,D, \epsilon)\in \Comp(U_K)_{\nc}$ we have a natural map of pairs\footnote{Consisting of a smooth adic space and a normal crossings divisor.} $(U_K^\an,\emptyset)\ra (X^\an, D^\an)$. Precomposing it with $\psi_U\colon \widehat U_K\ra U_K^\an$ we also get a map $(\widehat U_K,\emptyset)\ra (X^\an, D^\an)$, which further induces a map $$\RG(X^\an, \Omega^\bullet_{X^\an,D^\an,\log}\widehat\otimes B_\dR) \arr \RG(\widehat U_K, \Omega^\bullet_{\widehat U_K,\emptyset ,\log}\widehat\otimes B_\dR)\ism \RG_\dR(\widehat U_K/B_\dR);
		$$
		here for the equivalence on the right see \Cref{rem:log BdR complex for the empty divisor}. Recall also the map $\Theta$ defined in \Cref{constr:map theta}. Applying it to pairs $(X^\an,D^\an)$ and $(\widehat U_K,\emptyset)$ we get a natural commutative square
		$$
		\xymatrix{\RG_\et(U_{C}^\an,\mbb Q_p)\otimes_{\mbb Q_p} B_\dR\ar[r]^{\psi_U^{-1}}\ar[d]_{\Theta_{X,D}}^\sim& \RG_\et(\widehat U_{C},\mbb Q_p)\otimes_{\mbb Q_p} B_\dR \ar[d]^{\Theta_{\widehat U_K,\emptyset}}\\ 
			\RG(X^\an, \Omega^\bullet_{X^\an,D^\an,\log}\widehat\otimes B_\dR) \ar[r] & \RG_\dR(\widehat U_K/B_\dR),}
		$$
		where the left vertical arrow is an equivalence. Passing to colimit over $(X,D)\in \Comp(U_K)_{\nc}$, using the equivalence $\RG(-^\an, \Omega^\bullet_{-^\an,-^\an,\log}\widehat{\otimes}B_\dR)\simeq\ul{\RG_\dR(U_K/K)\otimes_K B_\dR}$ in (1) and the fact that $N(\Comp(U_K)_{\nc})$ is weakly contractible (\Cref{rem:nerve of comp is contraction}) we obtain a natural commutative square 		
	\begin{equation}\label{eq:prekey commutative diagram}
	\xymatrix{\RG_\et(U_{C}^\an,\mbb Q_p)\otimes_{\mbb Q_p} B_\dR\ar[r]^{\psi_U^{-1}}\ar[d]^\sim& \RG_\et(\widehat U_{C},\mbb Q_p)\otimes_{\mbb Q_p} B_\dR \ar[d]^{\Theta_{\widehat U_K,\emptyset}}\\ 
		\RG_\dR(U_K/K)\otimes_K B_\dR \ar[r] & \RG_\dR(\widehat U_K/B_\dR).}
	\end{equation}

\item Finally, by \cite[Theorem 3.8.1]{Huber_adicSpaces} one has a natural equivalence 
$$
\phi_{U}^{-1}\colon \RG_\et(U_{C},\mbb Q_p) \ism  \RG_\et(U_{C}^\an, \mbb Q_p)
$$
(see also \cite[Section 3.3 and Remark 4.1.16]{KubrakPrikhodko_pHodge} to recall the precise setup). The composition $$\psi^{-1}_U\circ \phi^{-1}_U\colon \RG_\et(\widehat U_{C},\mbb Q_p) \tto  \RG_\et(U_{C}, \mbb Q_p)$$ is by definition the comparison map 
$\Upsilon_{U,\mbb Q_p}$ that appears in \Cref{rem:Q_p-acyclic stacks} (see also \cite[Construction 4.1.14]{KubrakPrikhodko_pHodge} for more details). Precomposing left vertical and upper horizontal map in \Cref{eq:prekey commutative diagram} with $\phi_U^{-1}$ we then get the natural commutative square \eqref{eq:key commutative diagram}.

\end{enumerate}
\end{construction}

\begin{rem}\label{rem:about key diagram}
		It is not hard to see from the construction and identification in \Cref{rem:product agrees with one in the analytic setting} that the lower horizontal map  $\RG_\dR(U_K/K)\otimes_K B_\dR \ra \RG_\dR(\widehat U_K/B_\dR)$ in diagrams \Cref{eq:prekey commutative diagram} and \Cref{eq:key commutative diagram} is exactly the map that induces an equivalence in \Cref{prop:key proposition}(2) in the Hodge-proper context. 
\end{rem}

\subsection{The map $\Upsilon_{-,\mbb Q_p}$ for $d$-Hodge-proper stacks}
In this section we establish $\mbb Q_p$-local acyclicity of a Hodge-proper stack over $\mc O_K$. First we extend the square \eqref{eq:key commutative diagram} to smooth Artin $\mathcal O_K$-stacks.
\begin{prop}\label{prop:commutative diagram}
Let $\mstack X$ be a smooth quasi-compact quasi-separated Artin stack over $\mathcal O_K$.	
	 Then there is a commutative square of the form
\begin{equation}\label{eq_main_diagram}
\xymatrix{
R\Gamma_\et(\mstack X_{C}, \mathbb Q_p)\otimes_{\mathbb Q_p} B_\dR \ar[r]^{\Upsilon_{\mstack X,\mbb Q_p}}\ar[d]_\sim & R\Gamma_\et(\widehat{\mstack X}_{C}, \mathbb Q_p)\otimes_{\mathbb Q_p} B_\dR \ar[d] \\
R\Gamma_\dR(\mstack X_K/K)\otimes_{K} B_\dR \ar[r] & R\Gamma_\dR(\widehat{\mstack X}_K/B_\dR),
}\end{equation}
where the left vertical map is an equivalence.

\begin{proof}
Consider the right Kan extension of the commutative square in \Cref{eq:key commutative diagram} along the embedding $(\Aff^\sm_{\mc O_K})^\op \hookrightarrow (\PStk_{\mc O_K})^\op$. We claim that the terms appearing in this Kan extension are the ones in the diagram \Cref{eq_main_diagram}. For the lower right term this holds by definition (\Cref{constr:Hodge and de Rham cohomology over BdR}). For the others note that if a functor $F\colon (\Aff^\sm_{\mc O_K})^\op \ra \DMod{L}^{\ge 0}$ (where $L$ is either $K$ or $\mbb Q_p$) satisfies smooth descent then so does $F\otimes_L B_\dR$: indeed, by \cite[Corollary 3.1.13]{KubrakPrikhodko_pHodge} $-\otimes_L B_\dR$ preserves totalizations of $0$-coconnective objects. The functors $R\Gamma_\et((-)_{C}, \mathbb Q_p)$, $R\Gamma_\et((\widehat{-})_{C}, \mathbb Q_p)$ and $R\Gamma_\dR((-)_K/K)$ satisfy smooth descent, thus picking a hypercover $|U_\bullet|\xra{\sim} \mstack X$ (as in \cite[Theorem 4.7]{Pridham_ArtinHypercovers}) with $U_i$ being smooth affine schemes we reduce to the case of a smooth affine scheme, which is tautological.
\end{proof}
\end{prop}

\begin{prop}\label{prop:Upsilon for d-Hodge-proper}
	Let $\mstack X$ be a smooth $d$-de Rham-proper stack over $\mathcal O_K$. Then the map $\Upsilon_{\mstack X,\mbb Q_p}$ induces an embedding 
	$$
	 \begin{tikzcd}        
	 	H^i_\et(\mstack X_{C}, \mathbb Q_p) \arrow[hook]{r} & H^i_\et(\widehat{\mstack X}_{C}, \mathbb Q_p)  
	 \end{tikzcd} 
	$$
	for $i\le d$.
\end{prop}
\begin{proof}
	From the commutative diagram (\ref{eq_main_diagram}) and \Cref{prop:key proposition}(3) it follows that in the $d$-de Rham-proper case the composition of left vertical and low horizontal maps in (\ref{eq_main_diagram}) is injective on cohomology in the range $i\le d$. Thus so is the vertical map; the claim of the proposition follows.
\end{proof}

\begin{ex}
	Recall that any $d$-Hodge-proper stack is $d$-de Rham proper, so Proposition \ref{prop:Upsilon for d-Hodge-proper} holds for them as well. However, there are examples (see \Cref{ssec:d-de Rham vs d-Hodge}) of $d$-de Rham proper stacks that are not $d$-Hodge-proper, which motivates stating \Cref{prop:Upsilon for d-Hodge-proper} in that generality. 
\end{ex}

We get the following result for Hodge-proper stacks as an immediate consequence.
\begin{thm}\label{thm:main theorem}
Let $\mstack X$ be a smooth Hodge-proper stack over $\mathcal O_K$. Then the natural map
$$\Upsilon_{\mstack X,\mbb Q_p}\colon R\Gamma_\et(\mstack X_{C}, \mathbb Q_p) \tto R\Gamma_\et(\widehat{\mstack X}_{C}, \mathbb Q_p)$$
is an equivalence.

\begin{proof}
Moreover, since the left vertical arrow in \Cref{eq_main_diagram} is an isomorphism we have $\dim_{\mbb Q_p} H^i_\et(\mstack X_{C}, \mathbb Q_p)=\dim_K H^i_\dR(\mstack X_K/K)$. Moreover, since $\mstack X$ is Hodge-proper, we have $\dim_K H^i_\dR(\mstack X_K/K)=\dim_{\mbb Q_p} H^i_\et(\widehat{\mstack X}_{C}, \mathbb Q_p)$ by \cite[Proposition 4.3.14]{KubrakPrikhodko_pHodge}. Thus: $\dim H^i_\et(\mstack X_{C}, \mathbb Q_p)=\dim H^i_\et(\widehat{\mstack X}_{C}, \mathbb Q_p)$  for any $i\ge 0$. On the other hand a Hodge-proper stack is $d$-de Rham-proper for any $d$, and thus by \Cref{prop:Upsilon for d-Hodge-proper} $\Upsilon_{\mstack X,\mbb Q_p}$ is injective on cohomology. So $\Upsilon_{\mstack X,\mbb Q_p}$ induces an isomorphism on cohomology, and thus is a quasi-isomorphism.
\end{proof}
\end{thm}
 
\begin{cor}\label{cor:Fontaine's conjecture for Hodge-proper stacks}
	Let $\mstack X$ be a smooth Hodge-proper stack over $\mc O_K$. Then for any $i\ge 0$ the Galois representation given by $H^i({\mstack X}_{C},\mbb Q_p)$ is crystalline and $D_\crys(H^i({\mstack X}_{C},\mbb Q_p))=H^i_\crys(\mstack X_k/W(k))[\frac{1}{p}]$.
\end{cor}
\begin{proof}
	From \Cref{thm:main theorem} for any $i\ge 0$ we get an isomorphism $H^i({\mstack X}_{C},\mbb Q_p)\simeq H^i(\widehat{\mstack X}_{C},\mbb Q_p)$, which is automatically $G_K$-equivariant. By \cite[Theorem 4.3.25]{KubrakPrikhodko_pHodge}, $H^i(\widehat{\mstack X}_{C},\mbb Q_p)$ is crystalline and $D_\crys(H^i(\widehat{\mstack X}_{C},\mbb Q_p))=H^i_\crys(\mstack X_k/W(k))[\frac{1}{p}]$.
\end{proof}

\begin{rem}\label{rem:rational p-adic Hodge theory for Hodge proper stacks}
	More generally, using \cite[Proposition 4.3.35 and Theorem 4.3.38]{KubrakPrikhodko_pHodge}, from \Cref{thm:main theorem} it follows that for a Hodge-proper stack over $\mc O_K$ one has all the comparisons that one usually has in $p$-adic Hodge theory. Namely, for any $n\ge 0$
	one has a $(\phi,G_K)$-equivariant isomorphism 
	$$
	H^n_\et({\mstack X}_{{C}}, \mathbb Q_p)\otimes_{\mbb Q_p} B_\crys \simeq H^n_{\crys}({\mstack X}_k/W(k))[\tfrac{1}{p}]\otimes_{W(k)[\frac{1}{p}]} B_\crys,
	$$
	a filtered $G_K$-equivariant isomorphism 
	$$
	H^n_\et({\mstack X}_{{C}}, \mathbb Q_p)\otimes_{\mbb Q_p} B_\dR \xymatrix{\ar[r]^\sim &} H^n_\dR(\mstack X/K)\otimes_K B_\dR
	$$
	and the Hodge-Tate decomposition 
	$$
	H^n_{\et}({\mstack X}_{{C}},\mbb Q_p)\otimes_{\mbb Q_p}{C} \simeq \bigoplus_{i+j=n} H^j({\mstack X}_{K},\wedge^i\mbb L_{\mstack \mstack X/K})\otimes_{K}{C}(-i),
	$$
	where ${C}(-i)$ denotes the $-i$-th Tate twist. 
\end{rem}

\begin{rem}
	Note that from the proof of \Cref{thm:main theorem} it follows that all arrows in  \eqref{eq_main_diagram} except the right vertical one are in fact equivalences. As a result, it is also true for the right vertical one. Using the fact that all maps in the diagram are in fact filtered and $G_K$-equivariant, one can obtain from this another proof of the de Rham comparison for $H^n_\et(\widehat{\mstack X}_{C},\mbb Q_p)$ (for the previous one using prismatic cohomology see \cite[Proposition 4.3.35]{KubrakPrikhodko_pHodge}). However, it is not clear to us how to see from the same diagram that this representation is in fact crystalline.
\end{rem}

\begin{rem}
 In fact one could produce another proof of \Cref{thm:main theorem} by indentifying the right vertical map in \Cref{eq_main_diagram} with the comparison map from \cite{KubrakPrikhodko_pHodge} which was constructed via prismatic cohomology. By \cite[Theorem 4.3.25]{KubrakPrikhodko_pHodge}, in Hodge-proper case the latter map is an equivalence. Consequently, for Hodge-proper $\mstack X$ all arrows in \Cref{eq_main_diagram} except the top horizontal one are equivalences and thus so is the latter.
\end{rem}

\subsection{Example: the complement to a complete intersection in $\mbb P^n$.} \label{ssec:d-de Rham vs d-Hodge}	Here we would to illustrate a difference between $d$-de Rham properness and $d$-Hodge-properness on a particular example. 

Namely, let  $Z\hookrightarrow \mbb P^n_{\mc O_K}$ be a complete intersection given by $f_1=f_2=\ldots= f_d=0$, where each $f_i$ is a homogeneous polynomial of degree $d_i$. We will also assume that $Z$ is flat over $\mc O_K$. 
Let $\mstack X\coloneqq \mbb P^n_{\mc O_K}\!\backslash Z$. By \Cref{lem:example of d-Hodge-proper}, $\mstack X$ is $(d-2)$-Hodge-proper over $\mc O_K$. We claim that

\begin{itemize}
	\item $\mstack X$ is not $(d-1)$-Hodge-proper,
	\item $\mstack X$ is still $(d-1)$-de Rham proper.
\end{itemize}

Denote by $j\colon \mstack X \ra \mbb P^n_{\mc O_K}$ the natural embedding. Also, let $\mc I_Z\subset \mc O_{\mbb P^n}$ be the sheaf of ideals defining $Z$.  It is generated by the images of maps $\alpha_{i}\colon \mc O_{\mbb P^n}(-d_i)\ra \mc O_{\mbb P^n}$ given by $f_i\in H^0(\mbb P^n, \mc O_{\mbb P^n}(d_i))$. Since $Z$ is a complete intersection, the sheaf $\mc O_Z\coloneqq \mc O_{\mbb P^n}/\mc I_Z$ has a free Koszul-type resolution, namely 
$$
{\mr{Kos}_{\mbb P^n}^*}(\alpha_1,\ldots,\alpha_d)\coloneqq \otimes_{i=1}^d (\mc O_{\mbb P^n}(-d_i)\xra{\alpha_{i}} \mc O_{\mbb P^n})\simeq \mc O_Z,
$$ where $\mc O_{\mbb P^n}(-d_i)\xra{\alpha_{i}} \mc O_{\mbb P^n}$ is considered as a two-term complex concentrated in cohomological degrees 0 and -1. For a $d$-tuple $\ul s\coloneqq (s_1,\ldots,s_d)\in \mbb N^d$ denote by $Z^{[\ul s]}\subset \mbb P^n_{\mc O_K}$ the subscheme defined by $f_1^{s_1}=f_2^{s_2}=\ldots=f_d^{s_d}=0$. The subscheme $Z^{[\ul s]}$ is still a complete intersection, and we have 
$$
{\mr{Kos}}_{\mbb P^n}^*(\alpha_1^{s_1},\ldots,\alpha_d^{s_d})\simeq \mc O_{Z^{[\ul s]}},
$$
where $\alpha_{i}^{s_i}\colon \mc O_{\mbb P^n}(-s_id_i)\ra \mc O_{\mbb P^n}$ is the map induced by $f_i^{s_i}$.
Sheaves $\mc O_{Z^{[\ul s]}}$, as well as the Koszul complexes\footnote{Here, the map ${\mr{Kos}}_{\mbb P^n}^*(\alpha_1^{s_1},\ldots,\alpha_i^{s_i},\ldots, \alpha_d^{s_d})\ra {\mr{Kos}}_{\mbb P^n}^*(\alpha_1^{s_1},\ldots,\alpha_i^{s_i-1},\ldots, \alpha_d^{s_d})$ is induced by the commutative square
$$
\xymatrix{\mc O_{\mbb P^n}(-s_id_i)\ar[r]^(.6){\alpha_i^{s_i}}\ar[d]_{\alpha_i}& \mc O_{\mbb P^n}\ar[d]^{\id}\\
\mc O_{\mbb P^n}(-(s_i-1)d_i)\ar[r]^(.65){\alpha_i^{s_i-1}}& \mc O_{\mbb P^n}}
$$
with left vertical arrow induced by $f_i\in H^0(\mbb P^n,\mc O_{\mbb P^n}(-d_i)),$ and identity on the other components of the tensor product.} ${\mr{Kos}}_{\mbb P^n}^*(\alpha_1^{s_1},\ldots,\alpha_i^{s_i},\ldots, \alpha_d^{s_d})$ naturally form an inverse system (via the termwise partial ordering on $\mbb N^d$). Also note that one has the following formular for the dual complex: 
$$
{\mr{Kos}}_{\mbb P^n}^*(\alpha_1^{s_1},\ldots, \alpha_d^{s_d})^{\vee}\simeq \mr{Kos}_{\mbb P^n}^*(\alpha_1^{s_1},\ldots, \alpha_d^{s_d})(s_1d_1+
\ldots +s_dd_d)[-d]
$$
(this follows from the isomorphism of complexes $(\mc O_{\mbb P^n}(-s_id_i)\ra \mc O_{\mbb P^n})^\vee \simeq (\mc O_{\mbb P^n}(-s_id_i)\ra \mc O_{\mbb P^n})(s_id_i)[-1]$).
	
One has a fiber sequence 
$$
\ul{\RG}_Z(\mc O_{\mbb P^n})\ra \mc O_{\mbb P^n} \ra Rj_*\mc O_{\mstack X}, 
$$
where $\ul{\RG}_Z(\mc O_{\mbb P^n})$ is the (sheaf of) local cohomology. We get that $H^i(\mstack X,\mc O_{\mstack X})$ is finitely-generated if and only if $H^{i-1}(\mbb P^n,\ul{\RG}_Z(\mc O_{\mbb P^n}))$ is. Thus to show that $\mstack X$ is not $(d-1)$-Hodge-proper it is enough to show that $H^d(\mbb P^n,\ul{\RG}_Z(\mc O_{\mbb P^n}))$ is not finitely generated.

The local cohomology $\ul{\RG}_Z(\mc O_{\mbb P^n})$ is given explicitly by 
$$\ul{\RG}_Z(\mc O_{\mbb P^n})\simeq \colim_{\ul{s}\in \mbb N^d}\mr{RHom}_{\mc O_{\mbb P^n}}(\mc O_{Z^{[\ul s]}},\mc O_{\mbb P^n})
$$ (see e.g. \cite[Tag 0956]{StacksProject}). Naturally, one can use Koszul resolutions to compute $\mr{RHom}$: one sees that
$$
\mr{RHom}_{\mc O_{\mbb P^n}}(\mc O_{Z^{[\ul s]}},\mc O_{\mbb P^n})\simeq 
{\mr{Kos}}_{\mbb P^n}^*(\alpha_1^{s_1},\ldots, \alpha_d^{s_d})^{\vee}\simeq \mc O_{Z^{[\ul s]}}(s_1d_1+\ldots +s_dd_d)[-d].
$$
This way 
$$
\ul{\RG}_Z(\mc O_{\mbb P^n})\simeq \colim_{\ul{s}\in \mbb N^d} \mc O_{Z^{[\ul s]}}(s_1d_1+\ldots +s_dd_d)[-d].
$$
In particular, $\tau^{<d} \ul{\RG}_Z(\mc O_{\mbb P^n})\simeq 0$

Note that $\mc O_{Z^{[\ul s]}}$ has a finite filtration\footnote{The subsheaf corresponding to a given $\ul t$ is defined by the ideal spanned by $f_1^{\ell_1}\cdot \ldots\cdot f_d^{\ell_d}$ with $\ul t\le \ul \ell\le s$.} by a poset $P_{\ul s}^<\coloneqq \{\ul t\in \mbb N^d_{\ge 0}\ | \ \ul t< \ul s\}$ with the associated graded given by $\oplus_{\ul t\in P_{\ul s}} \mc O_Z(-t_1d_1-\ldots -t_dd_d)$. For the twist $\mc O_{Z^{[\ul s]}}(s_1d_1+\ldots +s_dd_d)$ this gives a filtration with the associated graded $\oplus_{\ul t\in P_{\ul s}} \mc O_Z((t_1+1)d_1+\ldots +(t_d+1)d_d)$. Moreover, for $\ul s'>\ul s$ the corresponding map on the associated graded 
$$
\oplus_{\ul t\in P_{\ul s}} \mc O_Z((t_1+1)d_1+\ldots +(t_d+1)d_d) \arr \oplus_{\ul t\in P_{\ul s'}} \mc O_Z((t_1+1)d_1+\ldots +(t_d+1)d_d)
$$
is the embedding induced by $P_{\ul s}\hookrightarrow P_{\ul s'}$. This way one sees that for the cokernel $K_{\ul s,\ul s'}$ of the above map has an associated graded given by direct sum of line bundles $\mc O_Z(\ell)$ with $\ell>\min_i(s_id_i)$; moreover, as $s'\ra \infty$ the number of these line bundles grows to $\infty$ as well. Thus, for a fixed $\ul s\gg 0$ and $\ul s'\ra \infty$ we have $\rk_{\mc O_K}H^0(\mbb P^n, K_{\ul s,\ul s'})\ra \infty $. Using the exact sequence
$$
0\ra H^0(\mc O_{Z^{[\ul s]}}(\textstyle \sum_i s_id_i))\ra H^0(\mc O_{Z^{[\ul s']}}(\textstyle \sum_i s_i'd_i))\ra H^0(K_{\ul s,\ul s'})\ra H^1(\mc O_{Z^{[\ul s]}}(\textstyle \sum_i s_id_i))
$$
one sees that 
$$
H^d(\ul\RG_Z(\mc O_{\mbb P^n}))\simeq H^0(\colim_{\ul s'\in \mbb N^d}\mc O_{Z^{[\ul s']}}(\textstyle \sum_i s_i'd_i)))\simeq \colim_{\ul s'\in \mbb N^d} H^0(\mc O_{Z^{[\ul s']}}(\textstyle \sum_i s_i'd_i)))
$$
has infinite rank over $\mc O_K$, and so is not finitely generated.

We now consider the de Rham cohomology of $\mstack X$. We need to show that $H^{d-1}_\dR(\mc X/\mc O_K)$ is finitely generated over $\mc O_K$. For any $k\ge 0$ by projection formula we have 
\begin{equation*}
	Rj_*\Omega^k_{\mstack X}\simeq Rj_*j^*\Omega^k_{\mbb P^n}\simeq Rj_*\mc O_{\mstack X}\otimes_{\mc O_{\mbb P^n}} \Omega^k_{\mbb P^n}.
\end{equation*}
Since $\tau^{<d}\ul{\RG}_Z(\mc O_{\mbb P^n})=0$ we have that $\tau^{<d-1}Rj_*\mc O_{\mstack X}\simeq \mc O_{\mbb P^n}$ and, consequently $\tau^{<d-1}Rj_*\Omega^k_{\mstack X}\simeq \Omega^k_{\mbb P^n}$. In particular, for $i>0$ and $j<d-1$ we get that $H^{i,j}(\mstack X/\mc O_K)\simeq H^{i,j}(\mbb P^n/\mc O_K)$. This way, on the first page of the Hodge-to-de Rham spectral sequence there is a single Hodge cohomology of total degree $d-1$ that is infinitely-generated, namely $H^{0,d-1}(\mstack X/\mc O_K)\coloneqq H^{d-1}(\mstack X,\mc O_{\mstack X})$. Let $E_\infty^{0,d-1}\subset H^{0,d-1}(\mstack X/\mc O_K)$ be the corresponding term on the infinite page; it would be enough to show that $E_\infty^{0,d-1}$ is finitely generated. However, we identified $H^{0,d-1}(\mstack X/\mc O_K)$ with $H^d(\ul\RG_Z(\mc O_{\mbb P^n}))\simeq \colim_{\ul s'\in \mbb N^d} H^0(\mc O_{Z^{[\ul s']}}(\textstyle \sum_i s_i'd_i))$, and so it is $p$-torsion free since we assumed $Z$ to be flat over $\mc O_K$. Consequently so is  $E_\infty^{0,d-1}$ and thus it will be enough to show that it is finitely generated after inverting $p$. Note that the $p$-localization of the Hodge-to-de Rham spectral sequence or $\mstack X$ exactly gives the one for $\mstack X_K$. Since over a field of characteristic 0 via excision we have $H^i_\dR(\mstack X_K/K)\simeq H^{i}_\dR(\mbb P_K^n/K)$ for $i<2d$, we get that $H^{d-1}_\dR(\mstack X_K/K)$, and $E_\infty^{0,d-1}[\frac{1}{p}]$ in particular, are finite-dimensional. 

\begin{rem}\label{rem:complement to a locally complete intersection}
	In fact, in the proof of $(d-1)$-de Rham properness we never used that the ambient scheme is $\mbb P^n$, only that it is smooth and proper. Indeed, given a codimension $d$ locally complete intersection $Z$ in an $\mc O_K$-scheme $\ol X$ we have $\tau^{< d}\ul \RG_Z(\mc O_{\ol X})\simeq 0$. Moreover, if $Z$ is flat over $\mc O_K$ then so is the $p$-th local cohomology sheaf $\mc H^p_Z(\mc O_{\ol X})$. Then, if $\ol X$ is smooth and proper we already know that $\mstack X\coloneqq \ol X\backslash Z$ is $(d-2)$-Hodge-proper, so it remains to show that $H^{d-1}_\dR(\mstack X/\mc O_K)$ is finitely generated. Here, using the same argument as above, we can reduce to $H^{d-1}_\dR(\mstack X_K/K)$, which is finitely-generated since $H^{i}_\dR(\mstack X_K/K)\simeq H^{i}_\dR(\ol X_K/K)$ for $i<2d$.
\end{rem}

\begin{rem}
	One can also see that the bound on de Rham-properness is sharp; namely in general (e.g. when $Z$ is a point) $\mstack X$ as above will not be $d$-de Rham proper.
\end{rem}
\section{A variant of $p$-adic Hodge theory for $d$-Hodge-proper stacks}\label{sec:p-adic Hodge-theory for d-Hodge proper stacks}

\subsection{Coherence properties of various cohomology}\label{ssec:complete modules miscellany}
Let $R$ be a (classical) Noetherian ring, and let $I=(x_1,\ldots, x_n)\subset R$ be an ideal generated by a regular sequence (so that $R/I\simeq \otimes_{i=1}^n \cofib(R\xra{\cdot x_i}R)$). Assume also that $R$ is $I$-adically complete (in the classical sense). 

\begin{lem}\label{lem:useful_lemma_2}
	Let $M\in\DMod{R}$ be derived $I$-complete and assume that $\tau^{\le d}(M\otimes_R R/I)\in \Coh(R/I)$. Then $\tau^{\le d}M\in \Coh(R)$.
\end{lem}	
\begin{proof}
	By induction on the number of generators we can assume that $I$ is principal and generated by a non-zero divisor $x\in R$. Note that $\tau^{\le d}M$ is also derived $I$-complete; thus, e.g. by \cite[Proposition 1.1.7]{KubrakPrikhodko_pHodge}, it is enough to show that $\tau^{\le d}M\otimes_R R/I\in \Coh(R)$. Since $\tau^{\le d-1}(M\otimes_R R/I)\simeq \tau^{\le d-1}(\tau^{\le d}M\otimes_R R/I)$, which is coherent, this reduces to showing that $H^d(\tau^{\le d}M\otimes_R R/I)\simeq H^d(M)/I$ is a finitely generated $R/I$-module. But by the universal coefficients formula $H^d(M)/I$ embeds into $H^d(M\otimes_R R/I)$ which is finitely generated by the assumption.
\end{proof}

Here comes the first application.

\begin{cor}\label{cor:finiteness of crystalline}
	Let $\mstack X$ be a $d$-de Rham proper stack over a perfect field $k$ of characteristic $p$. Then 
	$$
	\tau^{\le d}(\RG_\crys(\mstack X/W(k)))\in \Coh(W(k)).
	$$
\end{cor}

\begin{proof}
	By de Rham and crystalline comparisons \cite[Propositions 2.3.20 and 2.5.7]{KubrakPrikhodko_pHodge}, we have $$\RG_\crys(\mstack X/W(k))\otimes_{W(k)} k\simeq \RG_\dR(\mstack X/k).$$ The rest then follows from \Cref{lem:useful_lemma_2}, since $\RG_\crys(\mstack X/W(k))$ is derived $p$-complete.
\end{proof}

\begin{rem}\label{rem:condition on crystalline cohomology}
	We note that $M\coloneqq \tau^{\le d+1}\RG_\crys(\mstack X/W(k))[d+1]$ also satisfies the following finiteness condition:
	$$
	\tau^{<0}([M/p])\in \Coh(k).
	$$
\end{rem}

More generally, given a stack $\mstack X$ over $\mc O_K$ one can establish similar properties for the prismatic cohomology (over the Breuil-Kisin prism). Here, let $K$ be a discretely valued complete field extension of $\mbb Q_p$ with a perfect residue field $k$. For a choice of uniformizer $\pi\in \mc O_K$ we have the corresponding Breuil-Kisin prism $(\mf S, E(u))=(W(k)[[u]], E(u))$ with $\mf S/E(u) =\mc O_K$. Recall (\cite[Definition 2.2.1]{KubrakPrikhodko_pHodge}) that for a prestack $\mstack X$ over $\mc O_K$ we have the prismatic cohomology complex $\RG_\Prism(\mstack X/\mf S)\in \DMod{\mf S}$, as well as its twisted version $\RG_{\Prism^{(1)}}(\mstack X/\mf S)\in \DMod{\mf S}$, which, in the case of the Breuil-Kisin prism is also identified with the "Frobenius twist" $\phi_{\mf S}^*\RG_\Prism(\mstack X/\mf S)$. The complexes $\RG_\Prism(\mstack X/\mf S)$, $\RG_{\Prism^{(1)}}(\mstack X/\mf S)$ are derived $(p,u)$-adically complete and one has a natural "Frobenius map"
$$
\phi_\Prism\colon \RG_{\Prism^{(1)}}(\mstack X/\mf S) \tto \RG_{\Prism}(\mstack X/\mf S),
$$ 
which becomes an equivalence after inverting $E(u)$ in the case $\mstack X$ is a smooth Artin stack (\cite[Remark 2.2.16]{KubrakPrikhodko_pHodge}).

\begin{cor}\label{cor:BK prismatic cohomology d-coherent}
	Let $\mstack X$ be a smooth Artin stack over $\mc O_K$. If the reduction $\mstack X_k$ is $d$-de Rham proper over $k$, then both
	$$
	\tau^{\le d}\RG_{\Prism}(\mstack X/\mf S)\in \Coh(\mf S) \quad \text{ and } \quad \tau^{\le d}\RG_{\Prism\fr}(\mstack X/\mf S)\in \Coh(\mf S).
	$$ 
\end{cor}
\begin{proof}
	The map $\phi_{\mf S}\colon \mf S \ra \mf S$ is faithfully flat, so by \Cref{lem:descent for Coh} it is enough to show that $\tau^{\le d}\RG_{\Prism\fr}(\mstack X/\mf S)\in \Coh(\mf S)$. The claim then follows from crystalline comparison $\RG_{\Prism\fr}(\mstack X/\mf S)\otimes_{\mf S}\mf S/u\simeq \RG_\crys(\mstack X_k/W(k))$ (\cite[Remark 2.5.10]{KubrakPrikhodko_pHodge}) and \Cref{cor:finiteness of crystalline} above.
\end{proof}

\begin{rem}\label{rem:prismatic cohomology is a BK-module up to degree d}
	By \Cref{cor:BK prismatic cohomology d-coherent}, we get that if the reduction $\mstack X_k$ is $d$-de Rham proper, the $\mf S$-modules $H^i_\Prism(\mstack X/\mf S)$ are finitely generated if $i\le d$. The isomorphisms
	$$
	\phi_\Prism[\tfrac{1}{E}]\colon \phi_{\mf S}^*H^i_\Prism(\mstack X/\mf S)[\tfrac{1}{E}] \ism H^i_\Prism(\mstack X/\mf S)[\tfrac{1}{E}]
	$$ 
	then endow each of them with a Breuil-Kisin module structure. Just given the existence of this structure, we get that $H^i_\Prism(\mstack X/\mf S)[\frac{1}{p}]$ is a finite free $\mf S[\frac{1}{p}]$-module for $i\le d$ (see e.g. \cite[Proposition 4.3]{BMS1}). In fact, we will see that (under some torsion-free assumptions) the corresponding lattice in a crystalline Galois representation is described by \'etale cohomology of the Raynaud generic fiber (see \Cref{rem:lattice corresponding to prismatic cohomology}). 
\end{rem}

\begin{rem}\label{rem:condition on the cohomology}
	The $d$-de Rham properness of $\mstack X_k$ also implies some important properties of prismatic cohomology in degrees higher than $d$, that will turn out to be rather crucial later. Namely, even though $M\coloneqq \tau^{\ge d+1}\RG_{\Prism\fr}(\mstack X/\mf S)[d+1]$ is not an object of $\Coh(\mf S)$, it is still derived $(p,u)$-complete and satisfies the following finiteness condition: namely, $\tau^{<0}[M/(p,u)]\in \Coh(k)$. Indeed, let $C\coloneqq \RG_{\Prism\fr}(\mstack X/\mf S)[d+1]$; we have $M\simeq \tau^{\ge 0}(C)$. Note that by our assumptions on $\mstack X_k$ and de Rham comparison we have $\tau^{<0}([C/(p,u)])\in \Coh(k)$. We have a fiber sequence 
	$
	\tau^{<0}C \ra C \ra M 
	$
	which gives a fiber sequence 
	$$
	[\tau^{<0}C/(p,u)]\ra [C/(p,u)] \ra [M/(p,u)]
	$$
	We have that $\tau^{<0}C$ is coherent by \Cref{cor:BK prismatic cohomology d-coherent}, and so it follows that $\tau^{<0}([M/(p,u)])\in \Coh(k)$.  
	
	 Consequently, applying $[-/(p,u)]$ to the fiber sequence $H^0(M)[0] \ra M \ra \tau^{> 0}M$ and considering the corresponding long exact sequence of cohomology we get that $H^{-2}([H^0(M)/(p,u)])\simeq H^{-2}([M/(p,u)])$ while there is also an embedding $H^{-1}([H^0(M)/(p,u)])\hookrightarrow H^{-1}([M/(p,u)])$, and so both are finite-dimensional vector spaces over $k$. Recalling the definition of $M$ we get from this that the $(p,u)$-torsion $H^{d+1}_{\Prism\fr}(\mstack X/\mf S)[p][u]$ is a finite dimensional $k$-vector space.
\end{rem}

We record a part of the discussion in \Cref{rem:condition on the cohomology} as \Cref{lem:condition on the cohomology} below for a future reference.
\begin{lem}\label{lem:condition on the cohomology}
	Let $\mstack X$ be a smooth Artin stack over $\mc O_K$ such that the reduction $\mstack X_k$ is $d$-de Rham proper over $k$. Then $M\coloneqq \tau^{\ge d+1}\RG_{\Prism\fr}(\mstack X/\mf S)[d+1]$ satifies 
	$$
	\tau^{<0}([M/(p,u)])\in \Coh(k).
	$$
\end{lem}

Let $R$ be a $p$-bounded $p$-complete ring. The following notation will be convenient:
\begin{defn}
	We define the \textit{$p$-completed Hodge cohomology} 
	$\RG_\Hdg(\widehat{\mstack X}/R)$ of an Artin $R$-stack $\mstack X$ as the derived $p$-completion $\RG_\Hdg({\mstack X}/R)^\wedge_p$. We also denote $$H^{j,i}(\widehat{\mstack X}/R)\coloneqq H^i(\RG(\mstack X,\wedge^j\mbb L_{\mstack X/R})^\wedge_p).$$
	Similarly, if $\mstack X$ is smooth over $R$ we define the \textit{$p$-completed de Rham cohomology}
	$$\RG_\dR(\widehat{\mstack X}/R)\coloneqq \RG_\dR(\mstack X/R)^\wedge_p.$$
\end{defn}

Naturally, we have coherence results for $\RG_\Hdg(\widehat{\mstack X}/R)$ and $\RG_\dR(\widehat{\mstack X}/R)$ implied by $d$-Hodge and $d$-de Rham-properness of the \textit{mod $p$ reduction} $\mstack X_{R/p}$ of $\mstack X$.

\begin{cor} \label{cor:p-completed Hodge and de Rham d-coherent}Let $R$ be a Noetherian $p$-torsion free $p$-complete ring and let $\mstack X$ be a smooth qcqs Artin stack over $R$. 
	\begin{enumerate}[label=(\arabic*)]
		\item If the reduction $\mstack X_{R/p}$ is $d$-Hodge-proper over $R/p$ we have 
		$$\tau^{\le d}\RG_\Hdg(\widehat{\mstack X}/R)\in \Coh(R).$$
		Equivalently, $H^{j,i}(\widehat{\mstack X}/R)$ is a finitely generated $R$-module for any $i+j\le d $.
		\item If the reduction $\mstack X_{R/p}$ is $d$-de Rham-proper over $R/p$ we have 
		$$\tau^{\le d}\RG_\dR(\widehat{\mstack X}/R)\in \Coh(R).$$
	\end{enumerate}
\end{cor}
\begin{proof}
	Since $p$ is a non-zero divisor in $R$, we have $R/p\simeq \cofib(R\xra{\cdot p}R)$. Thus, for any $M\in \DMod{R}$ and the definition of derived $p$ completion we have $M^\wedge_p\otimes_R R/p\simeq M\otimes_R R/p$. Also, $R/p$ is of Tor-amplitude $[-1,0]$ as an $R$-module. This, together with base change for Hodge and de Rham cohomology (\cite[Proposition 1.1.8]{KubrakPrikhodko_HdR}), leads to equivalences 
	$$
	\RG_\Hdg(\widehat{\mstack X}/R)\otimes_R R/p \ism \RG_\Hdg({\mstack X}_{R/p}/(R/p)) \quad  \quad \RG_\dR(\widehat{\mstack X}/R)\otimes_R R/p  \ism \RG_\dR({\mstack X}_{R/p}/(R/p)).
	$$ 
	We then are done by \Cref{lem:useful_lemma_2}.
\end{proof}

We also record a statement analogous to \Cref{cor:BK prismatic cohomology d-coherent} but which applies to more general base prisms. We refer to \cite[Definition 2.2.1]{KubrakPrikhodko_pHodge} for the definitions of prismatic cohomology $\RG_\Prism(\mstack X/A)\in \DMod{A}$, as well as its twisted version $\RG_{\Prism^{(1)}}(\mstack X/A)\in \DMod{A}$. 

\begin{cor}\label{prop:prismatic cohomology for Hodge-proper of degree d}
	Let $(A,I)$ be a transversal\footnote{Recall that by definition this means that $A/I$ is $p$-torsion free.} prism such that the underlying ring $A$ is Noetherian. Let $\mstack X$ be a qcqs smooth Artin stack over $A/I$ such that the reduction $\mstack X_{A/(I,p)}$ is $d$-de Rham-proper over $A/(I,p)$. Then 
	$$
	\tau^{\le d}\RG_{\Prism\fr}(\mstack X/A)\in \Coh(A).
	$$
	If the reduction $\mstack X_{A/(I,p)}$ is $d$-Hodge-proper over $A/(I,p)$, then 
	$$
	\tau^{\le d}\RG_{\Prism}(\mstack X/A)\in \Coh(A).
	$$
\end{cor}
\begin{proof}
	 Note that we have isomorphisms
	 $$\RG_{\Prism\fr}(\mstack X/A)\otimes_A A/I \ism \RG_{L\dR^\wedge_p}(\mstack X/(A/I)) \ism \RG_\dR(\widehat{\mstack X}/(A/I))^\wedge_p.$$ Here the first arrow is given by the de Rham comparison (\cite[Proposition 2.3.20]{KubrakPrikhodko_pHodge}), while the second equivalence holds for any smooth Artin $\mstack X$ (see \cite[Proposition 2.3.18]{KubrakPrikhodko_pHodge}).  Since $(A,I)$ is transversal we can apply \Cref{cor:p-completed Hodge and de Rham d-coherent}(2), getting that $\tau^{\le d}\RG_\dR(\widehat{\mstack X}/(A/I))$. The first part of the proposition then follows from \Cref{lem:useful_lemma_2}.

For the second part recall that the Hodge-Tate complex $\RG_{\Prism/I}(\mstack X/A)\coloneqq \RG_{\Prism}(\mstack X/A)\otimes_A^{\mbb L} A/I$ has a natural conjugate filtration with the associated graded $\gr_i = \RG(\mstack X, \wedge^i\mbb L_{\mstack X/(A/I)})^\wedge_p\{-i\}[-i]$ (see \cite[Section 2.4]{KubrakPrikhodko_pHodge}). Since $\mstack X$ is smooth this filtration is exhaustive \cite[Proposition 2.4.3]{KubrakPrikhodko_pHodge}) and induces a convergent spectral sequence $E^{i,j}=H^{j,i}(\widehat{\mstack X}/(A/I))\{-i\}\Rightarrow H^{i+j}_{\Prism/I}(\mstack X/A)$ (see \cite[Section 4.3.3]{KubrakPrikhodko_pHodge}). Recall that the twist $\{-i\}$ here is just tensoring with the invertible $A/I$-module $((I/I^2)^\vee)^{\otimes i}$ and so doesn't change coherence. It is then enough to show that $H^{j,i}(\widehat{\mstack X}/(A/I))$ is finitely generated for $i+j\le d$, which follows from \Cref{cor:p-completed Hodge and de Rham d-coherent}(1).
\end{proof}

\subsection{Some integral $p$-adic Hodge theory in $d$-de Rham-proper setting}\label{ssec:integral p-adic Hodge theory}
In this section we adapt some results of \cite[Sections 2 and 4]{KubrakPrikhodko_pHodge} about prismatic cohomology of Hodge-proper stacks to the setting of stacks that are only $d$-Hodge (or rather $d$-de Rham) proper.

In this section we will mostly assume that the base prism $(A,I)=(\mf S,E(u))$ is a Breuil-Kisin prism.
The goal is to establish a truncated version of integral $p$-adic Hodge theory in the setting of $d$-de Rham-proper stacks over $\mc O_K$. The main reference to some relevant background material is \cite[Section 4.3]{KubrakPrikhodko_pHodge}.

Let $\pi\in \mc O_K$ be the uniformizer to which the chosen Breuil-Kisin prism $(\mf S,E(u))$ corresponds (so $\pi$ is the image of $u$ under the identification $\mf S/E(u)\xra{\sim} \mc O_K$). A  compatible choice of $p^n$-roots of $\pi$ in $\mc O_C$ gives an element $\pi^\flat\coloneqq (\pi,\pi^{1/p}, \pi^{1/p^2},\ldots)\in \mc O_C^\flat$ in the tilt $\mc O_C^\flat\coloneqq \lim_{x\mapsto x^p}\mc O_C/p$. This then defines a homomorphism $s_{\pi^\flat}\colon \mf S\ra \Ainf\coloneqq W(\mc O_{C}^\flat)$ by sending $u \mapsto [\pi^\flat]$. We can also compose it further with the natural map $\Ainf\ra W(C^\flat)$ induced by $\mc O_C^\flat\ra C^\flat\coloneqq \mc O_C^\flat[\frac{1}{\pi^\flat}]$. Recall that by the \'etale comparison (\cite[Corollary 4.3.3]{KubrakPrikhodko_pHodge}) and base change we have an equivalence 
\begin{equation}\label{eq:etale comparison}
R\Gamma_{\et}(\widehat{\mstack X}_{{C}}, \mathbb Z_p) \ism \left(R\Gamma_{\Prism^{(1)}}(\mstack X/ \mf S)\widehat{\otimes}_{\mf S} W(C^\flat)\right)^{\phi_{\Prism} = 1}
\end{equation}
for any smooth qcqs Artin $\mc O_K$-stack $\mstack X$.

\begin{rem}
	We use the twisted version of the \'etale comparison here for a future convenience. The difference only affects
	the natural map 
	$$
	R\Gamma_{\et}(\widehat{\mstack X}_{{C}}, \mathbb Z_p)\widehat\otimes_{\mbb Z_p}W(C^\flat) \arr R\Gamma_{\Prism^{(1)}}(\mstack X/ \mf S)\widehat{\otimes}_{\mf S} W(C^\flat),
	$$
	namely it becomes twisted by Frobenius on $W(C^\flat)$ (see \cite[Remark 4.3.7]{KubrakPrikhodko_pHodge}; to apply the remark, note that by base change \cite[Proposition 2.2.13]{KubrakPrikhodko_pHodge} $R\Gamma_{\Prism^{(1)}}(\mstack X/ \mf S)\widehat{\otimes}_{\mf S} W(C^\flat)\simeq R\Gamma_{\Prism^{(1)}}(\mstack X_{\mc O_C}/ \Ainf)\widehat{\otimes}_{\Ainf} W(C^\flat)$).
\end{rem}

\begin{prop} \label{prop:around etale comparison} Let $\mstack X$ be a smooth qcqs Artin stack over $\mc O_K$ such that the reduction $\mstack X_k$ is $d$-de Rham-proper. Then
	 \begin{enumerate}[label=(\arabic*)]
		\item There is a natural equivalence $$\tau^{\le d}\RG_{\Prism\fr}(\mstack X/\mf S)\otimes_{\mf S} W(C^\flat)\ism \tau^{\le d}(\RG_{\Prism\fr}(\mstack X/\mf S)\cotimes_{\mf S}W(C^\flat)).$$
		\item The natural map 
		$$
		\tau^{\le d}R\Gamma_{\et}(\widehat{\mstack X}_{{C}}, \mathbb Z_p)\otimes_{\mbb Z_p} W(C^\flat) \arr \tau^{\le d}(R\Gamma_{\Prism\fr}(\mstack X/ \mf S)\widehat{\otimes}_{\mf S} W(C^\flat))
		$$ 
		is also an equivalence. 
		\item Consequently, one has isomorphisms 
		$$
		H^i_\et(\widehat{\mstack X}_C,\mbb Z_p)\otimes_{\mbb Z_p}W(C^\flat) \ism H^i_{\Prism\fr}(\mstack X/\mf S)\otimes_{\mf S}W(C^\flat)
		$$
		for all $i\le d$.
	\end{enumerate}
\end{prop}
\begin{proof} 1. First let us show that the natural map $\tau^{\le d} \RG_{\Prism\fr}(\mstack X/\mf S) \ra \RG_{\Prism\fr}(\mstack X/\mf S)$ induces an equivalence
	$$
	\tau^{\le d} \RG_{\Prism\fr}(\mstack X/\mf S)\cotimes_{\mf S}W(C^\flat) \ism \tau^{\le d} (\RG_{\Prism\fr}(\mstack X/\mf S)\cotimes_{\mf S}W(C^\flat)).
	$$ 
	Denote $M\coloneqq \RG_{\Prism\fr}(\mstack X/\mf S)$. Note that the $p$-completed tensor product functor $-\cotimes_{\mf S}W(C^\flat)$ is $t$-exact up to a right shift by 1 since $W(C^\flat)$ is flat over $\mf S$. We have a fiber sequence 
	$$
	\tau^{\le d} M \cotimes_{\mf S}W(C^\flat)\tto M\cotimes_{\mf S}W(C^\flat) \tto \tau^{>d}M\cotimes_{\mf S}W(C^\flat)
	$$
	and so to get the claim it is just enough to show that $H^d(\tau^{>d}M\cotimes_{\mf S}W(C^\flat))$ is 0. Note that 
	$$H^d(\tau^{>d}M\cotimes_{\mf S}W(C^\flat))\simeq H^d(H^{d+1}(M)[-d-1]\cotimes_{\mf S}W(C^\flat))\simeq H^{-1}(H^{d+1}(M)\cotimes_{\mf S}W(C^\flat)).$$
	By \Cref{rem:condition on the cohomology} $H^{d+1}(M)$ satisfies the conditions of \Cref{cor:completed tensor with W over S}, and so the latter group is 0. Since $\tau^{\le d} \RG_{\Prism\fr}(\mstack X/\mf S)\in \Coh(\mf S)$ by \Cref{cor:BK prismatic cohomology d-coherent}, the tensor product $\tau^{\le d} \RG_{\Prism\fr}(\mstack X/\mf S)\otimes_{\mf S}W(C^\flat)$ is already $p$-complete. We get a natural equivalence 
	$$
	\tau^{\le d}\RG_{\Prism\fr}(\mstack X/\mf S)\otimes_{\mf S} W(C^\flat)\ism \tau^{\le d}(\RG_{\Prism\fr}(\mstack X/\mf S)\cotimes_{\mf S}W(C^\flat)),
	$$
as desired.
	
	2. We keep the notation  $M\coloneqq \RG_{\Prism\fr}(\mstack X/\mf S)$. We have a fiber sequence 
	$$
	(\tau^{\le d} M\cotimes_{\mf S}W(C^\flat))^{\phi_\Prism=1} \ra (M\cotimes_{\mf S}W(C^\flat))^{\phi_\Prism=1} \ra (\tau^{>d} M\cotimes_{\mf S}W(C^\flat))^{\phi_\Prism=1}.
	$$
	As we saw in the proof of 1, $\tau^{>d} M\cotimes_{\mf S}W(C^\flat)\simeq \tau^{>d}(M\cotimes_{\mf S}W(C^\flat))$ is $d$-coconnected, and thus so is $(\tau^{>d} M\cotimes_{\mf S}W(C^\flat))^{\phi_\Prism=1}$. Moreover, $\tau^{\le d} M\cotimes_{\mf S}W(C^\flat)\simeq \tau^{\le d} M\otimes_{\mf S}W(C^\flat)$ is coherent and so $(\tau^{\le d} M\cotimes_{\mf S}W(C^\flat))^{\phi_\Prism=1}$ lies in\footnote{Indeed, we need to show that $\phi_\Prism -1$ is surjective on $H^d(M)\otimes_{\mf S}W(C^\flat)$. This is enough to show modulo $p$, where $H^d(M)\otimes_{\mf S}C^\flat$ is a finite-dimensional vector space over $C^\flat$ and the statement is standard (e.g. see \cite[Expose III, Lemma 3.3]{chambert1998cohomologie}).} $\DMod{\mf S}^{\le d}$. Together, this shows that
	$$
	\tau^{\le d}(M\otimes_{\mf S}W(C^\flat)^{\phi_\Prism=1})\ism(\tau^{\le d} M\otimes_{\mf S}W(C^\flat))^{\phi_\Prism=1}.
	$$ 
	Moreover, by \cite[Lemma 8.5]{Bhatt_LiteBMS} (again crucially using that $\tau^{\le d} M\cotimes_{\mf S}W(C^\flat)$ is coherent) we get that the natural map 
	$$
	\tau^{\le d}(M\otimes_{\mf S}W(C^\flat)^{\phi_\Prism=1})\otimes_{\mbb Z_p}W(C_p^\flat)	\tto \tau^{\le d} M\cotimes_{\mf S}W(C^\flat)$$ is an equivalence. It remains to note that by \Cref{eq:etale comparison} there is a natural equivalence
	$$
	\tau^{\le d}\RG_\et(\widehat{\mstack X}_C,\mbb Z_p)\ism \tau^{\le d}(M\cotimes_{\mf S}W(C^\flat)^{\phi_\Prism=1}).
	$$
	
	3. From 1 and 2 we get an equivalence
	$$
	\tau^{\le d}\RG_\et(\widehat{\mstack X}_C,\mbb Z_p)\otimes_{\mbb Z_p}W(C^\flat) \ism \tau^{\le d} \RG_{\Prism\fr}(\mstack X/\mf S)\otimes_{\mf S}W(C^\flat).
	$$
	The statement of 3 then follows by passing to cohomology.

\end{proof}

We will also need the following lemma which controls the (classical) base change of prismatic cohomology to $\Ainf$.
\begin{lem}\label{lem:comparison of prismatic with Ainf}
	Let $\mstack X$ be a stack over $\mc O_K$ such that its reduction $\mstack X_k$ is $d$-de Rham proper. Then we have a natural isomorphism 
	$$
	\tau^{\le d}\RG_{\Prism\fr}(\mstack X/\mf S)\otimes_{\mf S}\Ainf\simeq  \tau^{\le d}\RG_{\Prism\fr}(\mstack X_{\mc O_C}/\Ainf).
	$$ 
\end{lem}
\begin{proof}
	By base change for prismatic cohomology (\cite[Proposition 2.2.17]{KubrakPrikhodko_pHodge}) we have $$
	\RG_{\Prism\fr}(\mstack X_{\mc O_C}/\Ainf)\simeq \RG_{\Prism\fr}(\mstack X/\mf S)\cotimes_{\mf S}\Ainf.
$$ 
Note that $\Ainf$ is a $(p,u)$-completely free module over $\mf S$. Indeed, picking a basis $\{x_s\}_{s\in S}$ of $\mc O_C^\flat/\pi^\flat$ over $k$ and lifts $\tilde{x}_i\in \Ainf$ (under the projection $\Ainf\ra \mc O_C^\flat/\pi^\flat$) we get a map 
$$
\widehat{\bigoplus}_{s\in S}\mf S \tto \Ainf
$$
which is an isomorphism modulo $(p,u)$, and thus is itself an isomorphism by Nakayama lemma. Moreover, by \Cref{rem:condition on the cohomology} we have that $M\coloneqq \tau^{\ge d+1}\RG_{\Prism\fr}(\mstack X/\mf S)[d+1]$ satisfies the condition of \Cref{prop:completed direct sum preserves t-structure}, and thus $M\cotimes_\mf S\Ainf\in D(\Mod_{\Ainf})^{\ge d+1}$. Since the completed tensor product is right $t$-exact we get from this that 
$$
\tau^{\le d} \RG_{\Prism\fr}(\mstack X_{\mc O_C}/\Ainf)\simeq \tau^{\le d} (\RG_{\Prism\fr}(\mstack X/\mf S)\cotimes_{\mf S}\Ainf) \simeq \tau^{\le d} (\RG_{\Prism\fr}(\mstack X/\mf S))\cotimes_{\mf S}\Ainf.
$$
By the assumption on $\mstack X_k$ we have $\tau^{\le d} (\RG_{\Prism\fr}(\mstack X/\mf S))\in \Coh(\mf S)$, and so the classical tensor product $\tau^{\le d} (\RG_{\Prism\fr}(\mstack X/\mf S))\otimes_{\mf S}\Ainf$ is already derived $(p,u)$-complete. This way we get an isomorphism 
$$
\tau^{\le d} (\RG_{\Prism\fr}(\mstack X/\mf S))\otimes_{\mf S}\Ainf\ism \tau^{\le d} \RG_{\Prism\fr}(\mstack X_{\mc O_C}/\Ainf),
$$
as desired.
\end{proof}
\begin{cor}\label{cor:comparison of etale with Ainf}
	Let $\mstack X$ be a stack over $\mc O_K$ such that its reduction $\mstack X_k$ is $d$-de Rham proper. Then for any $0\le i\le d$ we have a $G_K$-equivariant isomorphism
	$$
	H^i_\et(\widehat{\mstack X}_C,\mbb Z_p)\otimes_{\mbb Z_p}\Ainf[\tfrac{1}{\mu}]\ism H^i_{\Prism\fr}(\mstack X_{\mc O_C}/\Ainf)\otimes_\Ainf \Ainf[\tfrac{1}{\mu}]
	$$
	that extends to a $(G_K,\phi)$-equivariant isomorphism 
	$$
	H^i_\et(\widehat{\mstack X}_C,\mbb Z_p)\otimes_{\mbb Z_p}B_\crys \ism H^i_{\Prism\fr}(\mstack X_{\mc O_C}/\Ainf)\otimes_\Ainf B_\crys.
	$$
\end{cor}
\begin{proof}
 Recall that the map $\mf S\ra \Ainf$ is flat (\cite[Lemma 4.30]{BMS1}). Thus from \Cref{lem:comparison of prismatic with Ainf} we get that for any $0\le i\le d$ there is an isomorphism $H^i_{\Prism\fr}(\mstack X_{\mc O_C}/\Ainf)\simeq H^i_{\Prism\fr}(\mstack X/\mf S)\otimes_{\mf S}\Ainf$. Applying \Cref{prop:around etale comparison}(3) we then get a $\phi$-equivariant isomorphism 
 \begin{equation}\label{eq:comparison of etale with Ainf}
 H^i_\et(\widehat{\mstack X}_C,\mbb Z_p)\otimes_{\mbb Z_p}W(C^\flat)\ism H^i_{\Prism\fr}(\mstack X_{\mc O_C}/\Ainf)\otimes_\Ainf W(C^\flat), 
\end{equation}
 which is in fact induced by the \'etale comparison map 
 $$
 \RG_\et(\widehat{\mstack X}_C,\mbb Z_p) \tto \RG_{\Prism\fr}(\mstack X_{\mc O_C}/\Ainf)\cotimes_{\Ainf} W(C^\flat)
 $$
 and thus is $G_K$-equivariant.
 
  By \Cref{rem:prismatic cohomology is a BK-module up to degree d}, $H^i_{\Prism}(\mstack X/\mf S)$ is a Breuil-Kisin module in the range $0\le i\le d$ and this way $H^i_{\Prism\fr}(\mstack X_{\mc O_C}/\Ainf)\simeq H^i_{\Prism\fr}(\mstack X/\mf S)\otimes_{\mf S}\Ainf$ is a Breuil-Kisin-Fargues module. Then by \cite[Lemma 4.26]{BMS1} the isomorphism in \Cref{eq:comparison of etale with Ainf} restricts to an isomorphism 
  $$
  H^i_\et(\widehat{\mstack X}_C,\mbb Z_p)\otimes_{\mbb Z_p}\Ainf[\tfrac{1}{\mu}]\ism H^i_{\Prism\fr}(\mstack X_{\mc O_C}/\Ainf)\otimes_\Ainf \Ainf[\tfrac{1}{\mu}],
  $$
  which is also $G_K$-equivariant, since $G_K$-action on $W(C^\flat)$ preserves\footnote{Indeed, $G_K$ acts on $\mu\coloneqq [\varepsilon]-1$ through the cyclotomic character and multiplies it by a unit (e.g. see the proof of \cite[Lemma 6.2.15]{KubrakPrikhodko_pHodge}).} $\Ainf[\tfrac{1}{\mu}]\subset W(C^\flat)$. This then automatically extends to a $(G_K,\phi)$-equivariant isomorphism 
  $$
  H^i_\et(\widehat{\mstack X}_C,\mbb Z_p)\otimes_{\mbb Z_p}\Ainf[\tfrac{1}{\mu},\tfrac{1}{\phi(\mu)}, \tfrac{1}{\phi^2(\mu)},\ldots]\ism H^i_{\Prism\fr}(\mstack X_{\mc O_C}/\Ainf)\otimes_\Ainf \Ainf[\tfrac{1}{\mu},\tfrac{1}{\phi(\mu)}, \tfrac{1}{\phi^2(\mu)},\ldots]
  $$
  where localization $\Ainf[\tfrac{1}{\mu},\tfrac{1}{\phi(\mu)}, \tfrac{1}{\phi^2(\mu)},\ldots]$ is the minimal $\phi$-invariant subring of $W(C^\flat)$ containing $\Ainf[\frac{1}{\mu}]$. The $G_K$-equivariant embedding $\Ainf[\frac{1}{\mu}]\ra B_\crys$ extends to a $(G_K,\phi)$-equivariant embedding 
  $$
  \Ainf[\tfrac{1}{\mu},\tfrac{1}{\phi(\mu)}, \tfrac{1}{\phi^2(\mu)},\ldots] \ra B_\crys.
  $$
  Taking base change under this morphism we arrive at $(G_K,\phi)$-equivariant isomorphisms 
 	$$
 H^i_\et(\widehat{\mstack X}_C,\mbb Z_p)\otimes_{\mbb Z_p}B_\crys \ism H^i_{\Prism\fr}(\mstack X_{\mc O_C}/\Ainf)\otimes_\Ainf B_\crys.
 $$
\end{proof}

We can now prove an analogue of \Cref{thm:main theorem} in the $d$-de Rham proper context.
\begin{thm}\label{thm:even mainer theorem}
	Let $\mstack X$ be a $(d+1)$-de Rham proper stack over $\mc O_K$. Then the map $\Upsilon_{\mstack X,\mbb Q_p}\colon \RG_\et(\mstack X_{C}, \mathbb Q_p)\ra \RG_\et(\widehat{\mstack X}_{C}, \mathbb Q_p)$ induces an isomorphism 
	$$
		H^i_\et(\mstack X_{C}, \mathbb Q_p) \simeq H^i_\et(\widehat{\mstack X}_{C}, \mathbb Q_p)  
	$$ for $i\le d$ and an embedding
	$$      
		H^{d+1}_\et(\mstack X_{C}, \mathbb Q_p) \hookrightarrow H^{d+1}_\et(\widehat{\mstack X}_{C}, \mathbb Q_p)   
	$$
	for $i=d+1$.
\end{thm}
\begin{proof}
	With \Cref{prop:Upsilon for d-Hodge-proper} in hand, it is enough to show that $\dim H^i_\et(\mstack X_{C}, \mathbb Q_p)\ge  \dim H^i_\et(\widehat{\mstack X}_{C}, \mathbb Q_p) $ for $i\le d$. By \Cref{prop:general properties of d-Hodge-properness}(2) we have that the reduction $\mstack X_k$ is $d$-de Rham proper. By \Cref{rem:prismatic cohomology is a BK-module up to degree d} we have that $H^i_{\Prism\fr}(\mstack X/\mf S)[\frac{1}{p}]$ is a free $\mf S[\frac{1}{p}]$-module for $i\le d$. Inverting $p$ in \Cref{prop:around etale comparison}(3) and comparing the dimensions of both sides over $W(C^\flat)[\frac{1}{p}]$ we get that $\dim_{\mbb Q_p} H^i_\et(\widehat{\mstack X}_{C}, \mathbb Q_p) = \rk_{\mf S[\frac{1}{p}]}H^i_{\Prism\fr}(\mstack X/\mf S)[\frac{1}{p}]$ for $i\le d$.
	
	Denote $H^i_\dR(\widehat{\mstack X}_K/K)\coloneqq H^i_\dR(\widehat{\mstack X}/\mc O_K)[\frac{1}{p}]$. By de Rham comparison \cite{KubrakPrikhodko_pHodge} we have an equivalence $\RG_\dR(\widehat{\mstack X}/\mc O_K)\simeq [\RG_{\Prism\fr}(\mstack X/\mf S)/E(u)]$. This gives a short exact sequence
	$$
	0\tto H^i_{\Prism\fr}(\mstack X/\mf S)[\tfrac{1}{p}]/E\tto H^i_\dR(\widehat{\mstack X}_K/K)\tto (H^{i+1}_{\Prism\fr}(\mstack X/\mf S)[\tfrac{1}{p}])[E]\tto 0
	$$
	for all $i$. In particular, for $i\le d$ we get that 
	$$\dim_K H^i_\dR(\widehat{\mstack X}_K/K)\ge \rk_{\mf S[\tfrac{1}{p}]}H^i_{\Prism\fr}(\mstack X/\mf S)[\tfrac{1}{p}]= \dim_{\mbb Q_p} H^i_\et(\widehat{\mstack X}_{C}, \mathbb Q_p). $$
	On the other hand, we claim that for $i\le d$ we have an isomorphism $H^i_\dR(\widehat{\mstack X}_K/K)\simeq H^i_\dR({\mstack X}_K/K)$. Indeed, derived $p$-completion is $t$-exact up to a right shift by 1 and so 
	$$\tau^{\le d}\RG_\dR(\widehat{\mstack X}/\mc O_K)\simeq \tau^{\le d}((\tau^{\le d+1}\RG_\dR({\mstack X}/\mc O_K))^\wedge_p)$$ by \Cref{lem:useful lemma}. However, by our assumption $\tau^{\le d+1}\RG_\dR({\mstack X}/\mc O_K)$ is coherent and so automatically $p$-complete. It follows that $\tau^{\le d}\RG_\dR(\widehat{\mstack X}/\mc O_K)\simeq \tau^{\le d}\RG_\dR({\mstack X}/\mc O_K)$, and consequently $H^i_\dR(\widehat{\mstack X}_K/K)\simeq H^i_\dR({\mstack X}_K/K)$ for $i\le d$. Finally, from the left vertical isomorphism in \Cref{prop:commutative diagram} we know that $\dim_{\mbb Q_p}H^i_\et(\mstack X_{C}, \mathbb Q_p) =\dim_K H^i_\dR({\mstack X}_K/K)$ for all $i$. We get that 
	$$
	\dim H^i_\et(\mstack X_{C}, \mathbb Q_p)\ge  \dim H^i_\et(\widehat{\mstack X}_{C}, \mathbb Q_p) 
	$$
	for $i\le d$ as desired.
\end{proof}
\begin{rem}\label{rem:E(u)-torsion in d+1-st prismatic cohomology}
Lest us record a fact with a similar flavor to \Cref{rem:condition on the cohomology}. By de Rham comparison for prismatic cohomology we have a short exact sequence 
$$
0\tto H^i_{\Prism\fr}(\mstack X/\mf S)/E(u)\tto H^i_\dR(\widehat{\mstack X}/\mc O_K)\tto H^{i+1}_{\Prism\fr}(\mstack X/\mf S)[E(u)]\tto 0.
$$
If $\mstack X$ is $(d+1)$-de Rham proper over $\mc O_K$, then $\mstack X_{\mc O_K/p}$ is $d$-de Rham proper over $\mc O_K/p$ and so $H^{d}_\dR(\widehat{\mstack X}/\mc O_K)$ is a finitely-generated $\mc O_K$-module. Consequently, we also get some partial information about $H^{d+1}_{\Prism\fr}(\mstack X/\mf S)$: namely, its $E(u)$-torsion $M\coloneqq H^{d+1}_{\Prism\fr}(\mstack X/\mf S)[E(u)]$ is finitely-generated over $\mc O_K$. However, from the proof of \Cref{thm:even mainer theorem} one can see more: namely $M$ dies after we invert $p$:  $M[\frac{1}{p}]=0$. Indeed, $M[\frac{1}{p}]$ is the last term in the short exact sequence
$$
0\tto H^{d}_{\Prism\fr}(\mstack X/\mf S)[\tfrac{1}{p}]/E(u)\tto H^{d}_\dR(\widehat{\mstack X}_K/K)\tto (H^{d+1}_{\Prism\fr}(\mstack X/\mf S)[\tfrac{1}{p}])[E(u)]\tto 0
$$
and we saw that $\dim_K H^{d}_{\Prism\fr}(\mstack X/\mf S)[\tfrac{1}{p}]/E(u)\simeq \dim_{\mbb Q_p} H^d_\et(\widehat{\mstack X}_{C}, \mathbb Q_p)$, as well as the equalities 
$$
\dim_{\mbb Q_p} H^{d}_\et(\widehat{\mstack X}_{C}, \mathbb Q_p)=\dim_{\mbb Q_p} H^{d}_\et({\mstack X}_{C}, \mathbb Q_p)=\dim_K H^{d}_\dR({\mstack X}_K/K)=\dim_K H^{d}_\dR(\widehat{\mstack X}_K/K).
$$
Thus $\dim_K M[\frac{1}{p}]=0$, meaning that $M[\frac{1}{p}]=0$.
\end{rem}

Let us also record the following analogue of \cite[Proposition 4.3.14]{KubrakPrikhodko_pHodge} for $d$-Hodge-proper stacks. Let $K_0\coloneqq W(k)[\frac{1}{p}]$.
\begin{cor}\label{cor: equality of dimensions of cohomology}
	Let $\mstack X$ be a smooth $(d+1)$-de Rham proper stack over $\mc O_K$. Then for any $0\le i\le d$
	$$
	\dim_{K_0} H^i_\crys(\mstack X_k/W(k))[\tfrac{1}{p}]=\dim_K H^i_\dR(\mstack X_K/K) = \dim_{\mbb Q_p} H^i_\et(\widehat{\mstack X}_C,\mbb Q_p)=\dim_{\mbb Q_p} H^i_\et({\mstack X}_C,\mbb Q_p).
	$$
\end{cor}
\begin{proof}
	All equalities except the first one follow from the proof of \Cref{thm:even mainer theorem}. Also, we showed that $H^i_\dR(\widehat{\mstack X}_K/K)\simeq H^i_\dR(\mstack X_K/K)$ for $i\le d$. Thus the first equality follows from the Berthelot-Ogus isomorphism (\Cref{prop:Berthelot-Ogus de Rham case}, which after restricting to the point $*\in (*)_\proet$ gives an isomorphism $R\Gamma_\crys(\mstack X_k / W(k))\otimes_{W(k)} K \simeq R\Gamma_\dR(\widehat{\mstack X}_K/ K)$). 
\end{proof}

\subsection{Fontaine's $C_\crys$-conjecture for $d$-de Rham proper stacks}\label{ssec: crystalline comparison}
In this subsection we prove that for $i\le d$ the $i$-th rational \'etale cohomology of the generic fiber of a $(d+1)$-de Rham proper stack over $\mc O_K$ is crystalline. We deduce it from a compariosn between truncated prismatic and crystalline cohomology for which we will use the simplified period ring $B_\mmax$.
\begin{construction}[Rings $A_\mmax$ and  $B_\mmax$]\label{constr:A_max and B_max}
The period ring $A_\mmax$ is defined as the $p$-completion $\left(\Ainf[\frac{\xi}{p}]\right)^\wedge_p$. It can be equivalently described as the completion of $\Ainf[\frac{\xi}{p}]=\Ainf[\frac{[p^\flat]}{p}]$ given by power series 
$$
A_\mmax\coloneqq \left\{\sum_{i\in \mbb Z} [x_i]\cdot p^i \ | \ x_i\in \mc O_C^\flat, \text{ s.t. } (\mr{val}_\flat(x_i)+ i)\ge 0 \text{ and } (\mr{val}_\flat(x_i)+i) \ra +\infty \text{ as } i\ra -\infty \right\},
$$
where $\mr{val}_\flat$ is the natural valuation on $\mc O_C^\flat$ satisfying $\mr{val}_\flat([p^\flat])=1$. 
We define the period ring $B_\mmax$ as $A_\mmax[\frac{1}{p}]$.
 Since for $x_i\in \mc O_C^\flat$ one has $\mr{val}_\flat(x_i^p)=p\cdot \mr{val}_\flat(x_i)$, Frobenius on $\mc O_C^\flat$ induces a map $\phi\colon A_\mmax\ra A_\mmax$ which is in fact an embedding and has the image 
$$
\phi(A_\mmax)\simeq \left\{\sum_{i\in \mbb Z} [x_i]\cdot p^i \text{ with } (\mr{val}_\flat(x_i)+ pi)\ge 0 \text{ and } (\mr{val}_\flat(x_i)+pi) \ra +\infty \text{ as } i\ra -\infty \right\}.
$$
We also have natural $G_K$-actions on $A_\mmax$ and $B_\mmax$ induced by the one on $\mc O_C^\flat$. Note that it commutes with $\phi$.

Recall (e.g. see \cite[Section 3.2.2]{Caruso}) that any element $x\in A_\crys$ has a (unique) expression as 
$$
x=\sum_{i\in \mbb Z}[x_i]p^i \text{ with } x_i\in \mc O_C^\flat \text{ such that } (\mr{val}_\flat(x_i)- \nu(i))\ge 0 \text { and } (\mr{val}_\flat(x_i)- \nu(i))\ra +\infty \text{ as } i\ra -\infty,
$$
where $\nu(i)=0$ if $i\ge 0$ and if $i\le 0$ one puts $\nu(i)$ to be the minimal integer $n$ such that $\mr{val}_p(n!)+i\ge 0$. It is not hard to see that for $i\le 0$ one has
$$
-ip \ge \nu(i)\ge -i  
$$
and this way comparing the convergence conditions one gets natural $G_K$-equivariant embeddings 
$$
\ldots \subset \phi(A_\crys)\subset\phi(A_\mmax)\subset A_\crys \subset A_\mmax \quad \text{ and } \quad \ldots \subset \phi(B_\crys)\subset\phi(B_\mmax)\subset B_\crys \subset B_\mmax.
$$
Consequently, we have embeddings $\subset \phi(B_\crys)^{G_K}\subset\phi(B_\mmax)^{G_K}\subset B_\crys^{G_K}\subset B_\mmax^{G_K}$ of $G_K$-invariants. One has $B_\crys^{G_K}\simeq K_0$ and $\phi\colon B_\crys \ra B_\crys$ maps it to itself via Frobenius on $K_0$; in particular $\phi(B_\crys)^{G_K}\simeq B_\crys^{G_K}$, and consequently $\phi(B_\mmax)^{G_K}\simeq K_0 \simeq B_\mmax^{G_K}$. More generally, having a finite-dimensional $G_K$-representation $V$ over $\mbb Q_p$ we get maps 
$$
(V\otimes_{\mbb Q_p}\phi(B_\mmax))^{G_K}\subset (V\otimes_{\mbb Q_p}B_\crys)^{G_K} \subset (V\otimes_{\mbb Q_p}B_\mmax)^{G_K}
$$
and, since $\phi$ induces a $G_K$-equivariant isomorphism $B_\mmax\xra{\sim}\phi(B_\mmax)$, we get that $$
\dim_{K_0}(V\otimes_{\mbb Q_p}\phi(B_\mmax))^{G_K}=\dim_{K_0}(V\otimes_{\mbb Q_p}B_\mmax)^{G_K}\  \Rightarrow \  \dim_{K_0}(V\otimes_{\mbb Q_p}B_\crys)^{G_K} = \dim_{K_0}(V\otimes_{\mbb Q_p}B_\mmax)^{G_K}.
$$
In particular, $V$ crystalline if and only if it is $B_\mmax$-admissible: namely, $$\dim_{K_0}(V\otimes_{\mbb Q_p}B_\mmax)^{G_K}=\dim_{\mbb Q_p}V.$$
	
\end{construction}

\begin{construction}[Ring $\mf S_\mmax$] Let $e$ be the ramification index of $\mf S_\mmax$ We define the ring $\mf S_\mmax$ as the $p$-adic completion 
	$$
	\mf S_\mmax\coloneqq\left(\mathfrak S[\tfrac{u^e}{ p}]\right)_p^\wedge.
	$$
	Picking a uniformizer $\pi\in \mc O_K$ with minimal polynomial $E(u)$ we have $E(u)\equiv u^e \mod p\mf S$ and so one also has $\mf S_\mmax \simeq \mf S[\tfrac{E(u)}{p}]^\wedge_p$. Via the choice of $\pi^\flat\in \mc O_C^\flat$ we have a map $\mf S\ra \Ainf$ sending $u$ to $[\pi^\flat]$. Since $\mr{val}_\flat([\pi^\flat]^e)=1$, this map naturally extends to a map 
	$$
	\mf S_\mmax \ra A_\mmax.
	$$

\end{construction}

\begin{rem}
	Note that $u^e\in p\cdot \mf S_\mmax$. Thus $(p,u)$-adic topology on $\mf S_\mmax$ is equivalent to $p$-adic one. Similarly, $(p,[p^\flat])$-adic topology on $A_\mmax$ is equivalent to $p$-adic one.
\end{rem}

With this notations the key result of this section is the following
\begin{thm}\label{crys_comp_Bmax}
Let $\mstack X$ be a smooth $(d+1)$-de Rham proper stack over $\mathcal O_K$. Then there is a $(G_K,\phi)$-equivariant equivalence
$$\left(\tau^{\le d} R\Gamma_{\Prism\fr}(\mstack X_{\mc O_C}/\Ainf)\right)\otimes_{\Ainf} B_\mmax \simeq \left(\tau^{\le d} R\Gamma_\crys(\mstack X_k/ W(k))\right)\otimes_{W(k)} B_\mmax.$$
\end{thm}

We will prove this theorem in the remaining sections using a bit of condensed mathematics. Before introducing some necessary notation let us deduce from \Cref{crys_comp_Bmax} the crystallinness of the $G_K$-representation given by the \'etale\footnote{Note that from \Cref{thm:even mainer theorem} it follows that in the range $i\le d+1$ it doesn't matter whether we consider the Raynaud or the algebraic generic fiber here.} cohomology $H^i_\et(\widehat{\mstack X}_{\mbb C},\mbb Q_p)$. 

Recall that a finite-dimensional $\mbb Q_p$ representation $V$ is called crystalline if the dimension of $D_\crys(V)\coloneqq (V\otimes_{\mbb Q_p}B_\crys)^{G_K}$ over $K_0$ is equal to the dimension of $V$.

\begin{cor}[Fontaine's $C_\crys$-conjecture for $d$-de Rham proper stacks]\label{cor:fontaine's conjecture for d-de Rham-proper stacks}
Let $\mstack X$ be a smooth $(d+1)$-de Rham proper stack over $\mathcal O_K$. Then $H^i_{\et}(\mstack X_C,\mbb Q_p)\simeq H^i_{\et}(\widehat{\mstack X}_C,\mbb Q_p)$ is a crystalline Galois representation for $i\le d$. Moreover,  $D_\crys(H^i_{\et}(\mstack X_C,\mbb Q_p))\simeq H^i_\crys(\mstack X_k/W(k))[\tfrac{1}{p}]. $
\end{cor}
\begin{proof}
	Recall that for  $0\le i\le d$ we have $H^i_{\Prism\fr}(\mstack X_{\mc O_C}/\Ainf)\simeq H^i_{\Prism\fr}(\mstack X/\mf S)\otimes_{\mf S}\Ainf$ by \Cref{lem:comparison of prismatic with Ainf}, and so by \Cref{rem:prismatic cohomology is a BK-module up to degree d} we have that $H^i_{\Prism\fr}(\mstack X_{\mc O_C}/\Ainf)[\frac{1}{p}]$ are free modules over $\Ainf[\frac{1}{p}]$. Thus applying cohomology to both sides of the isomorphism in \Cref{crys_comp_Bmax} we get $(G_K,\phi)$-equivariant isomorphisms
	$$
	H^i_{\Prism\fr}(\mstack X_{\mc O_C}/\Ainf)\otimes_\Ainf B_\mmax\simeq H^i_\crys(\mstack X_k/W(k))\otimes_{W(k)}B_\crys.
	$$
	Recall that by \Cref{cor:comparison of etale with Ainf} we have a $(G_K,\phi)$-equivariant isomorphisms
	$$
	H^i_{\et}(\widehat{\mstack X}_C,\mbb Q_p)\otimes_{\mbb Q_p}B_\crys\ism H^i_{\Prism\fr}(\mstack X_{\mc O_C}/\Ainf)\otimes_\Ainf B_\crys.
	$$
	Tensoring it further by $B_\mmax$ we then get $(G_K,\phi)$-equivariant isomorphisms
		$$
	H^i_{\et}(\widehat{\mstack X}_C,\mbb Q_p)\otimes_{\mbb Q_p}B_\mmax\ism H^i_{\Prism\fr}(\mstack X_{\mc O_C}/\Ainf)\otimes_\Ainf B_\mmax\ism H^i_\crys(\mstack X_k/W(k))\otimes_{W(k)}B_\mmax.
	$$
	Put $V\coloneqq H^i_{\et}(\widehat{\mstack X}_C,\mbb Q_p)$. By the discussion in the end of \Cref{constr:A_max and B_max} we then have 
	$$
	D_\crys(V)\simeq (V\otimes_{\mbb Q_p}B_\mmax)^{G_K}\simeq (H^i_\crys(\mstack X_k/W(k))\otimes_{W(k)}B_\mmax)^{G_K}\simeq H^i_\crys(\mstack X_k/W(k)\otimes_{W(k)}K_0,
	$$
	since the $G_K$-action on $H^i_\crys(\mstack X_k/W(k))$ is trivial and $B_\mmax^{G_K}\simeq K_0$. Since $H^i_\crys(\mstack X_k/W(k)\otimes_{W(k)}K_0\simeq H^i_\crys(\mstack X_k/W(k)[\frac{1}{p}]$, from \Cref{cor: equality of dimensions of cohomology} we get 
	$$
	\dim_{\mbb Q_p}V=\dim_{K_0}H^i_\crys(\mstack X_k/W(k)[\tfrac{1}{p}]= \dim_{K_0} D_\crys(V),
	$$
	and so $V$ is crystalline. 
\end{proof}
\begin{rem}
	By general theory of admissible rings of Fontaine, from \Cref{cor:Fontaine's conjecture for Hodge-proper stacks} we get that for any $i\le d$ there is a natural $(G_K,\phi)$-equivariant isomorphism 
	$$
	H^i_\et(\mstack X_{C}, \mathbb Q_p)\otimes_{\mbb Q_p}B_\crys \ism H^i_\crys(\mstack X_k/W(k))\otimes_{W(k)} B_\crys.
	$$ 
\end{rem}
\begin{rem}
	By a method similar to one in \cite[Section 4.3.3]{KubrakPrikhodko_pHodge} it should be possible to show that for $i\le d$ the filtered $K$-vector space $D_\dR(H^i_{\et}(\mstack X_C,\mbb Q_p))$ is given by $H^i_\dR(\mstack X_K/K)$ together with Hodge filtration. This then would formally imply the Hodge-Tate decomposition in the same range of degrees.
\end{rem}
\begin{rem}\label{rem:lattice corresponding to prismatic cohomology} Let $\mstack X$ be a smooth $(d+1)$-de Rham-proper stack over $\mc O_K$ and let $$H^i_\et(\widehat{\mstack X}_C,\mbb Z_p)_\free\coloneqq H^i_\et(\widehat{\mstack X}_C,\mbb Z_p)/H^i_\et(\widehat{\mstack X}_C,\mbb Z_p)_{\mr{tors}} \subset H^i_\et(\widehat{\mstack X}_C,\mbb Q_p).$$ By \Cref{cor:fontaine's conjecture for d-de Rham-proper stacks}, for $i\le d$ this is a lattice in a crystalline $G_K$-representation.  Following the same argument as in \cite[Remark 4.3.27]{KubrakPrikhodko_pHodge} one can identify the corresponding Breuil-Kisin module $\BK(H^i_\et(\widehat{\mstack X}_C,\mbb Z_p)_\free)$ with the free Breuil-Kisin module $H^i_{\Prism}(\mstack X/\mf S)_\free$ associated to $H^i_{\Prism}(\mstack X/\mf S)$ (see \cite[Proposition 4.3]{BMS1} for the definition). If $H^i_{\Prism}(\mstack X/\mf S)$ is free over $\mf S$, then by \Cref{prop:around etale comparison} $H^i_\et(\widehat{\mstack X}_C,\mbb Z_p)$ is $p$-torsion free and the statement becomes:
	$$
	\BK(H^i_\et(\widehat{\mstack X}_C,\mbb Z_p))\simeq H^i_{\Prism}(\mstack X/\mf S).
	$$ 
\end{rem}
The following lemma gives a sufficient condition for $H^i_{\Prism}(\mstack X/\mf S)$ to be a free $\mf S$-module.

\begin{lem}
	Let $\mstack X$ be a $(d+1)$-de Rham-proper stack over $\mathcal O_K$. If $i\le d$ and $H^i_\crys(\mstack X_k/W(k))$ is $p$-torsion free, then $H^i_\Prism(\mstack X/\mathfrak S)$ is free as an $\mathfrak S$-module.
	\end{lem}	
\begin{proof}
	In the proof all tensor products are assumed to be non-derived (unless noted otherwise).
	First note that $H^i_\Prism(\mstack X/\mathfrak S)$ is free over $\mf S$ if and only if $H^i_{\Prism^{(1)}}(\mstack X/\mathfrak S)$ is. Let more generally $C$ be a complex of $\mathfrak S$-modules such that $H^i(C)$ is finitely generated over $\mathfrak S$ and $H^i(C)[1/p]$ is free over $\mathfrak S[1/p]$ (e.g. $C=\RG_{\Prism^{(1)}}(\mstack X/\mathfrak S)$). We claim that if $H^i(C\otimes_{\mathfrak S}^{\mbb L} W(k))$ is $p$-torsion free then $H^i(C)$ is free. Indeed, since $H^i(C)\otimes_S W(k)$ is a submodule of $H^i(C\otimes^{\mbb L}_{\mathfrak S} W(k))$ it follows that $H^i(C)\otimes_{\mf S} W(k)$ is $p$-torsion free (equivalently free as a $W(k)$-module) as well. Hence it is enough to prove the following assertion: let $M$ be a finitely generated $\mathfrak S$-module such that
	\begin{itemize}
		\item $M[1/p]$ is free over $\mathfrak S[1/p]$.
		
		\item $M\otimes_{\mathfrak S} W(k)$ is free over $W(k)$.
	\end{itemize}
	Then $M$ is free over $\mathfrak S$. 
	
	But under these assumptions
	\begin{align*}
		\dim_k M\otimes_{\mathfrak S} k & = \rank_{W(k)} M\otimes_{\mathfrak S} W(k) = \rank_{W(k)[1/p]} M\otimes_{\mathfrak S} W(k)[\tfrac{1}{p}] = \\
		& = \rank_{\mathfrak S[1/p]} M\otimes_{\mathfrak S} \mathfrak S[\tfrac 1 p] = \dim_{\Frac \mathfrak S} M\otimes_{\mathfrak S} \Frac \mathfrak S.
	\end{align*}
	and thus $M$ is free by the semicontinuity of stalks.
	
	\end{proof}

\paragraph{Condensed prismatic and crystalline cohomology.}
First, we set up some notation. We denote by $X\mapsto \ul X$ the natural functor $\Top \ra \Cond$. It commutes with products and so sends topological groups/rings to group/ring objects in condensed sets. We consider the full subcategory $\Solid \subset \Cond(\Ab)$ of solid abelian groups; recall that it is closed under all limits and colimits. The restriction of the functor $A\ra \ul A$ to discrete topological groups factors through $\Solid$, inducing a fully faithful embedding $\Ab\ra \Solid$. It then also extends to a fully faithful embedding $D(\Ab)\ra D(\Solid)$ between derived categories which by slight abuse of notation we will still denote by $M\mapsto \ul{M}$. Having an $\mbb E_{\infty}$-ring $A$ in $D(\Solid)$ we will sometimes denote by $D(A)\simeq\Mod_A(D(\Solid))$ the category of $A$-modules. It is endowed with the symmetric monoidal structure given by solid tensor tensor product $-\sotimes_A-$.

For a set $\{f_i\}_{i\in I}$ of elements of $R$ and an object $M\in D(R)$ we denote 
$$
\Kos(M;\{f_i\}_{i\in I})\coloneqq M\otimes_{\mbb Z[x_i]_{i\in I}} \mbb Z
$$
where $M$ is considered as a $\mbb Z[x_i]_{i\in I}$-module via the homomorphism $\mbb Z[x_i]_{i\in I}\ra R$ sending $x_i$ to $f_i$ and all $x_i$'s act on $\mbb Z$ by 0. Note that $\Kos(M;\{f_i\}_{i\in I})$ is naturally a module over the derived ring $\Kos(R;\{f_i\}_{i\in I})$. In the case $M\in \Mod_R$, the complex $\Kos(M;\{f_i\}_{i\in I})$ can be explicitly computed by the "Koszul complex"
$$
\ldots \tto \oplus_{i<j} M\tto \oplus_i M \xymatrix{\ar[r]^{(f_i\cdot )_{i\in I}}&} M
$$
where we fix some auxiliary total ordering on $I$ (see \cite[Tag 0621]{StacksProject}). Recall that given a finitely generated ideal $J=(x_1,\ldots, x_s)\subset R$ one defines a functor $M\mapsto M^\wedge_J$ as
\begin{equation}\label{eq:derived completion}
M^\wedge_J\coloneqq \prolim[n] {\Kos(M;x_1^n,\ldots,x_s^n)}.
\end{equation}
The functor $(-)^\wedge_J$ is a left adjoint to the embedding $D(\Mod_R)_{J-\comp}\ra D(\Mod_R)$ (see \cite[Tag 091N]{StacksProject}) and so computes the "derived $J$-completion". In particular, it doesn't depend on the choice of generators $(x_1,\ldots,x_s)$.

\begin{construction}\label{constr:blacktriangle}\begin{enumerate}
		\item 	Let $R$ be a classical ring with a finitely generated ideal $J=(x_1,\ldots,x_s)\subset R$ (with a fixed set of generators). We define $$
		R^\blacktriangle_J\coloneqq \prolim[n] \ul{\Kos(R;x_1^n,\ldots,x_s^n)}\in D(\Solid),
		$$ to be the derived $J$-completion of the discrete ring $\ul R$. The solid group $R^\blacktriangle_J$ has the natural structure of a derived solid ring. In the case $(x_1,\ldots,x_s)$ is a regular sequence, $\Kos(R;x_1^n,\ldots,x_s^n)\simeq R/(x_1^n,\ldots,x_s^n)$ and $R^\blacktriangle_J\simeq \ul {R^\wedge_J}\in \Solid$ where $R^\wedge_J$ is considered as a topological ring via $J$-adic topology.
		\item Let $M\in D(\Mod_R)$ be a complex. We define 
		$$
		M^\blacktriangle_J\coloneqq \prolim[n]\ul{\Kos(M;x_1^n,\ldots,x_s^n)}
		$$ 
		where $\ul{\Kos(M;x_1^n,\ldots,x_s^n)}$ is the image of $\Kos(M;x_1^n,\ldots,x_s^n)\in D(\Mod_R)$ under the natural functor $D(\Ab)\ra D(\Solid)$. The functor $M\ra M^\blacktriangle_J$ defines a functor $$(-)^\blacktriangle_J\colon D(\Mod_R)\ra D(R^\blacktriangle_J).$$
		By construction it factors through the subcategory $D(R^\blacktriangle_J)_{J-\comp}\subset D(R^\blacktriangle_J)$ of derived $J$-complete solid $R^\blacktriangle_J$-modules.  
		\item Let $(R,J)$ be as in part 1 and let $A\simeq {R[\frac{1}{r_s}]}_{s\in S}$ be a localization of $R$ with respect to a set of elements $\{r_s\}_{s\in S}$. Then we define
		$$
		A^\blacktriangle_J\coloneqq {R^\blacktriangle_J[\tfrac{1}{r_s}]}_{s\in S}.
		$$
	\item When it is clear from the context what $J$ is, we will sometimes omit it to lighten the notation. We will occasionally call the operation $(-)^\bt\coloneqq (-)^\bt_J$ \textit{solid $J$-completion}.
	\end{enumerate}
\begin{rem}\label{rem:solid completion on the point}
	We note that since evaluation on the singleton $*\in (*)_\proet$ commutes will all limits and colimits, one has isomorphisms $R^\bt_J(*)\simeq R^\wedge_J$ and $M^\bt_J(*)\simeq M^\wedge_J$ where $(-)^\wedge_J$ is the derived $J$-completion. Similarly, $A^\blacktriangle_J(*)\simeq {R^\wedge_J[\tfrac{1}{r_s}]}_{s\in S}$.
\end{rem}

\begin{rem}
	Note that the solid module $M^\blacktriangle_J$ only depends on the pro-system $\{[M/(x_1^n,\ldots,x_s^n)]\}_n$ and not $M$ itself. Consequently\footnote{Since prosystems $\{[M/(x_1^n,\ldots,x_s^n)]\}_n$ and $\{[M^\wedge_J/(x_1^n,\ldots,x_s^n)]\}_n$ are naturally equivalent.}, the functor $(-)^\bt_J\colon D(\Mod_{R})\ra D(R^\bt_J)_{J-\comp}$ factors through the usual derived completion $(-)^\wedge_{J}\colon D(\Mod_{R}) \ra D(\Mod_{R})_{J-\comp}$.
\end{rem}

\begin{lem}\label{lem:solid completion is right t-exact}
	The functor $(-)^\bt_J\colon D(\Mod_R)\ra D(R^\blacktriangle_J)_{J-\comp}$ is right $t$-exact. If $J=(x_1,\ldots,x_s)\subset R$ is generated by $s$ elements it is also left $t$-exact up to a shift by $s$.
\end{lem}
\begin{proof}
	Indeed, $(-)^\bt_J$ is the composition of fully faithful embedding $\ul{(-)}\colon D(\Mod_R)\ra D(\ul(R))$ (where $R$ is endowed with discrete topology) which is $t$-exact and the derived $J$-adic completion in $D(\Mod_R)$ which is right $t$-exact. Indeed, $M^\wedge_J\simeq (((M^\wedge_{x_1})^\wedge_{x_2})\ldots)^{\wedge}_{x_s}$. Now, having $M\in \Mod_{\ul R}(\Solid)\subset D(\ul(R)$ and $x\in R$ the complex $[M/x]$ has two cohomology modules: $H^0([M/x])=M/x$ and $H^1([M/x])=M[x]$ ($x$-torsion). Since in the derived limit $\lim_{n\in \mbb N}$ (applied to objects in the heart) we only have two terms $\lim^0$ and $\lim^1$, we get the only non-zero term in $H^{>0}(M^\wedge_x)$ could be given by $\lim^1_nM/x^n$, which is zero by Mittag-Leffler since maps $M/x^n\ra M/x^{n-1}$ are surjective.
	
	For the second statement note that the functor $M\mapsto \ul{[M/(x_1^n,\ldots,x_s^n)]}$ is left $t$-exact up to a shift by $s$, while the homotopy limit $\prolim[n]$ is left $t$-exact.
\end{proof}
\begin{warn}
We warn the reader that the functor	$(-)^\bt_J$ doesn't need to be left $t$-exact even if we restrict to the subcategory of derived $J$-complete modules $D(\Mod_R)_{J-\comp} \subset D(\Mod_R)$. Nevertheless, it still satisfies some "left $t$-exactness" properties if we restrict to objects satisfying special finiteness conditions (see \Cref{lem_discretness_of_cond_completion}).
\end{warn}	
 
\end{construction}
We now give a variant of the prismatic and crystalline cohomology in this setting.
\begin{defn}[Solid $(p,I)$-completed prismatic cohomology]
Let $(A, I)$ be a bounded prism. Then we can consider the corresponding solid ring $A^\blacktriangle\coloneqq A^\blacktriangle_{(p,I)}$. We define functors
$$R\Gamma_{\Prism}^\blacktriangle(- / A), R\Gamma_{\Prism^{(1)}}^\blacktriangle(- / A) \colon \PStk_{A/I}^\op \tto \Mod_{A^\blacktriangle}(D(\Solid))$$
as the application of \Cref{constr:blacktriangle}(2) to functors 
$$R\Gamma_{\Prism}(- / A), R\Gamma_{\Prism^{(1)}}(- / A) \colon \PStk_{A/I}^\op \tto \Mod_{A}(D(\Ab))$$
(see \cite[Definition 2.2.1]{KubrakPrikhodko_pHodge}). Namely,
$$R\Gamma_{\Prism}^\blacktriangle(\mstack X/ A) \coloneqq R\Gamma_{\Prism}(\mstack X/ A)^\blacktriangle_{(p,I)}, \qquad R\Gamma_{\Prism^{(1)}}^\blacktriangle(\mstack X/ A) \coloneqq R\Gamma_{\Prism\fr}(\mstack X/ A)^\blacktriangle_{(p,I)}.$$
\end{defn}

\begin{rem}
	Since $A$ is bounded, the derived $(p,I)$-completion of $A$ is identified with the classical one. So in this case $A^\blacktriangle$ in fact lands in $\Solid\subset D(\Solid)$, and is the solid ring $\ul A$ represented by $A$ with $(p,I)$-adic topology. 
\end{rem}

\begin{defn}[Solid PD-completed crystalline cohomology]\label{def_cond_crys}
Let $S \inj T$ be a PD-thickening, where $T=\colim T_n$ is a formal scheme with $T_n$ being the $n$-th PD-neighborhood of $S$. Let $A \coloneqq \mc O(T)\coloneqq \lim_n \mc O(T_n)$ be the global functions on $T$ and $A^\blacktriangle\coloneqq \lim_n^0 \ul{\mc O(T_n)}\in  \Solid$ the corresponding completion in solid rings, where $\mc O(T_n)$ is considered as discrete toplogical group.  We define the functor
$$R\Gamma_\crys^\blacktriangle(-/T) \colon \PStk_{S}^{\op} \tto \Mod_{A^\blacktriangle}(D(\Solid))$$
as the limit in $D(\Solid)$
$$R\Gamma_\crys^\bt(-/T) \coloneqq  \prolim[n] \ul{R\Gamma_\crys(-/T_n)},$$
where $R\Gamma_\crys(-/T_n)$ is the right Kan extension of the crystalline cohomology functor from smooth $S$-schemes to prestacks. 
\end{defn}
\begin{rem}
In the following we will only consider a special case when $T=\Spf A$ is a $p$-adic formal scheme. In this case there is a natural equivalence $R\Gamma_\crys^\bt(\mstack X/T) \simeq (R\Gamma_\crys(\mstack X/T))_p^\bt\in D(\Mod_{A^\bt_p}(\Solid))$.
\end{rem}

With this definition one can formally deduce the base change and the comparison between condensation of the prismatic and crystalline cohomology from the "discrete" version:
\begin{prop}[Base change]\label{prop:case change for prismatic cohomology} Let $(A,I) \to (B, IB)$ be a morphism of bounded prisms such that the $(p,IB)$-completed tensor product functor $- \cotimes_A B$ is left $t$-exact up to a shift. Also assume that $\mstack X$ is quasi-compact quasi-separated and syntomic Artin stack. Then the natural maps
$$
R\Gamma_{\Prism}^\bt(\mstack X/A)\sotimes_{A^\bt} B^\bt\tto R\Gamma_{\Prism}^\bt(\mstack X_{B/IB}/B) \qquad \ R\Gamma_{\Prism^{(1)}}^\bt(\mstack X/A)\sotimes_{A^\bt} B^\bt \tto R\Gamma_{\Prism^{(1)}}^\bt(\mstack X_{B/IB}/B)
$$
are equivalences.

\begin{proof}
By \Cref{sotimes_vs_cotimes}, the solid tensor product $R\Gamma_{\Prism}^\bt(\mstack X/A)\sotimes_{A^\bt} B^\bt $ is naturally isomorphic to the solid $(p,I)$-completion $(R\Gamma_{\Prism}(\mstack X/A)\cotimes_A B)^\bt_{(p,I)}$, where the tensor product inside the brackets is also $(p,I)$-completed. By  \cite[Proposition 2.2.17]{KubrakPrikhodko_pHodge}, $R\Gamma_{\Prism}(\mstack X/A)\cotimes_A B\simeq R\Gamma_{\Prism}(\mstack X_{B/IB}/B)$ and we get the comparison. The proof for twisted cohomology is completely analogous, using the same reference.
\end{proof}
\end{prop}

\begin{prop}[Crystalline comparison]\label{prop:crystalline comparison}
	 If $\mstack X$ is smooth and $I = (p)$ then there is a natural equivalence
	$$R\Gamma_{\Prism^{(1)}}^\bt(\mstack X/A)\simeq R\Gamma_\crys^\bt(\mstack X/A).$$
\end{prop}
\begin{proof}
	This follows from the crystalline comparison $R\Gamma_{\Prism^{(1)}}(\mstack X/A)\simeq R\Gamma_\crys(\mstack X/A)^\wedge_p$ (\cite[Proposition 2.5.7]{KubrakPrikhodko_pHodge}) by applying $(-)^\bt$.
\end{proof}

\begin{rem}
One can also "topologize" the \'etale cohomology of a stack and (formally) upgrade the \'etale comparison of prismatic and \'etale cohomology to condensed setting, but we won't need and so don't discuss this here.
\end{rem}

\begin{notation}
Below, we will consider instances of \Cref{constr:blacktriangle} in the following situations:
\begin{itemize}
	\item $W(k)$,  $\mf S_\mmax$, $A_\crys$, $A_\mmax$ with $J=(p)$, and the corresponding $p$-localizations $K_0$,  $\mf S_\mmax[\frac{1}{p}]$, $B_\crys$, $B_\mmax$. We will denote the corresponding solid rings simply by $W(k)^\bt$,  $\mf S_\mmax^\bt$, $A_\crys^\bt$, $A_\mmax^\bt$ and $K_0^\bt$, $\mf S_\mmax^\bt[\frac{1}{p}]$, $B_\crys^\bt$, $B_\mmax^\bt$.
	\item $\mf S$, $\mf S\langle T\rangle\coloneqq \mf S[T]^\wedge_{(p,u)}$, $\Ainf$ with $J=(p,u)$ and $(p,[\pi^\flat])$ correspondingly. Similarly, we denote the corresponding solid rings by $\mf S^\bt$, $\mf S\langle T\rangle^\bt$, $A_\innf^\bt$.
\end{itemize}
\end{notation}
Below, we will also need the following remarks:
\begin{rem}[$-\sotimes_{\mf S^\bt}\mf S\langle T\rangle^\bt$ as completed direct sum]\label{rem:solid tensor product with Tate algebra}
	Note that by construction, the underlying $\mf S^\bt$-module of $\mf S\langle T\rangle^\bt$ is given by the $(p,u)$-completed direct sum $\widehat\oplus_{i\in \mbb N} \mf S^\bt\cdot T^i$. So, for any derived $(p,u)$-complete $\mf S^\bt$-module $M\in D(\mf S^\bt)_{(p,u)-\comp}$ one has 
	$$
	M\sotimes_{\mf S^\bt}\mf S\langle T\rangle^\bt\simeq \widehat\bigoplus_{i\in \mbb N} M \cdot T^i
	$$
	as $\mf S^\bt$-modules. Indeed, there is a natural map $\oplus_{i\in \mbb N} M \cdot T^i\ra M\sotimes_{\mf S^\bt}\mf S\langle T\rangle^\bt$, then by \Cref{sotimes_vs_cotimes} the right hand side is $(p,u)$-complete and so the map factors through the completion $\widehat\bigoplus_{i\in \mbb N} M \cdot T^i$. Moreover it is an isomorphism modulo $(p,I)$  since tensor product commutes with direct sums.
\end{rem}

\begin{rem}\label{rem:ses for S_max}
	 Also, one has a strict exact sequence 
$$
0\tto \mf S\langle T\rangle \xymatrix{\ar[r]^{\cdot (pT-u^e)}&}  \mf S\langle T\rangle \tto \mf S_\mmax \tto 0
$$ 
of topological $\mf S$-modules where the topology is $(p,u)$-adic on $\mf S\langle T\rangle$ and $\mf S$, $p$-adic on $\mf S_\mmax$ and the map on the left is given by multiplication by $pT-u^e$. This then gives a short exact sequence 
$$
0\tto \mf S\langle T\rangle^\bt \xymatrix{\ar[r]^{\cdot (pT-u^e)}&}  \mf S\langle T\rangle^\bt \tto \mf S_\mmax^\bt \tto 0
$$
of solid $\mf S^\bt$-modules. Indeed, terms in the sequence are represented by the corresponding topological groups. Then the first arrow is an embedding since the functor $V\ra \ul{V}$ is left exact and to see that the second arrow is a surjection one can use \cite[Lemma 3.1]{Griga}.
\end{rem}

\paragraph{Proof of the comparison.}

Having introduced the necessary notations we can now formulate a more precise form of \Cref{crys_comp_Bmax} which we actually prove:
\begin{prop}\label{crys_comp_Bmax_solid}
Let $\mstack X$ be a smooth $(d+1)$-de Rham proper stack over $\mathcal O_K$. Then there is a $(G_K,\phi)$-equivariant equivalence
$$\left(\tau^{\le d} R\Gamma_{\Prism\fr}^\bt(\mstack X_{\mc O_C}/\Ainf)\right)\sotimes_{A_\innf^\bt} B_\mmax^\bt \simeq \left(\tau^{\le d} R\Gamma_\crys^\bt(\mstack X_k/ W(k))\right)\sotimes_{W(k)^\bt} B_\mmax^\bt.$$

\begin{proof}
Recall that $B_\crys^\bt\coloneqq A_\crys^\bt[\frac{1}{p}]$ as solid ring. By base change for the map of prisms $(\Ainf,\xi)\ra (A_\crys,p)$ and crystalline comparison (Propositions \ref{prop:case change for prismatic cohomology} and \ref{prop:crystalline comparison}), after inverting $p$ we get a $(G_K,\phi)$-equivariant equivalence
$$
R\Gamma_{\Prism\fr}^\bt(\mstack X_{\mc O_C}/\Ainf)\sotimes_{A_\innf^\bt} B_\crys^\bt\simeq R\Gamma_\crys^\bt(\mstack X_{\mathcal O_{C}/p} / A_\crys)[\tfrac 1 p].
$$
Combining this with the Berthelot-Ogus isomorphism (\Cref{crys_O_C_W_k}) we obtain a $(G_K,\phi)$-equivariant equivalence
\begin{equation*}
R\Gamma_{\Prism\fr}^\bt(\mstack X_{\mc O_C}/\Ainf)\sotimes_{A_\innf^\bt} B_\crys^\bt \simeq  R\Gamma_\crys^\bt(\mstack X_k/ W(k)) \sotimes_{W(k)^\bt} B_\crys^\bt.
\end{equation*}
Finally, by tensoring it further with $B_\mmax^\bt$ over $B_\crys^\bt$ we obtain a $(G_K,\phi)$-equivariant equivalence
\begin{equation}
R\Gamma_{\Prism\fr}^\bt(\mstack X_{\mc O_C}/\Ainf)\sotimes_{A_\innf^\bt} B_\mmax^\bt \simeq R\Gamma_\crys^\bt(\mstack X_k/ W(k)) \sotimes_{W(k)^\bt} B_\mmax^\bt.
\end{equation}
Thus to deduce the assertion it would be enough to show that the natural maps
\begin{gather*}
\left(\tau^{\le d} R\Gamma_\crys^\bt(\mstack X_k/ W(k))\right)\sotimes_{W(k)^\bt} B_\mmax^\bt \tto \tau^{\le d} \left(R\Gamma_\crys^\bt(\mstack X_k/ W(k)) \sotimes_{W(k)^\bt} B_\mmax^\bt\right),\\
\left(\tau^{\le d} R\Gamma_{\Prism\fr}^\bt(\mstack X_{\mc O_C}/\Ainf)\right)\sotimes_{A_\innf^\bt} B_\mmax^\bt \tto \tau^{\le d} \left(R\Gamma_{\Prism\fr}^\bt(\mstack X_{\mc O_C}/\Ainf)\sotimes_{A_\innf^\bt} B_\mmax^\bt\right)
\end{gather*}
are equivalences.

To see that the first arrow is an equivalence note that $A_\mmax$ is $p$-completely free as a $W(k)$-module. Indeed, pick a $k$-basis ${\{x_s\}}_{s\in S}$ of $A_\mmax/p$, lifting it to $A_\mmax$ we get a map $\oplus_{s\in S} W(k)\cdot x_s\ra A_\mmax$ which becomes an isomorphism after $p$-completion. Consequently, $A_\mmax^\bt\simeq \widehat \oplus_{s\in S}W(k)^\bt$ as a $W(k)^\bt$-module. Then (by a similar argument to \Cref{rem:solid tensor product with Tate algebra}) the tensor product functor $-\sotimes_{W(k)^\bt} A_\mmax^\bt$ is given by $p$-completed direct sum when restricted to $p$-complete derived category $D(W(k)^\bt)_{p-\comp}$. Thus it is $t$-exact when restricted further to the essential image of the functor $(-)^\bt\coloneqq (-)^\bt_p\colon D(\Mod_{W(k)})\ra D(W(k)^\bt)_{p-\comp}$ by \Cref{prop_complited_direct_sums_t_exact_in_Solid}.  Since $B_\mmax^\bt\coloneqq A_\mmax^\bt[\frac{1}{p}]$ is a filtered colimit of free $A_\mmax^\bt$-modules, it is flat over $A_\mmax^\bt$, hence the composite functor $-\sotimes_{W(k)^\bt} B_\mmax^\bt=(-\sotimes_{W(k)^\bt} A_\mmax^\bt)\otimes_{A_\mmax^\bt} B_\mmax^\bt$ is $t$-exact on the essential image of $(-)^\bt$. Consequently, first arrow is an equivalence.

The second arrow requires a more involved analysis. Noting that $\Ainf$ is a $(p,u)$-completely free $\mf S$-module (again, by picking a basis of $\Ainf/(p,u)$ over $k\simeq \mf S/(p,u)$), similarly to the discussion above we get that the tensor product $-\otimes_{\mf S^\bt}A_{\innf}^\bt$ is $t$-exact when restricted to the essential image of solid $(p,u)$-completion functor $(-)^\bt\coloneqq (-)^\bt_{(p,u)}\colon D(\Mod_{\mf S})\ra D(\mf S^\bt)_{(p,u)-\comp}$. Thus, factorizing $\mf S^\bt\ra B_\mmax^\bt$ as $\mf S^\bt \ra A_\innf^\bt \ra B_\mmax^\bt$, using the above $t$-exactness and prismatic base change for $(\mf S,E(u))\ra (\Ainf,\xi)$ we can identify the second arrow with 
$$
\left(\tau^{\le d} R\Gamma_{\Prism\fr}^\bt(\mstack X/\mf S)\right)\sotimes_{\mf S^\bt} B_\mmax^\bt \tto \tau^{\le d} \left(R\Gamma_\Prism^\bt(\mstack X/\mf S)\sotimes_{\mf S^\bt} B_\mmax^\bt\right).
$$

Finally, note that $A_\mmax$ is $p$-completely free as a $\mf S_\mmax$-module. Indeed, $\mf S_\mmax/p\simeq k[u,t]/u^e$, where $t \coloneqq \frac{u}{p} \mod p$. Similarly, $A_\mmax/p\simeq \mc O_C^\flat[t]/[\pi^\flat]^e$, and the map $\mf S_\mmax/p\ra A_\mmax/p$ sends $u$ to $\pi^\flat$. A basis of $\mc O_C^\flat/[\pi^\flat]$ over $k$ then gives a basis of $A_\mmax/p$ as an $\mf S_\mmax/p$-module; its lift to $A_\mmax$ then provides a $p$-complete basis over $\mf S_\mmax$. Hence, as above, we deduce that the functor $-\sotimes_{\mf S_\mmax^\bt}\!\!\!\! B_\mmax^\bt$ is $t$-exact when restricted to the essential image of solid $p$-completion functor $(-)^\bt_p\colon D(\Mod_{\mf S_\mmax})\ra D(\mf S_\mmax^\bt)_{p-\comp}$. Thus, in the end it just remains to show that the natural map
$$\left(\tau^{\le d} R\Gamma_\Prism^\bt(\mstack X/\mathfrak S)\right)\sotimes_{\mathfrak S^\bt} \mf S_\mmax^\bt[\tfrac{1}{p}] \tto \tau^{\le d} \left(R\Gamma_\Prism^\bt(\mstack X/\mathfrak S)\sotimes_{\mathfrak S\bt} \mf S^\bt_\mmax [\tfrac{1}{p}]\right)$$
is an equivalence. 

Consider the fiber sequence
$$
\tau^{\le d}R\Gamma_{\Prism\fr}(\mstack X/\mathfrak S)\ra R\Gamma_{\Prism\fr}(\mstack X/\mathfrak S)\ra \tau^{\ge d+1}R\Gamma_{\Prism\fr}(\mstack X/\mathfrak S).
$$
Applying the functor $M\mapsto M^\bt\sotimes_{\mf S^\bt}\mf S^\bt_\mmax[\frac{1}{p}]$, and noting that $(\tau^{\le d}R\Gamma_{\Prism\fr}(\mstack X/\mathfrak S))^\bt\simeq \tau^{\le d}R\Gamma_{\Prism\fr}^\bt(\mstack X/\mathfrak S)$ by \Cref{lem:solid completion is right t-exact} we get that it's enough to show that
$$
\tau^{\le d}(\tau^{\ge d+1}R\Gamma_{\Prism\fr}(\mstack X/\mathfrak S))^\bt\sotimes_{\mf S^\bt}\mf S^\bt_\mmax[\tfrac{1}{p}]\simeq 0
$$
This follows from \Cref{lem:key lemma tensor with Bmax} below applied to $M\coloneqq \tau^{\ge d+1}R\Gamma_{\Prism\fr}(\mstack X/\mathfrak S)[d+1]$. It satisfies the conditions of the lemma by \Cref{lem:condition on the cohomology} and \Cref{rem:E(u)-torsion in d+1-st prismatic cohomology}.
\end{proof}
\end{prop}
\begin{lem}\label{lem:key lemma tensor with Bmax}
Let $M \in D^{\ge 0}(\Mod_{\mathfrak S})_{(p,u)-\comp}$ be a derived $(p,u)$-complete $\mathfrak S$-module and let $M^\bt\coloneqq M_{(p,u)}^{\bt}\in D(\mf S^\bt)$ denote its solid $(p, u)$-completion. Then, if $\tau^{<0}[M/(p,u)]\in \Coh(k)$ and $E(u)$-torsion $(H^0(M)[\frac{1}{p}])[E(u)]=0$, one has
$$\tau^{<0}\left(M^{\bt} \sotimes_{\mathfrak S} \mathfrak S_\mmax[\tfrac{1}{p}]\right) \simeq 0.$$

\begin{proof} First, note that by our assumption on $M$ and  \Cref{lem_discretness_of_cond_completion} we have that $M^\bt\in D(\Solid)^{\ge 0}$, or, in other words, $\tau^{<0}({M^\bt})\simeq 0$. Thus, it remains to show that the same holds after tensoring with $\mf{S}_\mmax^\bt[\tfrac{1}{p}]$. By \Cref{rem:ses for S_max} there is a short exact sequence of condensed $\mf S$-modules
	$$
	0\arr \mf S\langle T\rangle^\bt \xymatrix{\ar[r]^{\cdot (pT-u^e)}&} \mf S\langle T\rangle^\bt  \arr \mf S^\bt_\mmax \arr 0.
	$$ 
By \Cref{rem:solid tensor product with Tate algebra} and \Cref{prop:completed direct sum preserves t-structure} the functor $-\sotimes_{\mathfrak S} \mf S\langle T\rangle$ is $t$-exact on the essential image of $(-)^\bt$. Thus from the above short exact sequence we get that $$
	\tau^{<-1}\left({M^\bt} \sotimes_{\mathfrak S^\bt} \mathfrak S_\mmax^\bt\right) \simeq 0.
	$$
It remains to show that $H^{-1}({M^\bt} \sotimes_{\mathfrak S^\bt} \mathfrak S_\mmax^\bt)[\tfrac{1}{p}]\simeq 0$.

Let $N\coloneqq H^{-1}({M^\bt} \sotimes_{\mathfrak S^\bt} \mathfrak S_\mmax^\bt)$. We claim that $N$ is a finitely presented $\mf S_\mmax^\bt$-module.  Indeed, first note that $N$ is derived $p$-complete by \Cref{sotimes_vs_cotimes}. Thus, by \Cref{lem_fp_condensed} it is enough to check that the cohomology of $[N/p]$ are finitely generated modules over $\mf S_\mmax^\bt/p\simeq \ul{k[u,t]/u^e}$ (here $t$ is the class of $\tfrac{u}{p}$ modulo $p$, and the topology on $k[u,t]/u^e$ is trivial). We have a fiber sequence 
\begin{equation}\label{eq:fiber sequence for H^-1 Smax}
N[1]\tto {M^\bt} \sotimes_{\mathfrak S^\bt} \mathfrak S_\mmax^\bt \tto \tau^{\le 0}({M^\bt} \sotimes_{\mathfrak S^\bt} \mathfrak S_\mmax^\bt)	
\end{equation}
in $\Mod_{\mathfrak S}(D(\Solid))$-modules, which gives a fiber sequence 
$$
[N/p][1]\tto [({M^\bt} \sotimes_{\mathfrak S^\bt} \mathfrak S_\mmax^\bt)/p] \tto [\tau^{\le 0}({M^\bt} \sotimes_{\mathfrak S^\bt} \mathfrak S_\mmax^\bt)/p]
$$
The complex $[N/p]$ has at most two cohomology: $H^0$ and $H^{-1}$. From the fiber sequence above, using that $[\tau^{\le 0}({M^\bt} \sotimes_{\mathfrak S^\bt} \mathfrak S_\mmax^\bt)/p]\in D^{\ge -1}(\Solid)$, we see that 
$$H^{-1}(N/p)\simeq H^{-2}(([{M^\bt} \sotimes_{\mathfrak S^\bt} \mathfrak S_\mmax^\bt)/p]) \quad \text{ and } \quad  H^{0}(N)\hookrightarrow H^{-1}([({M^\bt} \sotimes_{\mathfrak S^\bt} \mathfrak S_\mmax^\bt)/p]).$$
 Note that the map $\mf S\ra \mf S_{\mmax}/p$ is given by the composition of the projection $W(k[[u]])\surj k[[u]]\surj k[u]/u^e$ and the unique $k[u]/u^e$-algebra map $k[u]/u^e\ra k[t,u]/u^e$. Thus we can identify $$[({M^\bt} \sotimes_{\mathfrak S^\bt} \mathfrak S_\mmax^\bt)/p]\simeq \ul{[M/(p,u^e)]\otimes_{{k[u]/u^e}} {k[u,t]/u^e}}$$
 and, consequently,
 $$\tau^{<0}([({M^\bt} \sotimes_{\mathfrak S} \mathfrak S_\mmax)/p])\simeq \ul{\tau^{<0}\left([M/(p,u^e)]\otimes_{{k[u]/u^e}} {k[u,t]/u^e}\right)}.$$ 
Since $\tau^{<0}[M/(p,u)]\in \Coh(k)$ by our assumptions on $M$ and $[M/(p,u^e)]$ has a $k$-step filtration with associated graded pieces given by $[M/(p,u)]$, one sees that $\tau^{<0}[M/(p,u^e)]\in \Coh(k[u]/u^e)$. Since $k[u,t]/u^e$ is flat over $k[u]/u^e$ from this we get that both $H^{-1}(N/p)$ and $H^{0}(N/p)$ are finitely generated $\ul{k[t,u]}$-modules, and so $N$ itself is finitely presented.

By \Cref{lem_fp_condensed}, we then see that $N$ comes as the image of the functor $\Mod_{\mf S_\mmax}^{\fg}\ra\Mod_{\mf S_\mmax^\bt}(\Solid)$ given by $L\mapsto \ul{L}\otimes_{\ul{\mf S_\mmax}} \mf S_\mmax^\bt$ (where here $\mf S_\mmax$ and $L$ are considered with discrete topology). More explicitly, we have $N\simeq \ul{N(*)}\otimes_{\ul{\mf S_\mmax}} \mf S_\mmax^\bt$ where $*\in \Cond$ is the singleton. In particular, we get that to show that $N[\frac{1}{p}]=0$ it is enough to show that the discrete $\mf S_{\mmax}[\frac{1}{p}]$-module $N(*)[\frac{1}{p}]$ is 0. 

First, let us show that $N(*)[\frac{1}{p}]$ is locally free.
Note that $\mf S_{\mmax}[\frac{1}{p}]\simeq K_0\langle\frac{u^e}{p}\rangle$ is a ring of functions on affinoid space of dimension 1 and so is a principal ideal domain. Thus it will be enough to show that the $\mf S_{\mmax}[\frac{1}{p}]$-module $N(*)[\frac{1}{p}]$ is $I$-torsion free for any maximal ideal $I\subset \mf S_{\mmax}[\frac{1}{p}]$. Since $I$ is automatically principal, the $I$-torsion $N(*)[\frac{1}{p}][I]$ is given by $H^{-1}([N(*)[\frac{1}{p}]/I])$ and so it will be enough to show that the latter group is 0. Any maximal ideal of $\mf S_{\mmax}[\frac{1}{p}]$ is given by the kernel of a surjective homomorphism $\mf S_\mmax\surj L$ onto some finite extension $L/K_0$ that sends $u\in \mf S_\mmax$ to some element $x\in L$ such that $x^e/p\in \mc O_L$. In particular, $x$ itself should lie in $\mc O_L$ and so $I$ is generated by the minimal polynomial $f_x\in W(k)[u]$ of $x$ over $W(k)$. Note that this way $f_x$ in fact lies in $\mf S$ and that also $\mf S[\frac{1}{p}]/(f_x)\simeq L$. Thus the functor $-\sotimes_{\mf S^\bt} \mf S_\mmax^\bt[\tfrac{1}{p}]/(f_x)$ can be rewritten as $-\sotimes_{\mf S^\bt}\mf S^\bt[\tfrac{1}{p}]/(f_x)$; this way from the fiber sequence (\ref{eq:fiber sequence for H^-1 Smax}) we get a fiber sequence 
\begin{equation}\label{eq:fiber sequence mod f_x}
[N[\tfrac{1}{p}]/f_x][1]\tto [{M^\bt}[\tfrac{1}{p}] /f_x] \tto [(\tau^{\le 0}{M^\bt} \sotimes_{\mathfrak S^{\bt}} \mathfrak S_\mmax^\bt[\tfrac{1}{p}])/f_x].
\end{equation}
Neither of the last two terms have $H^{-2}$, and so we get that $H^{-1}([N[\tfrac{1}{p}]/I])=0$. Consequently, $$H^{-1}([N(*)[\tfrac{1}{p}]/I])=H^{-1}([N[\tfrac{1}{p}]/I])(*)=0.$$

Thus, $N(*)[\frac{1}{p}]$ is locally free and (since $\Spec \mf S_{\mmax}[\frac{1}{p}]$ is obviously connected) to show that it is 0 it is enough to show that so is its reduction $N(*)[\frac{1}{p}]/I$ modulo any ideal $I\nsubseteq R$ is 0. Note that $\mf S_\mmax/E(u)\simeq \mc O_K$ via the map sending $u$ to $\pi$; since $e$ is the ramification index of $K$ the element $\pi^e/p\in \mc O_K$, and so the map above is well defined. Using the fiber sequence (\ref{eq:fiber sequence mod f_x}) for $f_x=E(u)$ we also get that $N[\frac{1}{p}]/E(u)\simeq H^0([N[\frac{1}{p}]/E(u)])$ embeds into $H^{-1}({M^\bt}[\tfrac{1}{p}] /E(u))$. Thus, since evaluation at $*\in (*)_\proet$ is $t$-exact, to check that $N(*)[\frac{1}{p}]/E(u)$ is 0 it is enough to show that $H^{-1}({M^\bt}[\tfrac{1}{p}] /E(u))(*)$ is. But since the original $M$ was $(p,u)$-complete by \Cref{rem:solid completion on the point} we have 
$$
{M^\bt}[\tfrac{1}{p}] /E(u)(*)\simeq M[\tfrac{1}{p}]/E(u),
$$ and so $H^{-1}$ in question is exactly given by $(H^0(M)[\tfrac{1}{p}])[E(u)]$, which is 0 by our assumptions on $M$.
\end{proof}
\end{lem}

\noindent\textit{Proof of \Cref{crys_comp_Bmax}:} We will obtain the isomorphism in \Cref{crys_comp_Bmax} from the isomorphism in \Cref{crys_comp_Bmax_solid} by evalutiong the latter on the point. By \Cref{rem:condition on crystalline cohomology} we have that $M\coloneqq \tau^{\le d+1}\RG_\crys(\mstack X_k/W(k))[d+1]$ satisfies $\tau^{<0}([M/p])\in\Coh(k)$, and so $(\tau^{\le d+1}\RG_\crys(\mstack X_k/W(k)))_p^\bt\in D^{\ge d+1}(W(k)^\bt)$ by \Cref{lem_discretness_of_cond_completion}. Since $(-)^\bt$ is right $t$-exact, from this we get that $\tau^{\le d}(\RG_\crys^\bt(\mstack X_k/W(k)))\simeq (\tau^{\le d}\RG_\crys(\mstack X_k/W(k)))^\bt$. Thus, by \Cref{sotimes_vs_cotimes} we have 
$$
\tau^{\le d}(\RG_\crys^\bt(\mstack X_k/W(k)))\otimes_{W(k)^\bt}B_\mmax^\bt\ism \left(\tau^{\le d}\RG_\crys(\mstack X_k/W(k))\cotimes_{W(k)}A_\mmax\right)^\bt[\tfrac{1}{p}],
$$
where the tensor product on the right is $p$-completed. Consequently, we have a natural equivalence
$$
\left(\tau^{\le d}(\RG_\crys^\bt(\mstack X_k/W(k)))\otimes_{W(k)^\bt}B_\mmax^\bt\right)(*)\ism \left(\tau^{\le d}\RG_\crys(\mstack X_k/W(k))\cotimes_{W(k)}A_\mmax\right)[\tfrac{1}{p}]
$$
for the value on the point $*\in (*)_\proet$.
 Since by \Cref{cor:finiteness of crystalline} we have $\RG_\crys(\mstack X_k/W(k))\in \Coh(k)$, the tensor product $\tau^{\le d}\RG_\crys(\mstack X_k/W(k))\otimes_{W(k)}A_\mmax$ is already $p$-complete and we get a natural equivalence 
 $$
 \left(\tau^{\le d}(\RG_\crys^\bt(\mstack X_k/W(k)))\otimes_{W(k)^\bt}B_\mmax^\bt\right)(*)\ism \left(\tau^{\le d}\RG_\crys(\mstack X_k/W(k))\right)\otimes_{W(k)}B_\mmax.
 $$

Also, from the proof of \Cref{crys_comp_Bmax_solid} we have an  equivalence 
$$
\left(\tau^{\le d} R\Gamma_{\Prism\fr}^\bt(\mstack X/\mf S)\right)\sotimes_{\mf S^\bt} B_\mmax^\bt \ism \left(\tau^{\le d} R\Gamma_{\Prism\fr}^\bt(\mstack X_{\mc O_C}/\Ainf)\right)\sotimes_{A_\innf^\bt} B_\mmax^\bt 
$$
and $\tau^{\le d} R\Gamma_{\Prism\fr}^\bt(\mstack X/\mf S)\simeq (\tau^{\le d} R\Gamma_{\Prism\fr}(\mstack X/\mf S))^\bt$. Thus, by \Cref{sotimes_vs_cotimes} we can rewrite left hand side as $(\tau^{\le d} R\Gamma_{\Prism\fr}(\mstack X/\mf S)\cotimes_{\mf S}A_\mmax)^\bt[\frac{1}{p}]$, where the tensor product is $(p,u)$-(or, equivalently, just $(p)$, since $u^e\in (p)$ in $A_\mmax$)-completed. By \Cref{cor:BK prismatic cohomology d-coherent} we have $\tau^{\le d} R\Gamma_{\Prism\fr}\in \Coh(\mf S)$ and so the tensor product $\tau^{\le d} R\Gamma_{\Prism\fr}(\mstack X/\mf S)\otimes_{\mf S}A_\mmax$ is already $p$-complete. This way (by \Cref{lem:comparison of prismatic with Ainf}) we can rewrite the left hand side further as $(\tau^{\le d} R\Gamma_{\Prism\fr}(\mstack X_{\mc O_C}/\Ainf)\otimes_{\Ainf}A_\mmax)^\bt[\frac{1}{p}]$. Thus, restricting to the point, we get that the natural map 
$$
\tau^{\le d} R\Gamma_{\Prism\fr}(\mstack X_{\mc O_C}/\Ainf)\otimes_\Ainf B_\mmax \tto \left(\left(\tau^{\le d} R\Gamma_{\Prism\fr}^\bt(\mstack X_{\mc O_C}/\Ainf)\right)\sotimes_{A_\innf^\bt} B_\mmax^\bt \right)(*)
$$
is an equivalence. Finally, this way we get a commutative diagram 
$$
\xymatrix{\left(\left(\tau^{\le d} R\Gamma_{\Prism\fr}^\bt(\mstack X_{\mc O_C}/\Ainf)\right)\sotimes_{A_\innf^\bt} B_\mmax^\bt \right)(*) \ar[r]^\sim& \left(\tau^{\le d}(\RG_\crys^\bt(\mstack X_k/W(k)))\otimes_{W(k)^\bt}B_\mmax^\bt\right)(*)\ar[d]^\wr\\
\left(\tau^{\le d} R\Gamma_{\Prism\fr}(\mstack X_{\mc O_C}/\Ainf)\right)\otimes_\Ainf B_\mmax \ar[u]^\wr& \left(\tau^{\le d}\RG_\crys(\mstack X_k/W(k))\right)\otimes_{W(k)}B_\mmax
}
$$
that provides the desired isomorphism in \Cref{crys_comp_Bmax}.

\section{Applications for schemes}\label{sec:applications for schemes}
Here we record a couple of applications of the above results about \'etale cohomology of $d$-de Rham proper stacks that lead to some new results in $p$-adic Hodge theory for schemes.

Let $X$ be a scheme over $\mc O_K$ and let $Z\hookrightarrow X_k$ be a closed subscheme of the special fiber. Then the geometric Raynaud generic fiber $\widehat{(X/Z)}_C\subset \widehat X_C$ can be viewed as a complement (of a sort) in $\widehat{X}_C$ to an open tubular neighborhood around $Z$ (meaning in particular that on the level of classical points of the corresponding rigid-analytic spaces we throw out the "residue discs" around all points of $Z$).

\begin{thm}\label{cor:applications for schemes}
	Let $X$ be a proper scheme over $\mc O_K$ that is Cohen-Macauley. Let $Z\hookrightarrow X_k$ be a codimension $d$ closed subscheme of the special fiber that contains the singularities of $X$ (meaning that the complement $X\backslash Z$ is smooth over $\mc O_K$). Then 
	\begin{enumerate}
		\item  (Purity) There are natural isomorphisms $H^i_\et(\widehat{(X\backslash Z)}_C,\mbb Q_p)\simeq H^i_\et(\widehat X_C,\mathbb Q_p)$ for $i\le d-2$ and an embedding $H^{d-1}_\et(\widehat X_C,\mathbb Q_p)\hookrightarrow H^{d-1}_\et(\widehat{(X\backslash Z)}_C,\mbb Q_p)$;
		\item (Crystallinity) $H^i_\et(X_C,\mbb Q_p)$ is a crystalline Galois representation for $i\le d-2$.
	\end{enumerate}
\end{thm}
\begin{proof}
	By \Cref{lem:example of d-Hodge-proper} $X\backslash Z$ is $(d-1)$-Hodge-proper, and so also $(d-1)$-de Rham proper over $\mc O_K$. Note that $H^i_\et(\widehat X_C,\mathbb Q_p)\simeq H^i_\et(X_C,\mathbb Q_p)$ since $X$ is proper. Also note that $(X\backslash Z)_C\simeq X_C$, since $Z$ lies in the closed fiber, and so $H^i_\et(\widehat X_C,\mathbb Q_p)\simeq H^i_\et((X\backslash Z)_C,\mathbb Q_p)$. This way part (1) is a corollary of \Cref{thm:even mainer theorem} while part (2) is a corollary of \Cref{cor:Fontaine's conjecture for Hodge-proper stacks}, both applied to $X\backslash Z$.
\end{proof}

\begin{rem}
	The part 2 of \Cref{cor:applications for schemes} can be reformulated as follows. Let $X$ be a smooth proper scheme over $K$. Assume that it has an integral model $\ol X$ over $\mc O_K$ that is Cohen-Macauley and such that the singularities of $\ol X$ are of codimension $d$ in the closed fiber.
	Then $H^i_\et(X_C,\mbb Q_p)$ is a crystalline $G_K$-representation for $i\le d-2$. In other words, we see that the existence of an integral Cohen-Macauley model such that singularities are of large codimension forces the \'etale cohomology be crystalline up to some degree depending on this codimension.
\end{rem}

\appendix
\section{Some homological algebra of complete modules}\label{sec:appendix about complete sums}
Here we prove some results on $t$-exactness of completed tensor products in certain situations. They will ultimately lead us to the correct bounds on the cohomological degrees in Sections \ref{ssec:integral p-adic Hodge theory} and \ref{ssec: crystalline comparison}. 

Recall that for a ring $R$ we denote by $\Mod_R$ the \emph{abelian} category of $R$-modules, while by $D(\Mod_R)$ we denote the unbounded derived category of $\Mod_R$ (considered as an $\infty$-category). For a finitely generated ideal $I$ we denote by $D(\Mod_R)^\wedge_I\subset D(\Mod_R)$ the full subcategory of derived $I$-complete modules. 
\begin{lem}\label{lem:p-completed product abelian module}
Let $R$ be a ring and let $x\in R$ be an element. Let $\{M_\alpha\}$ be a set of derived $x$-complete $R$-modules. Then the derived completed direct sum $\widehat \bigoplus_\alpha M_\alpha$ is still concentrated in cohomological degree $0$.

\begin{proof}
We need to show that $H^{-1}\left(\widehat \bigoplus_\alpha M_\alpha\right) \simeq 0$. By Milnor's exacts sequence
$$H^{-1}\left(\widehat \bigoplus_\alpha M_\alpha\right) \simeq \prolim[n] H^{-1}\left(\bigoplus_\alpha [M_\alpha/x^n]\right) \simeq \prolim[n] \bigoplus_\alpha H^{-1} ([M_\alpha/x^n]).$$
Since direct sums embed into direct products and since limits preserve injections we have
$$\prolim[n] \bigoplus_\alpha H^{-1} ([M_\alpha/x^n]) \inj \prolim[n] \prod_\alpha H^{-1} ([M_\alpha/x^n]) \simeq \prod_\alpha \prolim[n] H^{-1}([M_\alpha/x^n]) \simeq 0,$$
where the right hand side vanishes by derived $x$-completeness of $M$.
\end{proof}
\end{lem}

\begin{cor}
	The $x$-completed direct sum functor 
	$$
	-^{\widehat\oplus_S}\colon  D(\Mod_{R})^\wedge_x \tto D(\Mod_{R})^\wedge_x
	$$
	is $t$-exact.
\end{cor}
\begin{cor}
 Let $k$ be a perfect field of characteristic $p$. Then the $p$-completed tensor product functor 
 $$
 -\widehat\otimes_{\mbb Z_p}W(k)\colon D(\Mod_{\mbb Z_p})^\wedge_p \tto D(\Mod_{W(k)})^\wedge_p
 $$ 
 is $t$-exact.
\end{cor}
\begin{proof}
	Note that the forgetful functor $\mr{obl}\colon D(\Mod_{W(k)})^\wedge_p \ra  D(\Mod_{\mbb Z_p})^\wedge_p$ is conservative and $t$-exact. Thus it is enough to show the claim for the composition $\mr{obl\circ}(-\widehat\otimes_{\mbb Z_p}W(k))$. Pick a basis $\{v_s\}_{s\in S}$ of $k$ over $\mbb F_p$, where $S$ is some indexing set. Consider a $\mbb Z_p$-submodule $F_A\coloneqq\oplus_{s\in S}\mbb Z_p\cdot [v_s] \subset W(k)$, where $[-]\colon k \dashrightarrow W(k)$ is the Teichm\"uller lift; it is easy to see that $W(k)$ is the derived $p$-completion of $F_S$. Thus, composition of $-\widehat\otimes_{\mbb Z_p}W(k)$ with $\mr{obl}$ can be identified with the $p$-completed direct sum $-^{\widehat\oplus_S}$. We are done by the corollary above.
\end{proof}

We will now consider a slightly more complicated context. Namely, let $\mf S\coloneqq W(k)[[u]]$ for some perfect field $k$ of characteristic $p$. We endow $\mf S$ with the $(p,u)$-adic topology. We would like to have an analogue of \Cref{lem:p-completed product abelian module} where we take $R=\mf S$ and replace the ideal $(x)$ by $(p,u)$. However, unfortunately it seems that the analogous statement just doesn't hold. Nevertheless, it does hold if we make some further assumptions on $M$.

\begin{prop}\label{prop:pu-completed direct sum}
Let $\{M_\alpha\}$ be a set of derived $(p,u)$-complete $\mathfrak S$-modules. Assume that $\tau^{<0}(M_\alpha\otimes^{\mathbb L}_{\mathfrak S}\mf S/(p,u))\in \Coh(k)$ for all $\alpha$. Then the $(p,u)$-completed direct sum $\widehat\bigoplus_\alpha M_\alpha$ is concentrated in degree $0$.

\begin{proof}
We need to show that
$$H^{-1}\left(\prolim{}_{n,m} \bigoplus_S [M_\alpha/(p^n, u^m)]\right) \simeq H^{-2}\left(\prolim{}_{n,m} \bigoplus_\alpha M_\alpha/(p^n, u^m)\right) \simeq 0.$$
By Milnor's exact sequence this is amount to the following vanishings:
\begin{align*}
H^{-1} & \simeq 0 \quad\Leftrightarrow\quad \prolim{}_{n,m}H^{-1}\left(\bigoplus_\alpha[M_\alpha/(p^n, u^m)]\right) \simeq 0 \ \text{ and }\ \prolim{}_{n,m}^1 H^{-2}\left(\bigoplus_\alpha [M_\alpha/(p^n, u^m)]\right) \simeq 0,\\
H^{-2} & \simeq 0 \quad\Leftrightarrow\quad \prolim{}_{n,m} H^{-2}\left(\bigoplus_\alpha [M_\alpha/(p^n, u^m)]\right) \simeq 0.
\end{align*}
Since direct sums embed into direct products and since limits preserve injections we have for $i\in \{-1, -2\}$
\begin{align*}
\prolim{}_{n,m}H^{i}\left(\bigoplus_\alpha[M_\alpha/(p^n, u^m)]\right) & \simeq \prolim{}_{n,m} \bigoplus_\alpha H^{i}\left([M_\alpha/(p^n, u^m)]\right)\inj \\
\inj \prolim{}_{n,m} \prod_\alpha H^{i} \left([M_\alpha/(p^n, u^m)]\right) & \simeq \prod_\alpha \prolim{}_{n,m} H^{i} \left([M_\alpha/(p^n, u^m)]\right) \simeq 0,
\end{align*}
where the right hand side vanishes by the derived $(p,u)$-completeness of $M$. Hence it is left to prove the vanishing of $\prolim{}_{n,m}^1 H^{-2}$. Note that
$$H^{-2}\left(\bigoplus_\alpha [M_\alpha/(p^n, u^m)]\right) \simeq \bigoplus_\alpha M_\alpha[p^n, u^m].$$
But by \Cref{lem_for_compl_dir_sums} below the diagram $\{\bigoplus_\alpha M[p^n, u^m]\}_{n,m}$ is pro-zero (since the composition of $(N+1)$ successive vertical or horizontal maps is given by multiplication with $p^{N+1}$ or $u^{N+1}$, hence vanishes), so its $\lim^1$ is zero.
\end{proof}
\end{prop}

\begin{defn}\label{defn:Tate module}
	Given a ring $R$ with an element $x\in R$, for a classical module $M\in \Mod_{\mf S}$ we denote by $T_x(M)$ the \textit{$x$-adic Tate module} of $M$: $$
	T_x(M)\coloneqq {\lim_n}^0(\ldots \xra{\cdot x}M[x^n]\xra{\cdot x} M[x^{n-1}]\xra{\cdot x} \ldots \xra{\cdot x}M[x]).
	$$
	Note that $T_x(M)\simeq H^{-1}(M^\wedge_x)$, where $M^\wedge_x$ is the derived $x$-adic completion. In particular, if $M\in \Mod_{\mf S}$ is derived $x$-adically complete, then $T_x(M)=0$.
\end{defn}

\begin{lem}\label{lem_for_compl_dir_sums}
Let $M\in \Mod_{\mf S}$ be a derived $(p,u)$-complete module such that $\tau^{<0}(M\otimes^{\mbb L}_{\mf S}\mf S/(p,u))\in \Coh(k)$. Then $M[p^\infty, u^\infty] = M[p^N, u^N]$ for some $N\ge 0$.

\begin{proof} Note that $H^{-2}(M\otimes^{\mbb L}_{\mf S}\mf S/(p,u))\simeq M[p][u]$ and so by our assumption on $M$, $\dim_k M[p,u] =: d < \infty$. Also, $M\otimes_{\mf S}^{\mbb L}\mf S/p$ is derived $u$-complete, and thus so is $M[p]$. Since $T_u(M[p])\simeq T_u(M[p][u^\infty])$, by \Cref{lem:about torsion module with bounded u-torsion} we get that the $k[[u]]$-module $M[p][u^\infty]$ is of finite length. 
	
Let $M_n\coloneqq M[p^n][u^\infty]$ and $N_n\coloneqq M_n/M_{n-1}$. We have $N_1=M_1=M[p][u^\infty]$. Multiplication by $p$ induces an embedding $N_i\hookrightarrow N_{i-1}$, providing a chain of inclusions $\ldots \hookrightarrow N_i\hookrightarrow N_{i-1}\inj \ldots \inj N_1$. Since $N_1$ is of finite length, this chain stabilizes starting from some $n$. Also, each $M_i$ is of finite length, in particular $M_i= M_n[p^N][u^N]$ for some $N\gg 0$. If $N_n=0$, $M[p^\infty][u^\infty]= M_n$, we are done. If $N_n\neq 0$, take the quotient $M'=M/M_{n-1}$. Module $M_{n-1}$ is coherent, in particular derived $(p,u)$-complete; consequently, so is $M'$. Let $M_i'\coloneqq M'[p^i]$ and $N_i'\coloneqq M'_i/M'_{i-1}$; note that $N_i'\simeq N_{i+n-1}$. Since multiplication by $p$ induces an isomorphism $N_{i+1}\simeq N_{i}$ for $i\ge n$, it induces isomorphisms $N'_{i+1}\simeq N'_{i}$ for all $i\ge 1$ and so map $M'_i\xra{\cdot p} M'_{i-1}$ is surjective for all $i$ (indeed, by the above, multiplication by $p$ is a surjection on the associated graded for the filtration $M'_1\hookrightarrow M'_2\inj\ldots$). Since $N'_1\simeq N_n\neq 0$ we get that the $p$-adic Tate module $T_p(M')\neq 0$, and thus $M'$ is not derived $p$-complete, contradiction.	
\end{proof}
\end{lem}

\begin{lem}\label{lem:about torsion module with bounded u-torsion}
	Let $M$ be a $u^\infty$-torsion $k[[u]]$-module of infinite length such that $\dim_k M[u]< \infty$ for all $n\ge 0$. Then $T_u(M)\neq 0$.
\end{lem}
\begin{proof}
The proof is similar to the one in the end of \Cref{lem_for_compl_dir_sums}. Let $M_n\coloneqq M[u^n]$ and $N_n\coloneqq M_n/M_{n-1}$. We have $N_1=M_1=M[u]$ and multiplication by $u$ induces a chain of embeddings $\ldots \inj N_n \inj N_{n-1}\inj \ldots \inj N_1$. Since $\dim_k N_1< \infty$ this chain stabilizes starting from some $n$. If $N_n=0$, then $M=M_n$ is of finite length, contradicting our assumptions on $M$. If $N_n\neq 0$, consider $M'\coloneqq M/M_{n-1}$. We have a fiber sequence $(M_{n-1})^\wedge_u \ra M^\wedge_u \ra (M')^\wedge_u$, which gives a left-exact sequence 
$$
0\tto T_u(M) \tto T_u(M') \tto M_{n-1} 
$$ 
since $M_{n-1}$ is finitely generated, and so is derived $u$-complete. Let $M_i'\coloneqq M'[u^i]$ and $N_i'\coloneqq M'_i/M'_{i-1}$; note that $N_i'\simeq N_{i+n-1}$. We now have that multiplication by $u$ gives isomorphisms $N'_i\xra{\sim}N'_{i-1}$ for all $i$. It follows that the map $M'_i\xra{\cdot u} M'_{i-1}$ is surjective for all $i$, and since $M'_1\simeq N_n$ is non-zero we get $T_u(M')\neq 0$. Moreover, if we put $d\coloneqq \dim_k M'_1$, then one has a non-canonical isomorphism $T_u(M')\simeq k[[u]]^{\oplus d}$ and, since $M_{n-1}$ is finite, this forces $T_u(M)$ to be isomorphic to $ k[[u]]^{\oplus d}$ as well, in particular to be non-zero.
\end{proof}

Consider the $p$-adic completion $\mf S\langle \frac{1}{u}\rangle\coloneqq \mf S[\frac{1}{u}]^{\wedge}_p$. We would also like to have an analogue of \Cref{lem:p-completed product abelian module} for the $p$-completed tensor product functor $-\widehat\otimes_{\mf S}\mf S\langle \frac{1}{u}\rangle$. Similarly to the completed direct sum, this functor is a composition of a filtered colimit of $(p,u)$-complete modules and derived completion. However, the colimit here is more complicated and this reduces the generality in which $-\widehat\otimes_{\mf S}\mf S\langle \frac{1}{u}\rangle$ is $t$-exact, as the examples below show.
\begin{ex}\label{ex:counterexample}
	Consider an $\mf S$-module $N=\mf S[\frac{u}{p}]/\mf S$. Explicitly, it is generated by elements $x_i\coloneqq \frac{u^i}{p^i}$ for $i\ge 0$ and with relations given by $x_0=0$ and $p\cdot x_i=u\cdot x_{i-1}$. It is easy to see that $N$ is $u$-torsion free and that $p^nN=u^nN$. Also, $p^n$-torsion $N[p^n]$ is generated by $x_1, \ldots, x_n$, so it is finitely-generated and in particular derived $u$-complete. It is also not hard to see that it is $u$-torsion free. We claim that the derived $(p,u)$-completion $M\coloneqq N^\wedge_{(p,u)}$ coincides with the classical $p$-completion of $N$ which can be identified with $\mf S\langle\frac{u}{p}\rangle/\mf S$. Indeed, the derived $p$-completion $N^\wedge_p$ maps to the classical $p$-adic completion $\mf S\langle\frac{u}{p}\rangle/\mf S$ with the fiber given by the two-term Milnor complex
	$$
	\prod_n N[p^n] \xymatrix{\ar[r]^{\Id - (\cdot p)}&} \prod_n N[p^n]
	$$
	where $(\cdot p)$ maps $(n_1,n_2,n_3,\ldots)$ to $(pn_2,pn_3,\ldots)$. Both source and target are derived $u$-complete and $u$-torsion free. Moreover $(\cdot p)$ is 0 modulo $u$, and so the induced map $\prod_n N[p^n]/u \ra \prod_n N[p^n]/u$ is simply the identity. By derived Nakayama lemma we get that the above two-term complex is quasi-isomorphic to 0. Finally, we have $M\simeq (N^\wedge_p)^\wedge u$, but $N^\wedge_p\simeq \mf S\langle\frac{u}{p}\rangle/\mf S$ is $u$-torsion free and classically $u$-complete; thus it is already derived $u$-complete as well.
	
	  We now claim that $M\widehat\otimes_{\mf S}\mf S\langle \frac{1}{u}\rangle$ is no longer concentrated in degree 0.
Indeed, $H^{-1}(M\widehat\otimes_{\mf S}\mf S\langle \frac{1}{u}\rangle)$ is given by the $p$-adic Tate module $\lim^0_n (M\widehat\otimes_{\mf S}\mf S\langle \frac{1}{u}\rangle)[p^n]$, which is no longer zero after we have inverted $u$. Indeed, we have $\frac{1}{p^n}=x_n\cdot \frac{1}{u^n}$ and so $(\ldots, x_3/u^3,x_2/u^2,x_1/u)$ gives a non-zero element of $H^{-1}(M\widehat\otimes_{\mf S}\mf S\langle \frac{1}{u}\rangle)$. 
\end{ex}
\begin{rem}
One can also show that the derived solid $(p,u)$-adic completion $ M^\bt_{(p,u)}\in D(\Solid)$ taken in (derived) condensed abelian groups lies in the heart $\Solid\subset D(\Solid)$, but the derived $p$-completed tensor product $M^\bt_{(p,u)}{\sotimes}_{{\mf S}^\bt} \mf S\langle \frac{1}{u}\rangle^\bt$ does not (since by \Cref{ex:counterexample} $H^{-1}\neq 0$ for the value on the point $*$). Thus, it is not true in the condenced setting that the $p$-completed tensor product $-{\sotimes}_{{\mf S}^\bt} \mf S\langle \frac{1}{u}\rangle^\bt$ is $t$-exact as functor on derived $(p,u)$-complete solid ${\mf S}^\bt_{(p,u)}$-modules. In particular, the condition on $M$ that $M^{\bt}\in D(\Solid)^\heartsuit$ in this case is not sufficient.
\end{rem}
Nevertheless, the finiteness assumptions on $M$ that we had in \Cref{prop:pu-completed direct sum} still turn out to be enough:
\begin{prop}\label{prop:p-completed inversion of u is sometimes t-exact}
	Let $M\in \Mod_{\mf S}$ be a derived $(p,u)$-complete module such that $\tau^{<0}(M\otimes^{\mbb L}_{\mf S}\mf S/(p,u))\in \Coh(k)$. Then the $p$-completed tensor product
	$
	M\widehat{\otimes}^{\mbb L}_{\mf S} \mf S\langle \frac{1}{u}\rangle
	$ 
	is concentrated in cohomological degree 0. 
\end{prop}
\begin{proof}Note that if we have a map $M\xra{f} M'$ with $M'$ satisfying the same conditions as $M$, such that the fiber $\fib(f)\in \Coh(\mf S)$ (equivalently, kernel and cokernel of $f$ are finitely generated) then the conclusion of the proposition holds for $M$ if it holds for $M'$. Indeed, by coherence $\fib(f)\widehat{\otimes}^{\mbb L}_{\mf S} \mf S\langle \frac{1}{u}\rangle\simeq \fib(f){\otimes}^{\mbb L}_{\mf S} \mf S\langle \frac{1}{u}\rangle$ and the latter is concentrated in cohomological degrees 0 and 1, since $M$ and $M'$ lie in the heart and $\mf S\langle \frac{1}{u}\rangle$ is flat over $\mf S$. The claim then follows by considering the long exact sequence of cohomology.
	
	Recall that, by \Cref{lem_for_compl_dir_sums} we have $M[p^\infty][u^\infty]=M[p^N][u^N]$ for some $N\ge 0$. Thus, by the above we can replace $M$ by $M/(M[p^N][u^N])$ and assume that $M[p^\infty]$ is $u$-torsion free. 
	
	Let $M_n\coloneqq M[p^n]$. We need to show that $H^{-1}(M\widehat\otimes_{\mf S}\mf S\langle \frac{1}{u}\rangle)$ vanishes, where 
	$$
	H^{-1}(M\widehat\otimes_{\mf S}\mf S\langle \tfrac{1}{u}\rangle) \ism \lim_n (\ldots \xra{\cdot p} M_2\otimes_{\mf S}\mf S[\tfrac{1}{u}]\xra{\cdot p}M_1\otimes_{\mf S}\mf S[\tfrac{1}{u}]).
	$$
Define $N_n\coloneqq M_n/M_{n-1}$; these are $k[[u]]$-modules. By our assumptions on $M$, $N_1=M_1=M[p]$ is finitely generated and $u$-torsion free, and so is isomorphic to $k[[u]]^{\oplus s}$ for some $s$. Multiplication by $p$ induces an embedding $N_n\hookrightarrow N_{n-1}$ for all $n$, ultimately giving a sequence of embeddings $\ldots\hookrightarrow N_n \hookrightarrow N_{n-1} \hookrightarrow \ldots \hookrightarrow N_1$. Since $N_1\simeq k[[u]]^{\oplus s}$, all $N_i$ are necessarily free; starting from some $k\gg 0$ the rank $\rk N_k$ stabilizes. Factoring $M$ by $M_k$ we can assume that all $N_i$ are isomorphic to $k[[u]]^{\oplus s}$ for a given $s$. Note that $s=0$ we are done (because then $p$-torsion in $M$ is in fact bounded).
	
	Our goal will now be to show that if $s\neq 0$ then either $M$ is not derived $p$-complete, or $M/p$ is too big: namely, either the $u$-torsion $(M/p)[u]\simeq \Tor^1_{\mf S}(M,\mf S/(p,u))$ is infinitely generated, or there is a $u$-divisible element in $M/p$ killed by $u$ (and thus $M/p$ is not derived $u$-complete). Note that we have an embedding $M[p^\infty]/pM[p^\infty]\hookrightarrow M/pM$. Moreover, for each $n$ we have an embedding $\iota_{n-1}\colon M_{n-1}/pM_n\hookrightarrow M_n/pM_{n+1}$, and $M[p^\infty]/pM[p^\infty]=\colim_n (M_{n-1}/pM_n)$. If $\iota_{n}$ are isomorphisms starting from some $k$, then replacing $M$ by $M/M_k$ we can assume that the maps $M_n \xra{\cdot p} M_{n-1}$ are surjective for all $n$ and thus the ($p$-adic) Tate module of $M$ is non-zero. It follows that $H^{-1}(M^\wedge_p)\neq 0$ and so $M$ is not derived $p$-complete --- a contradiction. We get that $\iota_k$ is a strict embedding for infinitely many $k$.
	
	Moreover, we claim that each $M_n/pM_{n+1}$ differs from $M_{n-1}/pM_n$ by a (finitely generated) $u$-torsion module. Indeed, we need to understand $(M_n/pM_{n+1})/(M_{n-1}/pM_n)\simeq M_n/(M_{n-1}+pM_{n+1})$. By our assumptions, each $M_{n+1}$ is an extension of $k[[u]]^{\oplus s}$ by $M_{n}$. Using the resolution
	$$
	0\tto \mf S \xymatrix{\ar[r]^{\cdot p}&} \mf S\tto k[[u]] \tto 0,
	$$
	for any $L\in \Mod_{\mf S}$ one can identify $\Ext^1_{\mf S}(k[[u]]^{\oplus s},L)$ with $\Hom_{\mf S}(k[[u]]^{\oplus s}, L/p)$. For each $n$, let $e_n\in \Hom_{\mf S}(k[[u]]^{\oplus s}, k[[u]]^{\oplus s})$ be the image of the class $[M_{n+1}]\in \Ext^1_{\mf S}(k[[u]]^{\oplus s},M_{n})$ under the natural map $\Ext^1_{\mf S}(k[[u]]^{\oplus s},M_{n})\ra \Ext^1_{\mf S}(k[[u]]^{\oplus s},k[[u]]^{\oplus s})$ (where we identify $M_{n}/M_{n-1}\simeq k[[u]]^{\oplus s}$) and the isomorphism $\Ext^1_{\mf S}(k[[u]]^{\oplus s},k[[u]]^{\oplus s})\simeq \Hom_{\mf S}(k[[u]]^{\oplus s}, k[[u]]^{\oplus s})$. Note that the image of $pM_{n+1}$ in $M_n/M_{n-1}$ is exactly the image of $e_n$. In particular, $M_n/(M_{n-1}+pM_{n+1})\simeq \coker(e_n)$. We claim that $e_n$ is an embedding: indeed, if $e_n(x)=0$ for some $x\in k[[u]]^{\oplus s}\simeq M_{n+1}/M_n$, then the extension $[M_{n+1}]\in \Ext^1_{\mf S}(k[[u]]^{\oplus s},M_{n})$ restricted to $k[[u]]\cdot x\hookrightarrow k[[u]]^{\oplus s}$ reduces to $M_{n-1}$ and this way any lift of $x$ to $M_{n+1}$ is in fact killed by $p^n$, and so belongs to $M_n$. We get that $x\in M_{n+1}/M_n$ is actually 0. It remains to note that $e_n$ is an embedding, its cokernel is a necessarily a finite $u$-torsion module.
	
	This way we get that $L\coloneqq M[p^\infty]/pM[p^\infty]$ has an infinite filtration $0\hookrightarrow L_1 \hookrightarrow L_2 \hookrightarrow \ldots \subset L$ such that $L_i/L_{i-1}$ are non-zero finitely-generated $u^{\infty}$-torsion modules. Consider $u$-torsion $L_i[u]$; starting from some $i$, $L_i[u]\simeq L_{i+1}[u]$, otherwise $\dim_k L[u]=\infty $, and we are done (since $L[u]\hookrightarrow M/p[u]$, but by assumption $\Tor^1_{\mf S}(M,k)\simeq M/p[u]$ should be finite-dimensional). But then by \Cref{lem:about torsion module with bounded u-torsion} the $u$-adic Tate module $T_u(L)$ is non-zero. Since $T_u$ is a left-exact functor (in the sense of abelian categories) it follows that $T_u(M/p)$ is non-zero, which is a contradiction, since $M\otimes^{\mbb L}\mf S/p$ and, consequently, $M/p$ should be derived $u$-adically complete.
	
\end{proof}	
We will apply the above proposition in the following context. Namely, let $K/\mbb Q_p$ be a complete discretely valued extension of $\mbb Q_p$ with a perfect residue field $k$. Consider the $p$-completed algebraic closure $C=\widehat{\ol K}$ and its tilt $C^\flat$. Choice of a uniformizer $\pi$ and a collection $\pi^\flat \coloneqq(\ldots, \pi^{1/p^2},\pi^{1/p},\pi)\in \mc O_C^\flat$ of compatible $p^n$-roots of $\pi$ gives a unique $W(k)$-linear map $\mf S \ra W(C^\flat)$ which sends $u$ to the Teichm\"uller lift $[\pi^\flat]$.
\begin{cor}\label{cor:completed tensor with W over S}
	Let $M\in \Mod_{\mf S}$ be a derived $(p,u)$-complete module such that $\tau^{<0}(M\otimes^{\mbb L}_{\mf S}\mf S/(p,u))\in \Coh(k)$. Then the $p$-completed tensor product
	$
	M\widehat\otimes_{\mf S}W(C^\flat) 
	$ is concentrated in cohomological degree 0.
\end{cor}
\begin{proof}
Note that the $p$-completed tensor product functor $-\widehat\otimes_{\mf S}W(C^\flat)$ decomposes as the composition of  $-\widehat\otimes_{\mf S}\mf S\langle \frac{1}{u}\rangle$ and $-\widehat\otimes_{\mf S\langle \frac{1}{u}\rangle}W(C^\flat)$. Then by \Cref{prop:p-completed inversion of u is sometimes t-exact} $M\widehat\otimes_{\mf S}\mf S\langle \frac{1}{u}\rangle$ is concentrated in degree 0 (and is derived $p$-complete), while $-\widehat\otimes_{\mf S\langle \frac{1}{u}\rangle}W(C^\flat)$ (considered as a functor to $\mf S\langle \frac{1}{u}\rangle$-modules) can be rewritten as a $p$-completed direct sum (indeed, pick a basis $\{x_s\}$, $s\in S$, of $C^\flat$ over $k((u))\simeq \mf S\langle \frac{1}{u}\rangle/p$; then $W(C^\flat)$ as a $\mf S\langle \frac{1}{u}\rangle$-module is identified with the $p$-adic completion of  $\oplus_{s\in S}\mf S\langle \frac{1}{u}\rangle\cdot [x_s]\subset W(C^\flat)$). Thus we are done by \Cref{lem:p-completed product abelian module}.
\end{proof}

\paragraph{Solid $I$-completions and direct sums}

We also include another proof of (a weaker form of) \Cref{prop:pu-completed direct sum} which uses condensed mathematics and which was suggested to us by Peter Scholze (see \Cref{lem_discretness_of_cond_completion}). Some results below that precede the proof are then also used in the proof of $B_\crys$-comparison in \Cref{ssec: crystalline comparison}. Recall our notations from \Cref{constr:blacktriangle}: for a ring $R$ equipped with a finitely generated ideal $I$ and $M \in D(\Mod_R)$ we denote
$$M^\ccond_I \coloneqq  (\underline M)_I^\wedge \in {D(\Solid)}$$
the derived $I$-completion of $M$ in $D(\Solid)\subset D(\Cond(\Ab))$. $M^\ccond_I$ is naturally an $R^\bt_I$-module and naturally lands in the subcategory $D(R^\bt_I)_{I-\comp}$ of derived $I$-complete solid $R^\bt_I$-modules.

 Till the end of this section we will denote by $\ul{-}\colon  D(\Mod_{R})\ra D(\Solid)$ the natural functor extending the functor $\ul{-}\colon \Mod_{R} \ra \Solid$, associating to an $R$-module $M$ the solid group $\ul M$ represented by $M$ with discrete topology. We will completely ignore any other natural topologies that $M$ by accident can have.

\begin{prop}\label{prop_complited_direct_sums_t_exact_in_Solid}
Let $R$ be a ring equipped with a finitely generated ideal $I$. Consider the full subcategory $(D(\Mod_R))^\bt$ of $D(R^\bt_I)_{I_\comp}$ given by the essential image of the (solid) $I$-completion functor $(-)^\ccond_I\colon D(\Mod_R) \to D(R^\bt_I)_{I_\comp}$. Then for any set $S$ the $I$-completed direct sum functor
\begin{equation}\label{eq_solid_comp_direct_sum}
\widehat \bigoplus_S \colon (D(\Mod_R))^\bt \tto (D(\Mod_R))^\bt
\end{equation}
is $t$-exact.

\begin{proof}
For $M \in D(\Mod_R)$ we have
$$\underline{\Hom}_{D(\Solid)}\left(\prod_S \mathbb Z, M^\ccond_I\right) \simeq \prolim[n] \underline{\Hom}_{D(\Solid)}\left(\prod_S \mathbb Z, [\underline M/I^n]\right) \simeq \prolim[n] \bigoplus_S [\underline M/I^n] \simeq \widehat \bigoplus_S \underline M \simeq \widehat \bigoplus_S M^\ccond_I,$$
where in the second equivalence we have used the fact that for any complex of abelian groups $N$ the natural map
$$\bigoplus_S \underline N \tto \underline{\Hom}_{D(\Solid)}\left(\prod_S \mathbb Z, \underline N\right)$$
is an equivalence. This follows from the fact that by \cite[Corollary 6.1]{ClausenScholze_Condensed1} the product $\prod_S \mathbb Z$ is compact in $D(\Solid)$, hence the functor $\underline{\Hom}_{D(\Solid)}\left(\prod_S \mathbb Z, -\right)$ preserves small colimits, and the basic computation \cite[Proof of Proposition 5.7]{ClausenScholze_Condensed1}
$$\bigoplus_S \mathbb Z \xymatrix{\ar[r]_-\sim &} \underline{\Hom}_{D(\Solid)}\left(\prod_S \mathbb Z, \mathbb Z\right).$$
Since by \cite[Corollary 6.1]{ClausenScholze_Condensed1} the object $\prod_S \mathbb Z$ is projective, we deduce the desired $t$-exactness of the $I$-completed direct sum functor.
\end{proof}
\end{prop}

Under additional finiteness assumptions on $M \in D(\Mod_R)$ we can also deduce the (left) $t$-exactness of the completed direct sum $\widehat \bigoplus_S M$ in $D(\Mod_R)$ (as opposed to $(D(\Mod_R))^\bt\subset D(R^\bt_I)_{I-\comp}$).
\begin{prop}\label{prop:completed direct sum preserves t-structure}
Let $R$ be a Noetherian ring equipped with an ideal $I$ generated by a regular sequence. Let $M \in D^{\ge 0}(\Mod_R)$ be a derived $I$-complete complex of $R$-modules such that $\tau^{<0}([M/I])$ is coherent over $R/I$. Then for any set $S$ the $I$-completed direct sum $\widehat\bigoplus_S M$ is still concentrated in degrees $\ge 0$.

\begin{proof}
Since the evaluation on a point functor preserves limit and colimits we see that the $I$-competed direct sum functor $M \mapsto \widehat \bigoplus_S M$ factors as the composition
\begin{equation}\label{eq_factorization_of_completion}
\xymatrix{D(\Mod_R) \ar[r]^-{(-)^\ccond_I} & D(R^\bt_I)_{I-\comp} \ar[r]^-{\widehat\bigoplus_S} & D(R^\bt_I)_{I-\comp} \ar[r]^-{-(*)} & D(\Mod_R)_{I-\comp}.
}\end{equation}
Since by \Cref{prop_complited_direct_sums_t_exact_in_Solid} the completed direct sum functor \eqref{eq_solid_comp_direct_sum} is $t$-exact we deduce that if $M^\ccond_I \in D^{\ge 0}(\Solid)$ then so is the $I$-completed direct sum $\widehat \bigoplus_S M^\ccond_I$. Moreover, since the evaluation at a point is $t$-exact too, by factorization \eqref{eq_factorization_of_completion} for any $M$ as above the $I$-completed direct sum $\widehat \bigoplus_S M$ is also concentrated in non-negative cohomological degrees. So it is left to prove that if $M \in D^{\ge 0}(\Mod_R)$ is derived $I$-complete complex of $R$-modules such that $\tau^{<0}([M/I])$ is coherent over $R/I$ then $M^\ccond_I$ lies in $D^{\ge 0}(\Solid)$. This is the content of \Cref{lem_discretness_of_cond_completion} below.
\end{proof}
\end{prop}

Before proving \Cref{lem_discretness_of_cond_completion} we first establish a few auxiliary results useful in their own rights.
\begin{prop}\label{prop_sub_quot_of_disc}
Any subgroup and any quotient group of a discrete condensed abelian group is discrete.

\begin{proof}
Let $M$ be an abelian group and let $\underline M \to Q$ be a surjective map of condensed abelian groups. We claim that $Q$ is also discrete. To see this note that the counit of adjunction gives rise to a commutative square
$$\xymatrix{
\underline M \ar[d]^\sim \ar[r] & \underline{Q(*)} \ar[d]^q \\
\underline M \ar@{->>}[r] & Q. \\
}$$
Since the left vertical and the bottom horizontal arrows are surjective, the right vertical map $q$ is also surjective.

It is left to prove that $q$ is injective. In fact, we claim that for any condensed set $X$ the natural map $q\colon \underline{X(*)} \to X$ is injective. To see this let $S = \prolim S_\alpha$ be a pro-finite set. Unwinding the definitions we see that $q(S)$ is given by the natural map
$$\indlim[\alpha] X(S_\alpha) \tto X(S).$$
But note that for each $\alpha$ the projection $S \to S_\alpha$ admits a section (since the target is a disjoint union of points), hence the induced map $X(S_\alpha) \to X(S)$ is injective. The assertion then follows from the fact that a filtered colimit of monomorphisms of sets is a monomorphism.

Finally, let $N \to \underline M$ be an injection of condensed abelian groups. Then $N \simeq \ker(\underline M \surj \underline M/N)$. By the previous we know that $\underline M/N$ is discrete. Since the condensation functor $L \mapsto \underline L$ is fully faithful and exact we deduce that $N$ is also discrete.
\end{proof}
\end{prop}

\begin{lem}\label{lem_fp_condensed}
Let $R$ be a Noetherian ring equipped with an ideal $I$ generated by a regular sequence. Let $L \in D^{<\infty}(\Mod_R(\Solid))$ be an eventually connective $I$-complete complex of $R$-modules such that all cohomology of $[L/I]$ are finitely generated over $\underline{R/I}$. Then all cohomology of $L$ are finitely presentable $R^\ccond_I$-modules. In particular, for each $m\in \mathbb Z$ the natural map
$$\underline{H^m(L(*))} \sotimes_{\underline{R_I^\wedge}} R^\ccond_I \tto H^m(L)$$
is an isomorphism.

\begin{proof}
By shifting if necessary we can assume without loss of generality that $H^{>0}(L) \simeq 0$. By induction it is enough to prove that $H^0(L)$ is finitely presentable. Before proceed further note that by Mittag--Leffler (which applies in the setting of condensed abelian groups e.g.~by \cite[Proposition 4.2.8, Proposition 3.2.3(1), and Proposition 3.1.9]{BS_proetale}) $R^\ccond_I$ is concentrated in degree $0$. Let now $\alpha\colon (R^\ccond_I)^{\oplus m} \to L$ be a map inducing a surjection $(R/I)^{\oplus m} \surj H^0([L/I])$. We claim that $\alpha$ induces surjection on $H^0$. To see this it is enough to prove that $H^0(C) \simeq 0$, where $C$ denotes the cofiber of $\alpha$. By $I$-completeness of $C$ and Milnor's exact sequence we have
$$\xymatrix{0 \ar[r] & \prolim[n]{}^1 H^{-1}([C/I^n]) \ar[r] & H^0(C) \ar[r] & \prolim[n] H^0([C/I^n]) \ar[r] & 0.}$$
By induction on $n$ all $H^0([C/I^n]) \simeq 0$. Moreover, from $H^0([C/I]) \simeq 0$ it follows that for all $n$ the map $H^{-1}([C/I^n]) \to H^{-1}([C/I^{n-1}])$ is surjective. By Mittag--Leffler it follows that the $\lim^1$-term vanishes as well. This shows that $H^0(L)$ is finitely generated $R^\ccond_I$-module. To prove that $H^0(L)$ is finitely presented it is left to show that the kernel of $H^0(\alpha)\colon (R^\ccond_I)^{\oplus m} \surj H^0(L)$ is also finitely generated. But $H^0(\alpha)$ is covered by $H^{-1}(C)$, which is finitely generated by the previous argument (applied to $L^\prime = C[-1]$). Hence so is the kernel of $H^0(\alpha)$, and hence $H^0(L)$ is finitely presented.

The last assertion follows from the fact that the functor
$$\Mod_{R^\ccond_I}(\Solid)^\fg \tto \Mod_{R_I^\wedge}^\fg, \qquad N \mapsto N(*)$$
right adjoint to $\underline - \otimes_{\underline{R_I^\wedge}} R^\ccond_I$ is fully faithful. This reduces to the fact that the evaluation at a point functor preserves finite colimits and that
$$R_I^\wedge \simeq \End_{\Mod_{R_I^\wedge}}(R_I^\wedge) \xymatrix{\ar[r]^-\sim &} \End_{\Mod_{R^\ccond_I}(\Solid)}(R^\ccond_I).\qedhere$$
\end{proof}
\end{lem}

\begin{lem}\label{lem_discretness_of_cond_completion}
Let $R$ be a Noetherian ring equipped with an ideal $I$ generated by a regular sequence. Let $M \in D^{\ge 0}(\Mod_R)$ be a derived $I$-complete complex of $R$-modules such that $\tau^{<0}([M/I])$ is coherent over $R/I$. Then $M^\ccond_I$ lies in $D^{\ge 0}(\Solid)$.

\begin{proof}
First we claim that $[\tau^{<0}(M^\ccond_I)/I]$ is coherent over $R/I$. Indeed, from the fiber sequence
$$\xymatrix{[\tau^{<0}(M^\ccond_I)/I] \ar[r] & [M^\ccond_I/I] \simeq [M/I] \ar[r] & [\tau^{\ge 0}(M^\ccond_I)/I]}$$
for each $m<0$ we obtain the $5$-term exact sequence
$$\xymatrix{
H^{m-1}([M/I]) \ar[r]^-{f_{m-1}} & H^{m-1}([\tau^{\ge 0}(M^\ccond_I)/I]) \ar[r] & H^{m}([\tau^{<0}(M^\ccond_I)/I]) \ar[r] & H^m([M/I]) \ar[r]^-{f_m} & H^m([\tau^{\ge 0}(M^\ccond_I)/I]).
}$$
It follows that $H^{m}([\tau^{<0}(M^\ccond_I)/I])$ is an extension of the $\ker(f_m)$ by $\Im(f_{m-1})$. Since by assumption both $H^{m-1}([M/I])$ and $H^m([M/I])$ are discrete finitely generated $R/I$-modules so are $\Im(f_{m-1})$ and $\ker(f_m)$ (here we use the fact that a subgroup and a quotient of a discrete condensed abelian group is discrete, see \Cref{prop_sub_quot_of_disc}) and hence also so is $H^{m}([\tau^{<0}(M^\ccond_I)/I])$.

By applying \Cref{lem_fp_condensed} to $L = \tau^{<0}(M^\ccond_I)$ we deduce that for each $m<0$ the natural map
$$\underline{H^m(M^\ccond_I)(*)} \otimes_{\underline{R_I^\wedge}} R^\ccond_I \tto H^m(M^\ccond_I)$$
is an isomorphism. But $H^m(M^\ccond_I)(*) \simeq H^m(M^\ccond_I(*)) \simeq H^m(M) \simeq 0$, hence $\tau^{<0}(M^\ccond_I) \simeq 0$.
\end{proof}
\end{lem}

\paragraph{Complete tensor product.}
In this paragraph we show that for a ring $R$ equipped with a finitely generated ideal $I$ the $I$-completed tensor product functor $\cotimes_R$ is closely related to the solid tensor product functor $\sotimes_{R^\bt_I}$ in $\Mod_{R^\bt_I}(D(\Solid))$. We would like to mention that most proofs here were explained to us by Peter Scholze (and in particular are due to him).
\begin{prop}
Let $R$ be a ring with a finitely generated ideal $I \unlhd R$. Then for a pair of bounded above derived $I$-complete complexes of $R^\bt_I$-modules $X, Y \in D(R^\bt_I)^{<\infty}_{I_\comp}$ their tensor product $X \sotimes_{R^\ccond_I} Y$ is also derived $I$-complete.

\begin{proof}
Note that for any $f\in I$ the natural map $\mbb Z[x]\ra R$ sending $x$ to $f$ extends to a map $\mbb Z[[x]]\ra R^\bt_I$ (since $ \mbb Z[[x]]\simeq \mbb Z[x]^\bt_x$), where $\mbb Z[[x]]$ is considered with $x$-adic topology. Note that a complex of $R^\ccond_I$-modules is $I$-complete if and only if it is $(x)$-complete as a $\mathbb Z[[x]]$-module for each such ring morphism $\mathbb Z[[x]] \to R^\bt_I$ (that maps $x$ to an element of $I$). So it is enough to prove that $X \sotimes_{R^\ccond_I} Y$ is $(x)$-complete for each such homomorphism $\mathbb Z[[x]] \to R^\bt_I$. 

To see the latter note that the tensor product $X \sotimes_{R^\ccond_I} Y$ is equivalent to the geometric realization of the simplicial complex of solid $\mathbb Z[[x]]$-modules with terms given by
\begin{equation}\label{eq_solid_tensor_product_auto_complete}
X \sotimes_{\mathbb Z[[x]]} R^\ccond_I \sotimes_{\mathbb Z[[x]]} R^\ccond_I\sotimes_{\mathbb Z[[x]]} \ldots \sotimes_{\mathbb Z[[x]]} R^\ccond_I \sotimes_{\mathbb Z[[x]]} Y.
\end{equation}
Note that by the boundness assumption on $X$ and $Y$ for each $i\in \mathbb Z$ the $i$-cohomology module of the complex $X \sotimes_{R^\ccond_I} Y$ receives contribution only from finitely many terms of \eqref{eq_solid_tensor_product_auto_complete}. Since derived complete modules are closed under (co)kernels and extensions and since $X \sotimes_{R^\ccond_I} Y$ is $(x)$-complete if and only if all of its cohomology $H^i(X \sotimes_{R^\ccond_I} Y)$ are, it is enough to prove that all $X \sotimes_{\mathbb Z[[x]]} R^\ccond_I \sotimes_{\mathbb Z[[x]]}\ldots \sotimes_{\mathbb Z[[x]]}R^\ccond_I \sotimes_{\mathbb Z[[x]]} Y$ are $(x)$-complete. This reduces the assertion to the special case $R = \mathbb Z[[x]]$, $I = (x)$.

To treat this case let $\omega_1$ be the smallest uncountable cardinal. Since the category $D(\Mod_{\mathbb Z[[x]]}(\Solid))$ is compactly generated by objects of the form $\prod_I \mathbb Z[[x]]$ it is also $\omega_1$-compactly generated by objects of the form $\bigoplus \prod_I \mathbb Z[[x]]$, where the direct sum is at most countable. In particular, both $X, Y$ can be written as a completion of an $\omega_1$-filtered colimit of objects of the form $\bigoplus \prod_I \mathbb Z[[x]]$. But note that the $(x)$-completion functor, being a countable limit (hence $\omega_1$-small), commutes with $\omega_1$-filtered colimits. Hence both sides of the comparison map
$$X\sotimes_{\mathbb Z[[x]]} Y \tto X \cotimes_{\mathbb Z[[x]]}^\solid Y$$
preserve $\omega_1$-filtered colimits in both variables. It follows that it is enough to prove that $X\sotimes_{\mathbb Z[[x]]} Y$ is $(x)$-complete for $X = Y = \widehat \bigoplus_{\mathbb N} \prod_I \mathbb Z[[x]]$.

Concretely, the completed direct sum $\widehat \bigoplus_{\mathbb N} \prod_I \mathbb Z[[x]]$ can be described as
$$\colim_{f\colon \mathbb N \to \mathbb N, f(n) \to \infty} \prod_{n \in \mathbb N} x^{f(n)} \prod_I \mathbb Z[[x]] \xymatrix{\ar@{^(->}[r] &} \prod_{\mathbb N} \prod_{I} \mathbb Z[[x]],$$
where the colimit is taken over the diagram of functions $f\colon \mathbb N \to \mathbb N$ tending to $\infty$ when $n\to \infty$ partially ordered by the pointwise inequality. Indeed, the module above is classically $x$-adicly complete and $x$-torsion free, hence is derived $x$-complete. Moreover, the natural embedding $\bigoplus_{\mathbb Z} \prod_I \mathbb Z[[x]] \inj \prod_{\mathbb Z} \prod_I \mathbb Z[[x]]$ factors through a map
$$\alpha \colon \bigoplus_{\mathbb Z} \prod_I \mathbb Z[[x]] \tto \colim_{f\colon \mathbb N \to \mathbb N, f(n) \to \infty} \prod_{n \in \mathbb N} x^{f(n)} \prod_I \mathbb Z[[x]].$$
So it is enough to see that $\alpha/(x)$ is an isomorphism, which is clear. Under this identification we have
\begin{gather*}
\widehat \bigoplus_{\mathbb N} \prod_I \mathbb Z[[x]] \sotimes_{\mathbb Z[[x]]} \widehat \bigoplus_{\mathbb N} \prod_I \mathbb Z[[x]] \simeq \colim_{f,g\colon \mathbb N \to \mathbb N, f(n),g(n) \to \infty} \prod_{\mathbb N\times \mathbb N} x^{f(n) + g(m)} \prod_{I\times I} \mathbb Z[[x]],\\
\widehat \bigoplus_{\mathbb N} \prod_I \mathbb Z[[x]] \cotimes_{\mathbb Z[[x]]}^\solid \widehat \bigoplus_{\mathbb N} \prod_I \mathbb Z[[x]] \simeq \colim_{h\colon \mathbb N \times \mathbb N \to \mathbb N, h(n, m) \to \infty} \prod_{\mathbb N \times \mathbb N} x^{h(n,m)} \prod_{I\times I} \mathbb Z[[x]].
\end{gather*}
So it is enough to prove that among all functions $h\colon \mathbb N \times \mathbb N \to \mathbb N$ tending to $\infty$ the functions of the form $f(n) + g(m)$ are cofinal. Concretely, for each such $h$ we need to construct a pair of functions $f, g$ such that $h(n,m) \le f(n) + g(m)$ for all $n,m \in \mathbb N$. This inequality is satisfied e.g.~by
$$f(n) = g(n) := \max_{i,j\le n} h(i,j).\qedhere$$
\end{proof}
\end{prop}

Note that the category $D(\Mod_{R})_{I-\comp}$ is endowed with the natural symmetric monoidal structure given by the $I$-completed tensor product $-\cotimes_R-$ 
\begin{cor}\label{sotims_vs_cotimes}
Let $R$ be a ring and let $I \unlhd R$ be a finitely generated ideal. Then the functor
$$D^{<\infty}(\Mod_{R})_{I-\comp}\tto D(\Mod_{R_I^\bt}(\Solid)), \qquad M \mapsto M^\ccond_I$$
is symmetric monoidal.

\begin{proof}
For $M,N \in \DMod{R}_I^\wedge$ there is a natural comparison map $M^\ccond_I \sotimes_{R^\ccond_I} N^\ccond_I \to (M\cotimes_{R^\ccond_I} N)^\ccond_I$ which becomes an equivalence modulo $I$. Hence it is enough to prove that $M^\ccond_I \sotimes_{R^\ccond_I} N^\ccond_I$ is derived $I$-complete in $D(\Mod_{R^\ccond_I}(\Solid))$, which is a content of the previous proposition.
\end{proof}
\end{cor}

\begin{cor}\label{sotims_vs_cotimes_with_localization}
Let $R$ be a ring equipped with a finitely generated ideal $I \unlhd R$ and assume that $R$ is derived $I$-complete. Let $S$ be a localization of $R$. Then for each $M,N \in D^{<\infty}(\Mod_S(\Solid))$ with condensed structure coming from derived $I$-complete $R$-lattices (i.e.~$M \simeq (M_0)^\ccond_I \sotimes_{R^\ccond_I} S^\ccond_I$ for some derived $I$-complete $R$-module $M_0$ and similarly for $N$) their solid tensor product $M\sotimes_{S^\ccond_I} N$ is equivalent to $(M_0\cotimes_R N_0)^\ccond_I \sotimes_{R^\ccond_I} S^\ccond_I$.

\begin{proof}
This follows from the associativity of the tensor product and the previous corollary:
$$M\otimes^\solid_{S^\ccond_I} N \simeq ((M_0)^\ccond_I\sotimes_{R^\ccond_I} S)\sotimes_{S^\ccond_I} ((N_0)^\ccond_I\sotimes_{R^\ccond_I} S^\ccond_I) \simeq ((M_0)^\ccond_I \sotimes_{R^\ccond_I} (N_0)^\ccond_I)\sotimes_{R^\ccond_I} S^\ccond_I \simeq (M_0\cotimes_R N_0)^\ccond_I \sotimes_{R^\ccond_I} S^\ccond_I.\qedhere$$
\end{proof}
\end{cor}
\begin{ex}
Take $R = \mathbb Z_p$, $I = (p)$, $S = \mathbb Q_p$. Then for a pair of $\mathbb Q_p$-vector spaces $V, W$ with topology coming from $(p)$-complete $\mathbb Z_p$-lattices $V_0, W_0$ the solid tensor product of their condensation $V^\ccond_p \sotimes_{\mathbb Q_p} W^\ccond_p$ is equivalent to the "continuous" tensor product $(V_0\cotimes_{\mathbb Z_p} W_0)^\ccond_p\sotimes_{\mathbb Z_p} \mathbb Q_p$.
\end{ex}

We will also need the following variant of \Cref{sotims_vs_cotimes}.
\begin{cor}\label{sotimes_vs_cotimes}
Let $R$ be a ring and let $I \unlhd R$ be a finitely generated ideal. Let $N \in D^{<\infty}(\Mod_R)_{I-\comp}$ be a complex of $R$-modules such that the derived $I$-completed tensor product functor $-\cotimes_R N$ has bounded cohomological amplitude. Then for each $M \in D(\Mod_{R})_{I-\comp}$ the natural map
$$M^\ccond_I \sotimes_{R^\ccond_I} N^\ccond_I \tto (M\cotimes_R N)^\ccond_I$$
is an equivalence.

\begin{proof}
It is enough to show that $M^\ccond_I \sotimes_R N^\ccond_I$ is $I$-complete. Write $M$ as a colimit of its Whitehead tower $M \simeq \indlim \tau^{\le n} M$. By \Cref{sotims_vs_cotimes_with_localization} for each $n$ the tensor product $(\tau^{\le n} M)^\ccond_I \sotimes_{R^\ccond_I} N^\ccond_I$ is equivalent to $(\tau^{\le n}(M) \cotimes_R N)^\ccond_I$, in particular it is $I$-complete. Moreover, since by assumption both functors $-\cotimes_R N$ and $(-)^\ccond_I$ have bounded cohomological amplitude, for each $i\in \mathbb Z$ the cohomology $H^i(M^\ccond_I \sotimes_{R^\ccond_I} N^\ccond_I)$ is isomorphic to $H^i((\tau^{\le n}(M) \cotimes_R N)^\ccond_I)$ for some $n \gg 0$. It follows that all cohomology of $M^\ccond_I \sotimes_R N^\ccond_I$ are derived $I$-complete, hence so is $M^\ccond_I \sotimes_R N^\ccond_I$.
\end{proof}
\end{cor}

\section{Berthelot--Ogus isomorphism}
In this section we prove the Berthelot--Ogus isomorphism between rational crystalline and de Rham cohomology of smooth Artin stacks in a possibly ramified base field case. In the main part of the text we need this isomorphism to take into account the natural $p$-adic topology present on both sides, so we prove the equivalence in question between the crystalline and de Rham cohomology considered as objects of $D(\Cond(\Ab))$. The argument is a more or less straightforward adaptation of the one from \cite[Section 2]{Berthelot_Fisocrystals}. As above let $K$ be a complete discretely valued field of characteristic $0$ with the ring of integers $\mathcal O_K$ and perfect residue field $k$ of characteristic $p$.

\begin{prop}\label{Verschiebung_crys}
Let $S$ be a base scheme of characteristic $p$ and let $S\inj T$ be a pd-thickening. Then for each smooth quasi-compact quasi-separated Artin $S$-stack $\mstack X$ and non-negative integer $n$ there exists a functorial morphism
$$V_{\le n} \colon \tau^{\le n} R\Gamma_\crys(\mstack X/T) \tto \tau^{\le n} R\Gamma_\crys(\mstack X^{(1)}/T)$$
such that the composition of $V_n$ with the $n$-th truncation of the pullback along the relative Frobenius $\phi_{\mstack X} \colon \mstack X \to \mstack X^{(1)}$ (in both orders) is equivalent to the multiplication by $p^n$:
$$\tau^{\le n}(\phi_{\mstack X}^*) \circ V_{\le n} \simeq V_{\le n} \circ \tau^{\le n}(\phi_{\mstack X}^*) \simeq p^n.$$

\begin{proof}
Since homotopy limits are left $t$-exact one has
$$\tau^{\le n} R\Gamma_\crys(-/T) \simeq \tau^{\le n}\left(\Ran_{\Aff_{S}^{\sm, \op}}^{\Stk_{S}^{\sm, \op}} \tau^{\le n} R\Gamma_\crys(-/T)\right),$$
and similarly for $\tau^{\le n} R\Gamma_\crys(-^{(1)}/T)$. It follows that it is enough to construct a functorial map $V_n$ for affine schemes over $S$. This is a content of \cite[Theorem 8.20]{BO_NotesOnCrys}.
\end{proof}
\end{prop}

Now recall the definition of condensation of crystalline cohomology \Cref{def_cond_crys}. With these notations we have:
\begin{cor}\label{Frob_invariance_rational_crys}
Let $S$ be a base scheme of characteristic $p$ and let $S\inj T$ be a pd-thickening, where $T$ is a $p$-adic formal scheme. Then for a smooth quasi-compact quasi-separated Artin $S$-stack $\mstack X$ the pullback along the relative Frobenius $\phi_{\mstack X} \colon \mstack X \to \mstack X^{(1)}$ induces an equivalence of rational crystalline cohomology
$$R\Gamma_\crys^\ccond(\mstack X^{(1)}/T)\otimes_{\mathbb Z} \mathbb Q \xymatrix{\ar[r]^-{\phi_{\mstack X}^*}_-\sim & } R\Gamma_\crys^\ccond(\mstack X/T)\otimes_{\mathbb Z} \mathbb Q.$$

\begin{proof}
Since the embedding $D(\Ab) \inj D(\Solid)$ is $t$-exact, since the derived $p$-completion functor
$$(-)_p^\wedge \colon D(\Solid) \tto D(\Solid)$$
is $t$-exact up to a shift, and since $\mathbb Q$ is flat in solid abelian groups we have that
$$R\Gamma_\crys^\ccond(-/T)\otimes_{\mathbb Z}\mathbb Q \simeq \indlim[n] \left((\tau^{\le n} R\Gamma_\crys(-/T))_p^\ccond \otimes_{\mathbb Z}\mathbb Q\right).$$
So it is enough to show that for all $n\in \mathbb Z_{\ge 0}$ the map
$$(\tau^{\le n}(\phi_{\mstack X}^*))_p^\ccond \otimes_{\mathbb Z}\mathbb Q \colon (\tau^{\le n} R\Gamma_\crys(\mstack X^{(1)}/T))_p^\ccond \otimes_{\mathbb Z}\mathbb Q \tto (\tau^{\le n} R\Gamma_\crys(\mstack X/T))_p^\ccond \otimes_{\mathbb Z}\mathbb Q$$
is an equivalence. But by applying the $p$-completion functor $(-)_p^\ccond\colon D(\Ab) \to D(\Solid)$ to the map $V_{\le n}$ from the previous proposition we see that $(\tau^{\le n}(\phi_{\mstack X}^*))_p^\ccond$ admits an inverse up to multiplication by $p^n$, hence rationally becomes an equivalence.
\end{proof}
\end{cor}
\begin{prop}[Nil-functoriality of the rational crystalline cohomology]\label{nil_invariance_rational_crys}
Let $i\colon S_0 \inj S$ be a pro-nilpotent thickening of characteristic $p$ schemes and let $S \inj T$ be a pd-thickening, where $T$ is a $p$-adic formal scheme. Then there exists a functor
$$R\Gamma_\crys^{\ccond, \prime}(-/T)\otimes\mathbb Q \colon \Stk_{/S_0}^{sm, \op} \tto \Mod_{\mathbb Q}(D(\Solid))$$
making the following diagram commutative
$$\xymatrix{
\Stk_{/S}^{sm, \op} \ar[d]_{-\times_{S} S_0} \ar[rrr]^-{R\Gamma_\crys^\ccond(-/T)\otimes_{\mathbb Z} \mathbb Q} &&& \Mod_{\mathbb Q}(D(\Solid)). \\
\Stk_{/S_0}^{sm, \op} \ar@{-->}[rrru]_(0.4){\qquad R\Gamma_\crys^{\ccond,\prime}(-/T)\otimes\mathbb Q}
}$$

\begin{proof}
By assumption $S \simeq \prolim S_\alpha$, where $S_\alpha$ are nilpotent thickenings of $S_0$. First we claim that for each $\alpha$ there is a functor
$$R\Gamma_{\crys, \alpha}^{\ccond,\prime}(-/T)\otimes\mathbb Q\colon \Stk_{/S_0}^{sm, \op} \tto \Mod_{\mathbb Q}(D(\Solid))$$
making the following diagram commutative
$$\xymatrix{
\Stk_{/S_\alpha}^{sm, \op} \ar[rr]^-{-\times_{S_\alpha} S} \ar[d]_{-\times_{S_\alpha} S_0} && \Stk_{/S}^{sm, \op} \ar[rrr]^-{R\Gamma_\crys^\ccond(-/T)\otimes_{\mathbb Z} \mathbb Q} &&& \Mod_{\mathbb Q}(D(\Solid)). \\
\Stk_{/S_0}^{sm, \op} \ar@{-->}[rrrrru]_(0.4){\qquad R\Gamma_{\crys,\alpha}^{\ccond,\prime}(-/T)\otimes\mathbb Q}
}$$

To construct it let $\mathcal I$ be the sheaf of ideals of the embedding $S_0 \inj S_\alpha$. Since $\mathcal I$ is nilpotent there exists an integer $n$ large enough so that $\mathcal I^{p^n} = 0$. In particular, there exists a dashed arrow $\rho_{\alpha, n}\colon S_\alpha \to S_0$ making the diagram below commutative
$$\xymatrix{
S_0 \ar[rr]^-{F_{S_0}^n} \ar@{^(->}[d] && S_0 \ar@{^(->}[d] \\
S_\alpha \ar[rr]^-{F_{S_\alpha}^n} \ar@{-->}[rru]^-{\rho_{\alpha, n}} && S_\alpha,
}$$
where $F_{S_o}$ and $F_{S_\alpha}$ are absolute Frobenii of $S_0$ and $S_\alpha$ respectively. We then define $R\Gamma_{\crys, \alpha, n}^{\ccond,\prime}(-/T)\otimes \mathbb Q$
as the composition
$$\xymatrix{\Stk_{/S_0}^{\sm, \op} \ar[rr]^-{-\times_{S_0, \rho_{\alpha, n}} S_\alpha} && \Stk_{/S_\alpha}^{\sm, \op} \ar[rr]^-{-\times_{S_\alpha} S} && \Stk_{/S}^{\sm, \op} \ar[rr]^-{R\Gamma_\crys^\ccond(-/T)} && \Mod_{\mathbb Q}(D(\Solid)).}$$
We claim that there is a natural equivalence
$$R\Gamma_{\crys, \alpha, n}^{\ccond, \prime}(-\times_{S_\alpha} S_0/T)\otimes \mathbb Q \simeq R\Gamma_\crys^\ccond(-\times_{S_\alpha} S/T)\otimes_{\mathbb Z} \mathbb Q.$$
In order to construct it note that by construction the composition of the pullbacks along the embedding $S_0 \inj S_\alpha$, $\rho_{\alpha, n}$ and the projection $S \to S_\alpha$ sends an $S_\alpha$-stack $X$ to the $n$-th relative (over $S$) Frobenius twist of $X\times_{S_\alpha} S$. The $n$-th relative Frobenius then induces a natural morphism
$$\phi_{\alpha, n, X}\colon X\times_{S_\alpha} S \tto (X\times_{S_\alpha} S)^{(n)} \simeq ((X\times_{S_\alpha} S_0) \times_{S_0, \rho_{\alpha, n}} S_\alpha) \times_{S_\alpha} S.$$
By applying the pullback in crystalline cohomology $R\Gamma_\crys^\ccond(-/T)\otimes_{\mathbb Z} \mathbb Q$ to $\phi_{\alpha, n, -}$ we then obtain a natural transformation
$$R\Gamma_\crys^\ccond(-\times_{S_\alpha} S/T) \otimes_{\mathbb Z} \mathbb Q \tto R\Gamma_{\crys, \alpha, n}^{\ccond, \prime}(-\times_{S_{\alpha}} S_0/T)\otimes \mathbb Q$$
which is an equivalence by \Cref{Frob_invariance_rational_crys}.

Moreover, since for any integer $m \ge n$ we have $\rho_{\alpha, m} \simeq \rho_{\alpha, n} \circ F_{S_\alpha}^{m-n}$ it follows that
$$(-\times_{S_0, \rho_{\alpha, m}} S_\alpha)\times_{S_\alpha} S \simeq ((-\times_{S_0, \rho_{\alpha, n}} S_\alpha)\times_{S_\alpha} S)^{(m-n)},$$
and the $(m-n)$-th relative Frobenius and pullback in crystalline cohomology induce a natural transformation of functors
$$R\Gamma_{\crys, \alpha, m}^{\ccond, \prime}(-/T)\otimes \mathbb Q \tto R\Gamma_{\crys, \alpha, n}^{\ccond, \prime}(-/T)\otimes \mathbb Q$$
which is an equivalence again by \Cref{Frob_invariance_rational_crys}. We then define
$$R\Gamma_{\crys, \alpha}^{\ccond, \prime}(-/T)\otimes \mathbb Q := \prolim R\Gamma_{\crys, \alpha, n}^{\ccond, \prime}(-/T)\otimes \mathbb Q$$
which is naturally equivalent to $R\Gamma_{\crys, \alpha, n}^{\ccond, \prime}(-/T)\otimes \mathbb Q$ for any $n$.

Finally, since for a morphism $f\colon S_\alpha \to S_\beta$ and large enough $n$ we have $\rho_{\alpha, n} \simeq \rho_{\beta, n} \circ f$ there is a natural equivalence $R\Gamma_{\crys, \alpha}^{\ccond,\prime}(-/T)\otimes\mathbb Q \simeq R\Gamma_{\crys, \beta}^{\ccond,\prime}(-/T)\otimes\mathbb Q$ under which the equivalence 
$$R\Gamma_{\crys,\beta}^{\ccond,\prime}(-\times_{S_\beta} S_0 /T)\otimes\mathbb Q \simeq R\Gamma_\crys^\ccond(-\times_{S_\beta} S/T)\otimes_{\mathbb Z} \mathbb Q$$
identifies with the restriction along the pullback $\Stk_{/S_\beta}^{sm, \op} \to \Stk_{/S_\alpha}^{sm, \op}$ of the corresponding equivalence for $R\Gamma_{\crys,\alpha}^{\ccond,\prime}(-\times_{S_\alpha} S_0 /T)$. We conclude passing to the colimit over $\alpha$ and using the fact that by \cite[Theorem 2.1.13]{KubrakPrikhodko_HdR} the pullbacks induce an equivalence $\indlim \Stk_{/S_\alpha}^{sm, \op} \areq \Stk_{/S}^{sm, \op}$.
\end{proof}
\end{prop}

\begin{prop}[Berthelot--Ogus comparison]\label{prop:Berthelot-Ogus de Rham case}
Let $\mstack X$ be a smooth quasi-compact quasi-separated Artin stack over $\mathcal O_K$. Then there exists a functorial equivalence
$$R\Gamma_\crys^\ccond(\mstack X_k / W(k))\sotimes_{W(k)} K \simeq R\Gamma_\dR^\ccond(\widehat{\mstack X} / \mathcal O_K)\sotimes_{\mathcal O_K} K.$$

\begin{proof}
The ideal $(p)$ in $\mathcal O_K$ admits a (unique) pd-structure and the inclusion $W(k) \inj \mathcal O_K$ is a pd-morphism. It follows from the base change for crystalline cohomology \Cref{basechange_crys} below that we have a natural equivalence
\begin{equation}\label{BerthelotOgus_eq1}
R\Gamma_\crys^\ccond(\mstack X_k / W(k)) \otimes_{W(k)} \mathcal O_K \simeq R\Gamma_\crys^\ccond((\mstack X_{k} \otimes_k \mathcal O_k/p) / \mathcal O_K).
\end{equation}
On the other hand, since $\mstack X$ is a smooth lifting of $\mstack X_{\mathcal O_k/p}$ to $\mathcal O_K$ by smooth descent for crystalline and de Rham cohomology and \cite[Corollary 7.4]{BO_NotesOnCrys} in the affine case we have
\begin{equation}\label{BerthelotOgus_eq2}
R\Gamma_\crys^\ccond(\mstack X_{\mathcal O_K/p} / \mathcal O_K) \simeq R\Gamma_\dR^\ccond(\widehat{\mstack X}	/\mathcal O_K).
\end{equation}
Moreover, since both $\mstack X_{k} \otimes_k \mathcal O_k/p$ and $\mstack X_{\mathcal O_K/p}$ are nil-extensions of $\mstack X_k$ we have
$$R\Gamma_\crys^\ccond((\mstack X_{k} \otimes_k \mathcal O_k/p) / \mathcal O_K) \otimes_{\mathcal O_K} K \simeq R\Gamma_\crys^\ccond(\mstack X_{\mathcal O_K/p} / \mathcal O_K) \otimes_{\mathcal O_K} K$$
by \Cref{nil_invariance_rational_crys}. Combining the last equivalence with the rationalizations of equivalences \eqref{BerthelotOgus_eq1} and \eqref{BerthelotOgus_eq2} we deduce the desired comparison.
\end{proof}
\end{prop}
\begin{prop}\label{basechange_crys}
Let $A \to A^\prime$ be a ring morphism of finite $p$-complete Tor amplitude. Then for a smooth quasi-compact quasi-separated Artin stack $\mstack X$ over $A/(p)$ the natural maps
\begin{gather*}
R\Gamma_\crys(\mstack X/A)\cotimes_A A^\prime \tto R\Gamma_\crys(\mstack X^\prime /A^\prime),\\
R\Gamma_\crys^\ccond(\mstack X/A)\sotimes_{A^\ccond} A^{\prime, \ccond} \tto R\Gamma_\crys^\ccond(\mstack X^\prime /A^\prime)
\end{gather*}
are equivalences, where $\mstack X^\prime$ denotes the base change $\mstack X \times_{\Spec A/(p)} \Spec A^\prime/(p)$.

\begin{proof}
Both sides of the first arrow are derived $p$-complete, hence it is enough to prove the base
change modulo $p$. In this case both parts identify with de Rham cohomology for which the base change is proved e.g.~in \cite[Proposition 1.1.8]{KubrakPrikhodko_HdR}. By construction and \Cref{sotimes_vs_cotimes} the second map is a condensation of the first one, hence is also an equivalence.
\end{proof}
\end{prop}

In the main part of the text we also use the following comparison of crystalline cohomology:
\begin{prop}\label{crys_O_C_W_k}
Let $\mstack X$ be a smooth quasi-compact quasi-separated Artin stack over $\mathcal O_K$. There is a natural equivalence
$$R\Gamma_\crys^\ccond(\mstack X_k / W(k))\sotimes_{W(k)} B_\crys^\ccond \simeq R\Gamma_\crys^\ccond(\mstack X_{\mathcal O_{\mathbb C_p}/p} / A_\crys)[\tfrac 1 p].$$

\begin{proof}
By the base change for crystalline cohomology we have the natural equivalence
$$R\Gamma_\crys^\ccond(\mstack X_k/W(k))\sotimes_{W(k)^\ccond} W(\overline k)^\ccond \simeq R\Gamma_\crys^\ccond(\mstack X_{\overline k}/W(\overline k)).$$
Hence it is enough to prove a more general statement: let $\mstack Y$ be a smooth quasi-compact quasi-separated Artin stack over $\mathcal O_{\mathbb C_p}/p$. Then there is a functorial equivalence
\begin{equation}\label{eq_crys_O_C_W_k_1}
R\Gamma_\crys^\ccond(\mstack Y_{\overline k}/ W(\overline k)) \sotimes_{W(\overline k)^\ccond} B_\crys^\ccond \simeq R\Gamma_\crys^\ccond(\mstack Y/A_\crys)[\tfrac 1 p].
\end{equation}
Note that the morphism $\mathcal O_{\mathbb C_p}/(p) \to \overline k$ is an ind-nilpotent extension: indeed, $\mathcal O_{\mathbb C_p}/(p) \simeq \mathcal O_{\overline{\mathbb Q_p}}/(p)$ and $\mathcal O_{\overline{\mathbb Q_p}} \simeq \indlim \mathcal O_L$, where $L$ runs over the poset of finite extensions of the maximal unramified extension $K^{\mathrm{nr}}$ of $K$, and for each such $L$ the ring $\mathcal O_L/(p)$ is a nilpotent extension of $\mathcal O_{K^{\mathrm{nr}}}/(p) \simeq \overline k$. It follows from \Cref{nil_invariance_rational_crys} that there is a natural equivalence
\begin{equation}\label{eq_crys_O_C_W_k_2}
R\Gamma_\crys^\ccond((\mstack Y_{\overline k} \otimes_{\overline k} \mathcal O_{\mathbb C_p}/p)/A_\crys)[\tfrac 1 p] \simeq R\Gamma_\crys^\ccond(\mstack Y/A_\crys)[\tfrac 1 p].
\end{equation}
Moreover, by the base change for crystalline cohomology again we have
\begin{equation}\label{eq_crys_O_C_W_k_3}
R\Gamma_\crys^\ccond(\mstack Y_{\overline k}/ W(\overline k)) \sotimes_{W(\overline k)^\ccond} A_\crys^\ccond \simeq R\Gamma_\crys^\ccond((\mstack Y_{\overline k} \otimes_{\overline k} \mathcal O_{\mathbb C_p}/p)/A_\crys).
\end{equation}
By combining \eqref{eq_crys_O_C_W_k_2} with \eqref{eq_crys_O_C_W_k_3} we deduce \eqref{eq_crys_O_C_W_k_1}.
\end{proof}
\end{prop}




%
%
%

\theoremstyle{plain}
\newtheorem{theorem}[equation]{Theorem}
\newtheorem{conjecture}[equation]{Conjecture}
\newtheorem{proposition}[equation]{Proposition}
\newtheorem{lemma}[equation]{Lemma}
\newtheorem{corollary}[equation]{Corollary}
\newtheorem{guess}[equation]{Guess}
\newtheorem{fact}[equation]{Fact}
\theoremstyle{definition}
\newtheorem{definition}[equation]{Definition}
\newtheorem{hypothesis}[equation]{Hypothesis}
\newtheorem{example}[equation]{Example}
\newtheorem{situation}[equation]{Situation}
\newtheorem{remark}[equation]{Remark}

\section{Cohomological descent, de Rham comparison, and local acyclicity for some singular schemes (by Haoyang Guo)}\label{Appendix: Cohomological descent, de Rham comparison, and resolution of singularities}

In this appendix, we recall cohomological descent for \'etale cohomology of an algebraic variety or a rigid space, and generalize the de Rham comparison in \cite{DiaoKaiWenLiuZhu_LogarithmicRH} to singular varieties.
As an application, we give a small extension of \Cref{thm:even mainer theorem} for some schemes over $\mathcal{O}_K$ that are singular within the generic fiber in \Cref{local acyc for sing}.

We fix a complete discretely valued extension $K$ of $\mathbb{Q}_p$, and its complete algebraic closure $C$ throughout the section.

\subsection{Cohomological descent for \'etale cohomology}
Denote $\mathrm{Var}_K$ to be the category of finite type schemes over $K$, and $\mathrm{Rig}_K$ to be the category of qcqs rigid spaces over $K$.
Let $X$ be either in $\mathrm{Var}_K$ or $\mathrm{Rig}_K$.
\begin{definition}\label{def Ros}
	Let $a\colon X_\bullet \to X$ be a simplicial object with an augmentation over $X$, in either $\mathrm{Var}_K$ or $\mathrm{Rig}_K$.
	It is called a \emph{simplicial resolution of singularities of $X$}\footnote{The name is non-standard.} if it is a proper hypercovering in the sense of \cite[Tag 0DHI]{StacksProject}, such that each $X_{n+1}\to (\mathrm{cosk_n sk_n} X_\bullet)_{n+1}$ is a resolution of singularity as in \cite[Theorem 1.1]{BM08}.
\end{definition}
\begin{remark}
	By definition, a simplicial resolution of singularities for an algebraic variety is preserved under a field extension.
\end{remark}
\begin{remark}\label{embedded ros}
	When $X$ is an open subspace inside of a proper algebraic/rigid variety $\ol X$, with a Zariski closed complement $D$, one can apply the embedded resolution of singularities on the pair $(\ol X, D)$ to get a pair of simplicial spaces $(\ol X_\bullet, D_\bullet)\to (\ol X, D)$, so that each $D_n$ is a simple normal crossing divisor in the smooth variety $\ol X_n$, and both $\ol X_\bullet\to X$ and the simplicial open subset $\ol X_\bullet\backslash D_\bullet \to X$ are simplicial resolution of singularities.
\end{remark}
A simplicial resolution of singularities provides a tool of computing cohomology of singular spaces using that of smooth ones, by the following result.
\begin{theorem}\label{Ros and coh}
	Let $a\colon X_\bullet \to X$ be a simplicial resolution of singularities in either $\mathrm{Var}_K$ or $\mathrm{Rig}_K$.
	Then we have
	\[
	R\Gamma_\et(X_C, \mathbb{Q}_p)\simeq R\Gamma_\et (X_{\bullet C}, \mathbb{Q}_p),
	\]
	where the latter is computed as the homotopy limit  $R\lim_{[n]\in \Delta^\op}R\Gamma_\et (X_{n C}, \mathbb{Q}_p)$ over the simplicial diagram $\Delta^{\op}$.
\end{theorem}
\begin{proof}
For algebraic varieties, this is proved for example in \cite[Theorem 7.7]{Con}, where the only non-formal input is the proper base change theorem for a pullback from a point.
For the case of rigid space, we apply the same proof and the proper base change theorem for rigid spaces as in \cite[Theorem 4.4.1.(b)]{Huber_adicSpaces}.
\end{proof}

\subsection{Du Bois complex}
We then recall the notion of the (Deligne)-Du Bois complex, introduced by Deligne and studied in Du Bois's thesis.
The idea is to generalize the de Rham complex to the non-smooth setting, using cohomological descent and resolution of singularities.
\begin{definition}
	Let $X$ be either in $\mathrm{Var}_K$ or $\mathrm{Rig}_K$, and let $a\colon X_\bullet\to X$ be a simplicial resolution of singularities.
	The \emph{Du Bois complex} of $X$, denoted as $\underline{\Omega}_{X/K}^\bullet$ is an object in the $\mathbb{N}^\op$-filtered derived category of sheaves of $K$-modules over $X$, defined as the homotopy limit of the Hodge-filtered de Rham complexes as below
	\[
	\underline{\Omega}_{X/K}^\bullet:=R\varprojlim_{[n]\in \Delta^\op} Ra_{n *}\Omega_{X_n/K}^\bullet.
	\]
\end{definition}
\begin{remark}\label{DB rmk}
	It can shown that the construction is independent of the choice of the simplicial resolution of singularities $a\colon X_\bullet\to X$.
	See for example \cite[Theorem 7.22]{PS08} for algebraic varieties, and \cite[Section 5]{Guo19} for rigid spaces (with a slightly different notation as $\Omega_{\mathrm{\acute{e}h}}^i$).
	In particular, when $X$ is smooth itself, the Du Bois complex is filtered isomorphic to its de Rham complex with Hodge filtration.
\end{remark}
\begin{remark}
	When $K$ is replaced by the field of complex numbers, it can be shown that cohomology of the Du Bois complex for a complex algebraic variety is isomorphic to singular cohomology of $X^{\mathrm{an}}$, and the induced filtration from  $\underline{\Omega}_{X/K}^\bullet$ is the Hodge filtration on each singular cohomology group.
	See for example \cite[Section 7.3]{PS08}.
\end{remark}
\begin{remark}
	By construction, the $i$-th graded piece of the Du Bois complex, denoted as $\underline{\Omega}_{X/K}^i[-i]$, is a bounded below complex with coherent cohomology over $\mathcal{O}_X$.
	Using either classical Hodge theory (\cite[Section 7.3]{PS08}) or $p$-adic Hodge theory (\cite[Theorem 1.2.2]{Guo19}, or an upcoming work of Bhatt-Lurie on $p$-adic Riemann-Hilbert correspondence), it  can be shown that the coherent complex $\underline{\Omega}_{X/K}^i$ lives in cohomological degree $[0,\dim (X)-i]$. 
\end{remark}
We also provide a useful computation tool for Du Bois complex using resolution of singularities.
\begin{proposition}\label{DB triangle}
	Let $X$ be a finite type variety over $K$, and let $i\colon Z\to X$ be a Zariski closed subspace of $X$.
	Denote by $\pi\colon X'\to X$ the blowup of $X$ at $Z$, and let $Z'$ be the preimage of $Z$.
	Then for each $i\in \mathbb{N}$, there is a natural distinguished triangle as below
	\[
	\underline{\Omega}_{X/K}^i \longrightarrow R\pi_* \underline{\Omega}_{X'/K}^i \oplus i_*\underline{\Omega}_{Z/K}^i \longrightarrow R\pi_*\underline{\Omega}_{Z/K}^i.
	\]
\end{proposition}
\begin{proof}
	One can prove this for example following a filtered version of \cite[Example 7.25]{PS08}, or \cite[Proposition 5.0.4]{Guo19}.
\end{proof}

\subsection{de Rham comparison for singular varieties}
In this subsection, we state the de Rham comparison for non-smooth algebraic varieties.
\begin{theorem}
	Let $X$ be in $\mathrm{Var}_K$.
	Then there exists a natural isomorphism of complexes of $\mathrm{B_{dR}}$-modules
	\[
	R\Gamma_\et(X_C,\mathbb{Q}_p)\otimes_{\mathbb{Q}_p} \mathrm{B_{dR}} \simeq R\Gamma(X, \underline{\Omega}_{X/K}^\bullet)\otimes_K \mathrm{B_{dR}}.
	\]
\end{theorem}
\begin{proof}
	Let $X \to \ol X$ be a compatification with the complement $D$, and let $a\colon X_\bullet\to X$ be a simplicial  resolution of singularities, extended to an embedded resolution for $\ol a\colon \ol X_\bullet\to \ol X$ as in Remark \ref{embedded ros}.
	By the result of \cite{DiaoKaiWenLiuZhu_LogarithmicRH} in Construction \ref{constr:map theta} and Example \ref{ex: log-BdR vs usual dR}, for each $n\in \mathbb{N}$ we have a natural isomorphism
	\[
	R\Gamma_\et(X_{n C}, \mathbb{Q}_p)\otimes_{\mathbb{Q}_p} \mathrm{B_{dR}} \simeq R\Gamma_{\mathrm{dR}} (X_n/K)\otimes_K \mathrm{B_{dR}}.
	\]
	On the other hand, by cohomological descent as in Theorem \ref{Ros and coh}, we have
	\[
	R\Gamma_\et(X_C, \mathbb{Q}_p)\simeq R\varprojlim_{[n]\in \Delta^\op} R\Gamma_\et(X_{n C}, \mathbb{Q}_p).
	\]
	Combine the above two isomorphisms, we obtain the following formula
	\[
	R\Gamma_\et(X_C, \mathbb{Q}_p) \otimes_{\mathbb{Q}_p} \mathrm{B_{dR}} \simeq (R\varprojlim_{[n]\in \Delta^\op} R\Gamma_{\mathrm{dR}} (X_n/K))\otimes_K \mathrm{B_{dR}},
	\]
	where we implicitly uses the convergence of the spectral sequence associated to this simplicial diagram (\cite[Tag 0DHP]{StacksProject}) to commute the homotopy limits with the tensor product with $\mathrm{B_{dR}}$.
	Finally, using the definition of the Du Bois complex, the last term above can be rewritten as $R\Gamma(X,\underline{\Omega}_{X/K}^\bullet) \otimes_K \mathrm{B_{dR}}$, thus we get the formula as below
	\[
	R\Gamma_\et(X_C, \mathbb{Q}_p) \otimes_{\mathbb{Q}_p} \mathrm{B_{dR}} \simeq R\Gamma(X,\underline{\Omega}_{X/K}^\bullet) \otimes_K \mathrm{B_{dR}}.
	\]

\end{proof}
As a consequence, we obtain the following result, which was first proved by Kisin using de Jong's alteration.
\begin{corollary}
	Let $X$ be in $\mathrm{Var}_K$.
	Then for each $i\in \mathbb{N}$, its \'etale cohomology $H^i_\et(X_C, \mathbb{Q}_p)$ is a de Rham Galois representation of $G_K$.
\end{corollary}

\subsection{Hodge properness and Hodge proper resolutions}
We consider the $d$-Hodge properness for spaces over $K$ that are not necessarily smooth in this subsection.
\begin{definition}
	Let $X$ be either in $\mathrm{Var}_K$ or $\mathrm{Rig}_K$.
	It is called \emph{$d$-Hodge proper} if for each integer $0\leq i\leq d$, its cohomology of (shifted) graded piece of Du Bois complex $R\Gamma(X,\underline{\Omega}_{X/K}^i)$ is a perfect $K$-linear complex.
\end{definition}
\begin{remark}\label{d Hodge prop rmk}
	When $X$ is a finite type algebraic variety that is smooth over $K$, the notion coincides with the definition as in Definition \ref{def:Hodge-proper up to degree d}, following Remark \ref{DB rmk}.
\end{remark}
Our goal is to prove the following result on $d$-Hodge propernes.
\begin{theorem}\label{d Hodge prop}
	Let $X$ be in $\mathrm{Var}_K$ or $\mathrm{Rig}_K$, and let  $X_{\mathrm{sing}}$ be its singular locus.
	Assume $X$ admits a Cohen-Macaulay \footnote{The Cohen-Macaulay condition for rigid spaces can be found for example in \cite{Ber93}, and is defined as local rings being Cohen-Macaulay.} compactification $X\to \ol X$ with the complement $D$ of codimension $d+2$, such that the closure of $X_{\mathrm{sing}}$ in $\ol X$ is $Z$ itself.
	Then $X$ is $d$-Hodge proper.
\end{theorem}

\begin{proof}
	We first notice that when $X$ is smooth admitting a compactification as in the hypothesis, this is the commutative algebra statement as in \Cref{lem:example of d-Hodge-proper} (together with Remark \ref{DB rmk}).

	In general, we reduce to the smooth case as follows.
	Let $Z$ be a Zariski closed subspace of $X$, inside of $X_{\mathrm{sing}}$.
	By assumption, its closure in $\ol X$ is equal to itself, and we can form the following diagram of varieties, with each square being cartesian
	\[
	\xymatrix{ Z' \ar[r] \ar[d] & X'=\mathrm{Bl}_{X}(Z) \ar[r] \ar[d]^\pi & \ol X'=\mathrm{Bl}_{X'}(Z) \ar[d]^{\ol f} & D \ar@{=}[d] \ar[l] \\
		Z \ar[r] & X \ar[r] & \ol X & D\ar[l].}
	\]
	In particular, as $Z$ is a closed subscheme of $\ol X$, its preimage $Z'=Z\times_X X'$ in $X'$ is also a closed subscheme of $\ol X'$.
	Moreover, both $Z$ and $Z'$ are proper and hence Hodge proper over $K$.
	
	To proceed, we use the resolution of singularities in \cite[Theorem 1.1.(2')]{BM08}, which says that there exists a finite sequence of blowups
	\[
	X=X_0 \longleftarrow X_1 \longleftarrow \cdots \longleftarrow X_t,
	\]
	such that each $X_{i+1} \to X_{i}$ is a blowup at a Zariski closed subspace $C_{i}$ of $X_{i}$ for $C_{i}\subseteq X_{i, \mathrm{sing}}$, and $X_t$ is smooth over $K$.
	In particular, inductively as in the paragraph above, the closure of $C_{i+1}$ in $\ol X_{i+1}=\mathrm{Bl}_{\ol X_{i}} (C_{i})$ is $C_{i+1}$ itself, and thus $C_{i}$ and $C_i\times_{X_i} X_{i+1}=C_i\times_{\ol X_i} \ol X_{i+1}$ are proper over $K$ for each $i$.
	
	At last, we apply Proposition \ref{DB triangle} to get the following distinguished triangle
	\[
	\underline{\Omega}_{X_i/K}^j \longrightarrow R\pi_* \underline{\Omega}_{X_{i+1}/K}^j \oplus \underline{\Omega}_{C_i/K}^j \longrightarrow R\pi_*\underline{\Omega}_{C_i\times_{X_i} X_{i+1}/K}^j.
	\]
	In this way, by the properness of $C_i$ and $C_i\times_{X_i} X_{i+1}$ and the descending induction from $X_t$, we get the $d$-Hodge properness of $X=X_0$ using long exact sequences applying at their derived global sections.
	
\end{proof}
\begin{proposition}\label{simp res with compat}
	Let $X$ be in $\mathrm{Var}_K$.
	Assume $X$ admits a Cohen-Macaulay compactification $X \to \ol X$ with the complement $D$ of codimension $d+2$.
	Then there is a proper hypercovering $(X_\bullet, \ol X_\bullet)$ of the pair $(X,\ol X)$, such that
	\begin{itemize}
		\item the map $X_\bullet\to X$ is a  simplicial resolution of singularities;
		\item each $X_n\to \ol X_n$ is a compactification with its complement isomorphic to $D=\ol X\backslash X$.
	\end{itemize} 
	In particular, each $X_n$ is smooth and $d$-Hodge proper.
\end{proposition}
\begin{proof}
	We give the construction inductively, using the same idea of proof as in Theorem \ref{d Hodge prop}.
	For $n=0$, we apply the resolution of singularity of \cite[Theorem 1.1]{BM08} at $X$, and extends to its compactification $\ol X$ via blowing up at the same centers to get the pair $(X_0, \ol X_0)$.
	By the second and the third paragraph of the proof of Theorem \ref{d Hodge prop}, the map $X_0\to X$ (and similarly for $\ol X_0 \to \ol X$) is a finite composition of blowups $X_{0, i+1}=\mathrm{Bl}_{X_{0,i}}(C_{0,i}) \to X_{0, i} \to \cdots \to X_{0,0}=X$, such that the closure of the blowup center $C_{0,i}$ within the associated compactification $\ol X_{0,i}$ is itself.
	In particular, the complement $\ol X_0\backslash X_0$ is naturally isomorphic to $D:=\ol X\backslash X$, which is Cohen-Macaulay of codimension $d+2$.
	So the pair $(X_0, \ol X_0)$ satisfies the requirement.
	Moreover, as each blowup center $C_{0,i}$ is within the singular loci of $X_{0,i}$, the natural map of pairs $(X_0, \ol X_0)\to (X,\ol X)$ is an isomorphism away from $X_{\mathrm{sing}}$.

	To get the construction for $n+1$, by Definition \ref{def Ros}, it suffices to notice that the singular locus of $(\mathrm{cosk_n sk_n} X_\bullet)_{n+1}$ is within the preimage of $X_{\mathrm{sing}}$ in $(\mathrm{cosk_n sk_n} X_\bullet)_{n+1}$.
	The latter is because by induction the map of pairs $(X_m, \ol X_m) \to (X,\ol X)$ are isomorphisms away from $X_\mathrm{sing}$ for all $m\leq n$.
	In particular, the pair of truncated simplicial diagram $\left( (\mathrm{cosk_n sk_n} X_\bullet)_{n+1}, (\mathrm{cosk_n sk_n} \ol X_\bullet)_{n+1}\right)$ is isomorphic to the constant diagram when restricted to $(X\backslash X_{\mathrm{sing}}, \ol X\backslash X_{\mathrm{sing}})$.
	Thus the singular locus of $(\mathrm{cosk_n sk_n} X_\bullet)_{n+1}$ is contained in $X_{\mathrm{sing}}\times_X  (\mathrm{cosk_n sk_n} X_\bullet)_{n+1}$, whose closure in $(\mathrm{cosk_n sk_n} \ol X_\bullet)_{n+1}$ is disjoint with $D$.
	So we are done.
\end{proof}

\subsection{Local acyclicity in the singular case}
In this subsection, we extend the local acyclicity in \Cref{thm:even mainer theorem} to a finite type scheme $X$ over $\mathcal{O}_K$ that has singularity away from the special fiber.

Precisely, we prove the following.
\begin{theorem}\label{local acyc for sing}
	Let $X$ ba a finite type scheme over $\mathcal{O}_K$. Assume the following conditions hold for $X$:
	\begin{enumerate}
		\item there is an open subscheme $U$ of $X$ containing the special fiber $X_k$, such that $U$ is smooth over $\mathcal{O}_K$;
		\item there is an open immersion $X\subset \ol X$ into a proper and Cohen-Macaulay $\mathcal{O}_K$-scheme $\ol X$, such that the closure of $X\backslash U$ in $\ol X$ is $X\backslash U$ itself, and the complement $D=\ol X\backslash X$ has codimension $d+2$.
	\end{enumerate}
Then the natural map below is an isomorhism for $0\leq i\leq d$ and an injection for $i=d+1$
\[
\Upsilon_{X, \mathbb{Q}_p}\colon \mathrm{H}_\et^i(X_\mathbb{C},\mathbb{Q}_p) \longrightarrow \mathrm{H}_\et^i(\widehat{X}_C, \mathbb{Q}_p).
\]
\end{theorem}
Here we note that by assumption, Raynaud's generic fiber $\widehat{X}_K$, which is equal to $\widehat{U}_K$, is smooth over $K$.
\begin{corollary}
	Assume $X$ is as in \Cref{local acyc for sing}.
	Then for each $i\leq d$, the \'etale cohomology $ \mathrm{H}_\et^i(X_\mathbb{C},\mathbb{Q}_p)$ is a crystalline representation.
\end{corollary}

\begin{proof}
	We first notice that if $X$ is smooth over $\mathcal{O}_K$ satisfying the condition (ii), then $X$ is $d$-Hodge proper over $\mathcal{O}_K$ (\Cref{lem:example of d-Hodge-proper}), and the statement is proved in \Cref{thm:even mainer theorem}.
	In general, by assumption (i) the $\mathcal{O}_K$-non-smooth locus $X_\mathrm{sing}\subset X\backslash U$ is supported within the generic fiber $X_K$.
	Moreover, the closure of $X_\mathrm{sing}$ in $\ol X$ is $X_\mathrm{sing}$ itself.
	As a consequence, by blowing up at subschemes within the generic fibers, we can apply the construction of \Cref{simp res with compat} to get a simplicial diagram of $\mathcal{O}_K$-schemes $(X_\bullet, \ol{X}_\bullet)$ over $(X,\ol X)$, satisfying
	\begin{enumerate}
		\item the generic fiber $X_{\bullet K}\to X_K$ is a  simplicial resolution of singularities;
		\item each $X_n\to \ol X_n$ is a Cohen-Macaulay compactification;
		\item the restriction of $\ol X_\bullet \to \ol X$ on the open subset below is isomorphic to the constant augmented diagram
		\[
		\ol X_\bullet\times_X (\ol X\backslash X_\mathrm{sing}) \to \ol X\backslash X_\mathrm{sing}.
		\] 
	\end{enumerate}
Notice that by assumption $(a)$ and $(c)$, each $X_n$ is in particular smooth over $\mathcal{O}_K$, and the compactification $X_n \to \ol X_n \leftarrow D$ satisfies the assumption in \Cref{lem:example of d-Hodge-proper}.
    Here to check $\ol X_n$ is Cohen-Macaulay, it suffices to do so after by restricting at the preimage of the covering $\ol X=(\ol X\backslash X_\mathrm{sing}) \cup X_K$, where the claim then follows from the assumption (ii) of $\ol X$, property (c), and the smoothness in property (a) above.
    
    Now let us consider the following diagram extending the derived version of $\Upsilon_{X, \mathbb{Q}_p}$
    \[
    \xymatrix{
    	R\Gamma_\et (X_C, \mathbb{Q}_p)  \ar[r] \ar[d]^\sim & R\Gamma_\et (\widehat{X}_C, \mathbb{Q}_p)  \ar[d]^\sim \\
    	R\Gamma_\et (X_{\bullet C}, \mathbb{Q}_p)   \ar[r]   & R\Gamma_\et (\widehat{X}_{\bullet C}, \mathbb{Q}_p).}
    \]
    Here the two vertical arrows are isomorphisms by cohomological descent  at $X_{\bullet K} \to X_K$ and $\widehat X_{\bullet K} \to \widehat X_K$ (\Cref{Ros and coh}).
    For each $[n]\in \Delta^\op$, by  applying \Cref{thm:even mainer theorem} at the compactification $X_n \to \ol X_n \leftarrow D$, the map below is an isomorphism at degree $[0,d]$ and induces an injection on $\mathrm{H}^{d+1}$
    \[
    R\Gamma_\et (X_{n C}, \mathbb{Q}_p)  \longrightarrow R\Gamma_\et (\widehat{X}_{n C}, \mathbb{Q}_p).
    \]
    On the other hand, we can consider the spectral sequence associated to the simplicial diagram as in \cite[Tag 0DHP]{StacksProject} (similarly for $\widehat{X}_C$)
    \[
    E_1^{i,j}(X_{\bullet,C}) = \mathrm{H}^j_\et(X_{i C}, \mathbb{Q}_p) \Rightarrow \mathrm{H}^{i+j}_\et(X_C, \mathbb{Q}_p).
    \]
    Then the map of spectral sequences $E_1^{i,j}(X_{\bullet,C})\to E_1^{i,j}(\widehat{X}_{\bullet,C})$ is an equivalence for $j \leq d$ and is an injection for $j=d+1$.
    Notice that since the differential map $d_r^{i,j}$ sends $E_r^{i,j}$ to $E_r^{i+r,j-r+1}$, the total degree  goes from $(i+j)$ to $(i+r)+(j-r+1)=(i+j+1)$.
    In particular, to calculate $E_\infty^{i,j}$, by the vanishing of negative terms in the spectral sequence, it only involves the following finite amount of degrees in the spectral sequence
    \[
    \begin{cases}
    	\{ (0,i+j-1), (1,i+j-2), \ldots, (i-1,j), (i,j), (i+1, j), \ldots, (i+j+1, 0)\},~\text{if}~i>0;\\
    	\{ (0,j), (1,j), \ldots, (j+1,0)\},~\text{if}~i=0.
    \end{cases}
    \]
    As a consequence, we get 
    \[
    \begin{cases}
    	E_\infty^{i,j}(X_{\bullet,C}) = E_\infty^{i,j}(\widehat{X}_{\bullet,C}),~\text{for}~i+j\leq d+1,~i>0;\\
    	E_\infty^{0,d+1}(X_{\bullet,C}) \hookrightarrow E_\infty^{0,d+1}(\widehat{X}_{\bullet,C}),
    \end{cases}
    \]
    which implies the statement by the general fact about computing spectral sequences.
\end{proof}

%

\addcontentsline{toc}{section}{References}
\printbibliography

\bigskip

\noindent Haoyang~Guo, {\sc Max Planck Institute for Mathematics, Bonn, Germany}
\href{mailto:hguo@mpim-bonn.mpg.de}{hguo@mpim-bonn.mpg.de}\smallskip 

\noindent Dmitry~Kubrak, {\sc Institut Hautes \'Etudes Scientifiques, France}
\href{mailto:dmkubrak@gmail.com}{dmkubrak@gmail.com}

\smallskip

\noindent 
Artem~Prikhodko, {\sc Center for Advanced Studies, Skoltech, Moscow; Department of Mathematics, NRU Higher School of Economics, Moscow,}
\href{mailto:artem.n.prihodko@gmail.com}{artem.n.prihodko@gmail.com}

\end{document}